\documentclass[12pt,a4paper]{article}
\usepackage{latexsym}
\usepackage{graphicx}
\usepackage{amsmath,amsthm}
\usepackage{amssymb}
\usepackage{epsfig}
\usepackage{epstopdf}
\usepackage[titletoc,title]{appendix}

\newtheorem{thm}{Theorem}[section]
\newtheorem{cor}[thm]{Corollary}

\newtheorem{lem}[thm]{Lemma}

\theoremstyle{definition}
\newtheorem{defn}{Definition}[section]
\theoremstyle{remark}
\newtheorem{rem}{Remark}[section]

\makeatletter \@addtoreset{equation}{section}
 \makeatother
\title{Relationship Between the $2n$-points Binary and $(3n-1)$-points Quaternary Approximating Subdivision Schemes}
\title{Relationship Between the $2n$-points Binary and $(3n-1)$-points Quaternary Approximating Subdivision Schemes}

\author{
\ Rabia Hameed\thanks{Corresponding Author E-mail: rabia.hameed@gscwu.edu.pk}
\ Sidra Nosheen\thanks{E-mail: ssbajwa786@gmail.com} \\
* \dag \small{Department of Mathematics, The Government Sadiq College Women University Bahawalpur}\\
}
\date{}
\begin{document}
\maketitle
\begin{abstract}
Geometric objects are primarily represented using curves and surfaces and the subdivision schemes are the basic tools for these representations. This study is based on a new thought that  there is a special relation between the binary and some kinds of the quaternary subdivision schemes. Due to the defined relation the quaternary subdivision schemes can also be formulated by the binary subdivision schemes. This study presents the generalized formula in compact form that contains the subdivision rules of $(3n-1)$-point quaternary approximating subdivision scheme which are based on the predefined $2n$-points binary approximating subdivision scheme. Firstly, we derive a relation between the quaternary approximating subdivision scheme and the even-point binary approximating subdivision scheme. By using this relation, we next derive two types of generalized quaternary approximating subdivision scheme that are based on the even and odd values of $n$. Then we apply these generalized formulas on the known binary schemes for specific values of $n$. This gives us the corresponding quaternary approximating subdivision schemes. We also analyze some of the well-known features of binary and its corresponding quaternary approximating subdivision schemes. These results are equally applicable on parametric and non-parametric subdivision schemes.
\end{abstract}
regularity, degree of precision, mask.
\section{Introduction}\label{Introduction}
Subdivision methods for curves were introduced and mathematically analyzed for the first time by de Rham \cite{de} in 1956 and re-invent for computer graphics community by Chaikin \cite{Chaikin} in 1974. Subdivision is actually an iterative method to generate smooth curves and surfaces. Subdivision schemes increase the points at each iteration to get smooth shapes. If the subdivision process increses points two times at each iteration then this process is known as the binary subdivision process, whereas if a subdivision processe increases points four times at each iteration then this process is known as quaternary subdivision process. The tools which are used for the binary and quaternary subdivision processes are known as the binary and ternary subdivision schemes respectively.

Mathematically, the general compact forms of univariate $r$-ary subdivision scheme which is used to get a refined polygon $G^{k+1}=\{g_{i}^{k+1}\}_{i \in \mathbb{Z}} \in n(\mathbb{Z})$ from the polygon $g^{k}=\{g_{i}^{k}\}_{i \in \mathbb{Z}} \in n(\mathbb{Z})$ can be defined in terms of a mask consisting of a finite set of non-zero coefficients $\beta=\{\beta_{j}\}_{j \in \mathbb{Z}}$ as follows:
\begin{eqnarray}\label{GSS}
	g_{r\phi+\eta}^{k+1} &=& \sum\limits_{j \in \mathbb{Z}}\beta_{rj+\eta}\,\ g^{k}_{\phi+j},
\end{eqnarray}
where the set of values $\{r=2, \,\ \eta=-1,0\}$ are for the binary subdivision rules and the set of values $\{r=4, \,\ \eta=-2,-1,0,1\}$  are used for the quaternary subdivision rules respectively,\,\ $n(\mathbb{Z})$ denote the space of scaler-valued sequences. The sequence $\beta=\{\beta_{j}\}_{j \in \mathbb{Z}}$ is called the refinement mask. The polynomial which uses this mask as coefficients is called the Laurent's polynomial. Therefore, the Laurent polynomial corresponding to subdivision scheme (\ref{GSS}) is
\begin{equation*}\label{}
	\mu(c)=\sum_{j \in \mathbb{Z}}{\beta_{rj+\eta}\,\ c^{rj+\eta}}.
\end{equation*}
A convergent subdivision scheme with the corresponding mask $\beta=\{\beta_{j}\}_{j \in \mathbb{Z}}$ necessarily satisfies the following convergence condition:
\begin{equation*}\label{}
	\sum_{j \in \mathbb{Z}}{\beta_{rj+\eta}}=1.
\end{equation*}

Binary subdivision schemes were firstly introduced by Dyn et al. \cite{Dyn1}. They also analyzed the convergence of sequences of control polygons produced by a binary subdivision scheme. After that, numerous research papers have been published for construction \cite{Ghaffar, Hameed, Mustafa1, Hameed10,  Siddiqi2} and analysis \cite{Ashraf, Dyn, Donat, Hameed, Hao, Mustafa, Shahzad1} of binary subdivision schemes. Quaternary subdivision schemes first and foremost presented by Mustafa and Khan \cite{Mustafa}. They constructed and analyzed the $4$-point approximating subdivision scheme with one shape parameter by using the Laurent's polynomial of the scheme. The construction and analysis of quaternary subdivision schemes are then published by \cite{Bari1, Hashmi, Mustafa2, Nawaz, Siddiqi2, Siddiqi3, Shahzad}.
 
  In this paper, we derive a connection between binary and quaternary subdivision schemes which will be used to convert every even-point binary subdivision scheme into a quaternary subdivision scheme. The conversion between binary and quaternary schemes has not yet been performed on our best knowledge. It is our opinion that this type of conversion can be useful in the design of smooth curves for different purposes, like curve and surface alignment. In comparison to binary schemes, quaternary schemes require fewer iterations to achieve the desired level of smoothness, so the conversion from binary to quaternary is important.

More precisely, the paper is organised as follows: In Section 2, we give the generalized procedure to define the relation between the binary and the quaternary subdivision schemes. In Section 3, we give the applications of the given technique along with the graphical comparisons of the pair of the binary and quaternary subdivision schemes. Section 4 is about the H\"{o}lder's regularity computation of the binary and the quaternary subdivision schemes. In Section 5, we give the response of these pairs of subdivision scheme on the polynomial data. Conclusion about the technique is given in Section 6.
\section{Link between the binary and quaternary subdivision schemes}
This section explains a new observation about the relation between the even-point binary approximating subdivision schemes and the quaternary approximating subdivision schemes. The generalized subdivision rules of $(3n-1)$-points quaternary subdivision schemes is deduced by using the subdivision rules of the $2n$-point binary subdivision schemes. The $2n$-point dual binary subdivision scheme which maps the polygon $g$ to a refined polygon $G$ after one level of refinement can be written as:
\begin{equation}\label{const1}
\left\{\begin{array}{c}
{g}_{2\varphi-1}^{k+1}=\sum\limits_{\lambda=-n+1}^{n}\beta_{2\lambda}\,\ g_{\varphi+\lambda-1}^{k},\\ \\
{g}_{2\varphi}^{k+1}=\sum\limits_{\lambda=-n+1}^{n}\beta_{2-2\lambda} \,\ g_{\varphi+\lambda}^{k},
\end{array}\right.
\end{equation}
where $\{g_{\varphi+\lambda}^{k}:\lambda=-n+1,\ldots,n: \varphi \in \mathbb{R}\}$ are the control points at $k$-th subdivision level and $\{\beta_{\lambda}:\lambda=-2n+2,\ldots,2n\}$ is the mask of the subdivision scheme.

The following lemma gives a new form of the subdivision rules of $2n$-point binary subdivision scheme which is defined in (\ref{const1}).
\begin{lem}\label{lem1}
	If we change $\varphi$ by the odd numbers $2\varphi-1$ and $2\varphi+1$ and the even number $2\varphi$ in the subdivision equations of the binary subdivision scheme (\ref{const1}), then these subdivision equations reduced into the four subdivision equations.
\end{lem}
\begin{proof}
Firstly, we re-write the subdivision scheme (\ref{const1}) in the form which is free from subdivision levels, hence we get two
subdivision equations given below:
\begin{equation}\label{aa2}
\left\{\begin{array}{c}
{g}_{2\varphi-1}=\sum\limits_{\lambda=-n+1}^{n}\beta_{2\lambda}\,\ g_{\varphi+\lambda-1},\\ \\
{g}_{2\varphi}=\sum\limits_{\lambda=-n+1}^{n}\beta_{2-2\lambda} \,\ g_{\varphi+\lambda}.
\end{array}\right.
\end{equation}
Since $\varphi \in \mathbb{R}$, hence we replace  $\varphi$  by $2\varphi-1 $ in the second subdivision equation of (\ref{aa2}). This gives
\begin{equation}\label{new-1}
 g_{4\varphi-2}=\sum\limits_{\lambda=-n+1}^{n}\beta_{2-2\lambda} \,\ g_{2\varphi+\lambda-1}.
\end{equation}
We again replace  $\varphi$ by $2\varphi$ in the first and second subdivision equations of (\ref{aa2}) respectively. Hence we get
\begin{equation}\label{new-2}
g_{4\varphi-1}=\sum\limits_{\lambda=-n+1}^{n}\beta_{2\lambda} \,\ g_{2\varphi+\lambda-1}
\end{equation}
and
\begin{equation}\label{new-3}
 g_{4\varphi}=\sum\limits_{\lambda=-n+1}^{n}\beta_{2-2\lambda} \,\ g_{2\varphi+\lambda}.
\end{equation}
Now we replace $\varphi$ by $2\varphi+1$ in the first subdivision eqaution of (\ref{aa2}). Which gives
\begin{equation}\label{new-4}
g_{4\varphi+1}=\sum\limits_{\lambda=-n+1}^{n}\beta_{2\lambda} \,\ g_{2\varphi+\lambda}.
\end{equation}
Combining (\ref{new-1}), (\ref{new-2}), (\ref{new-3}) and (\ref{new-4}), we get the four following equations
\begin{equation}\label{const6}
              \left\{\begin{array}{c}
 g_{4\varphi-2}=\sum\limits_{\lambda=-n+1}^{n}\beta_{2-2\lambda} \,\ g_{2\varphi+\lambda-1}, \,\ \,\ \,\
 g_{4\varphi-1}=\sum\limits_{\lambda=-n+1}^{n}\beta_{2\lambda} \,\ g_{2\varphi+\lambda-1},\\\\
 g_{4\varphi}=\sum\limits_{\lambda=-n+1}^{n}\beta_{2-2\lambda} \,\ g_{2\varphi+\lambda},\,\ \,\ \,\
 g_{4\varphi+1}=\sum\limits_{\lambda=-n+1}^{n}\beta_{2\lambda} \,\ g_{2\varphi+\lambda}.
                 \end{array}\right.
\end{equation}
Which completes the proof.
\end{proof}
Now we split the further process into two parts depending on the even and odd values of $n$. The first theorem is proved for the even values of $n$, while the second one is proved for the odd $n$.
\begin{thm}\label{thm-even-1}
	If $n$ is even, that is $n=2m: m\in \mathbb{N}$, then the subdivision rules $g_{4\varphi-2}$ and $g_{4\varphi-1}$ in (\ref{const6}) are the linear combination of $6m-1$ control points $g_{\varphi-3m+1}\ldots g_{\varphi+3m-1}$, while the subdivision rules $g_{4\varphi}$ and $g_{4\varphi+1}$ in (\ref{const6}) are the linear combination of $6m$ control points $g_{\varphi-3m+1}\ldots g_{\varphi+3m}$. 
\end{thm}
\begin{proof}
Since $n$ is even, so firstly we put $n=2m$ in (\ref{const6}). Thus we get
\begin{equation}\label{const7}
\left\{\begin{array}{c}
g_{4\varphi-2}=\sum\limits_{\lambda=-2m+1}^{2m}\beta_{2-2\lambda} \,\ g_{2\varphi+\lambda-1},\,\ \,\ \,\
g_{4\varphi-1}=\sum\limits_{\lambda=-2m+1}^{2m}\beta_{2\lambda} \,\ g_{2\varphi+\lambda-1},\\\\
g_{4\varphi}=\sum\limits_{\lambda=-2m+1}^{2m}\beta_{2-2\lambda} \,\ g_{2\varphi+\lambda},\,\ \,\ \,\
g_{4\varphi+1}=\sum\limits_{\lambda=-2m+1}^{2m}\beta_{2\lambda} \,\ g_{2\varphi+\lambda}.
\end{array}\right.
\end{equation}
Now we have to find out values of $g_{2\varphi-2m}$, $g_{2\varphi-2m+1}$, $\ldots$, $g_{2\varphi+2m-1}$, $g_{2\varphi+2m}$. For this, first we evaluate (\ref{aa2}) for $n=2m$ and then by increasing or decreasing the subscript, we get the required unknowns.
\begin{eqnarray}\label{const9}
\left\{\begin{array}{cccc}
		g_{2\varphi-1}&=& \beta_{2-4m} \,\ g_{\varphi-2m}+\beta_{4-4m} \,\ g_{\varphi-2m+1}+\beta_{6-4m} \,\ g_{\varphi-2m+2}+\beta_{8-4m} \,\ g_{\varphi-2m+3}\\&& +\beta_{10-4m} \,\ g_{\varphi-2m+4}+\beta_{12-4m} \,\ g_{\varphi-2m+5}+\ldots+\beta_{4m-8} \,\ g_{\varphi+2m-5}+\\&& \beta_{4m-6} \,\ g_{\varphi+2m-4}+\beta_{4m-4} \,\ g_{\varphi+2m-3} +\beta_{4m-2} \,\ g_{\varphi+2m-2}+\beta_{4m} \,\ g_{\varphi+2m-1},\\
		g_{2\varphi}&=&\beta_{4m} \,\ g_{\varphi-2m+1}+\beta_{4m-2} \,\ g_{\varphi-2m+2}+\beta_{4m-4} \,\ g_{\varphi-2m+3}+\beta_{4m-6} \,\ g_{\varphi-2m+4}\\&& +\beta_{4m-8} \,\ g_{\varphi-2m+5}+\beta_{4m-10} \,\ g_{\varphi-2m+6}+\ldots+\beta_{10-4m} \,\ g_{\varphi+2m-4}+ \\&&\beta_{8-4m} \,\ g_{\varphi+2m-3}+\beta_{6-4m} \,\ g_{\varphi+2m-2}+\beta_{4-4m} \,\ g_{\varphi+2m-1}+\beta_{2-4m} \,\ g_{\varphi+2m}.
	\end{array}\right.
\end{eqnarray}
Now we replace $\varphi$ by $\varphi-m$ in second rule of (\ref{const9}), we get
\begin{eqnarray*}\label{const11}
g_{2\varphi-2m}&=&\beta_{4m} \,\ g_{\varphi-3m+1}+\beta_{4m-2} \,\ g_{\varphi-3m+2}+\beta_{4m-4} \,\ g_{\varphi-3m+3}+\beta_{4m-6} \,\ g_{\varphi-3m+4}\\&& +\beta_{4m-8} \,\ g_{\varphi-3m+5}+\beta_{4m-10} \,\ g_{\varphi-3m+6}+\ldots+\beta_{10-4m} \,\ g_{\varphi+m-4}+\\&& \beta_{8-4m} \,\ g_{\varphi+m-3}+\beta_{6-4m} \,\ g_{\varphi+m-2} +\beta_{4-4m} \,\ g_{\varphi+m-1}+\beta_{2-4m} \,\ g_{\varphi+m}.
\end{eqnarray*}
Now we replace $\varphi$ by $\varphi-m+1$ in the first and second rules of (\ref{const9}), we get
\begin{eqnarray*}\label{const11}
g_{2\varphi-2m+1}&=&\beta_{2-4m} \,\ g_{\varphi-3m+1}+\beta_{4-4m} \,\ g_{\varphi-3m+2}+\beta_{6-4m} \,\ g_{\varphi-3m+3}+\beta_{8-4m} \,\ g_{\varphi-3m+4}\\&&+\beta_{10-4m} \,\ g_{\varphi-3m+5}+\beta_{12-4m} \,\ g_{\varphi-3m+6}+\ldots+\beta_{4m-8} \,\ g_{\varphi+m-4}+\beta_{4m-6}\\&& \times g_{\varphi+m-3}+\beta_{4m-4} \,\ g_{\varphi+m-2}+\beta_{4m-2} \,\ g_{\varphi+m-1}+\beta_{4m} \,\ g_{\varphi+m},\\
g_{2\varphi-2m+2}&=& \beta_{4m} \,\ g_{\varphi-3m+2}+\beta_{4m-2} \,\ g_{\varphi-3m+3}+\beta_{4m-4} \,\ g_{\varphi-3m+4}+\beta_{4m-6}\,\ g_{\varphi-3m+5}\\&&+\beta_{4m-8} \,\ g_{\varphi-3m+6}+\beta_{4m-10} \,\ g_{\varphi-3m+7}+\ldots+\beta_{10-4m} \,\ g_{\varphi+m-3}+ \beta_{8-4m}\\&&\times g_{\varphi+m-2}+\beta_{6-4m} \,\ g_{\varphi+m-1} +\beta_{4-4m} \,\ g_{\varphi+m}+\beta_{2-4m} \,\ g_{\varphi+m+1}.
\end{eqnarray*}
Now we replace $\varphi$ by $\varphi-m+2$ in the first and second rules of  (\ref{const9}), we get
\begin{eqnarray*}\label{const13}
g_{2\varphi-2m+3}&=& \beta_{2-4m} \,\ g_{\varphi-3m+2}+\beta_{4-4m} \,\ g_{\varphi-3m+3}+\beta_{6-4m} \,\ g_{\varphi-3m+4}+\beta_{8-4m}\\&& \times g_{\varphi-3m+5}+\beta_{10-4m} \,\ g_{\varphi-3m+6}+\beta_{12-4m} \,\ g_{\varphi-3m+7}+\ldots+\beta_{4m-8} \\&& \times g_{\varphi+m-3}+ \beta_{4m-6} \,\ g_{\varphi+m-2}+\beta_{4m-4} \,\ g_{\varphi+m-1} +\beta_{4m-2} \,\ g_{\varphi+m}+\\&& \beta_{4m} \,\ g_{\varphi+m+1},\\
	g_{2\varphi-2m+4}&=& \beta_{4m} \,\ g_{\varphi-3m+3}+\beta_{4m-2} \,\ g_{\varphi-3m+4}+\beta_{4m-4} \,\ g_{\varphi-3m+5}+\beta_{4m-6} \,\ g_{\varphi-3m+6}\\&& +\beta_{4m-8} \,\ g_{\varphi-3m+7}+\beta_{4m-10} \,\ g_{\varphi-3m+8}+\ldots +\beta_{10-4m} \,\ g_{\varphi+m-2}+\beta_{8-4m} \\&& \times g_{\varphi+m-1}+\beta_{6-4m} \,\ g_{\varphi+m} +\beta_{4-4m} \,\ g_{\varphi+m+1}+\beta_{2-4m} \,\ g_{\varphi+m+2}.
\end{eqnarray*}
Now we replace $\varphi$ by $\varphi-m+3$ in first and second rules of (\ref{const9}), we get
\begin{eqnarray*}\label{const15}
g_{2\varphi-2m+5}&=& \beta_{2-4m} \,\ g_{\varphi-3m+3}+\beta_{4-4m} \,\ g_{\varphi-3m+4}+\beta_{6-4m} \,\ g_{\varphi-3m+5}+\beta_{8-4m} \,\ g_{\varphi-3m+6}\\&& + \beta_{10-4m} \,\ g_{\varphi-3m+7}+\beta_{12-4m} \,\ g_{\varphi-3m+8}+\ldots +\beta_{4m-8} \,\ g_{\varphi+m-2}+\beta_{4m-6} \,\ \\&& \times g_{\varphi+m-1}+\beta_{4m-4} \,\ g_{\varphi+m}+\beta_{4m-2} \,\ g_{\varphi+m+1}+\beta_{4m} \,\ g_{\varphi+m+2},\\
g_{2\varphi-2m+6}&=& \beta_{4m} \,\ g_{\varphi-3m+4}+\beta_{4m-2} \,\ g_{\varphi-3m+5}+\beta_{4m-4} \,\ g_{\varphi-3m+6}+\beta_{4m-6} \,\ g_{\varphi-3m+7}+\\&& \beta_{4m-8} \,\ g_{\varphi-3m+8}+\beta_{4m-10} \,\ g_{\varphi-3m+9}+\ldots  +\beta_{10-4m} \,\ g_{\varphi+m-1}+\beta_{8-4m}\times \\&& g_{\varphi+m}+\beta_{6-4m} \,\ g_{\varphi+m+1}+\beta_{4-4m} \,\ g_{\varphi+m+2}+\beta_{2-4m} \,\ g_{\varphi+m+3}.
\end{eqnarray*}
\,\ \,\ \,\ \,\ \,\ \,\ \,\ \,\ \,\ \,\ \,\ \,\ \,\ \,\ \,\ \,\ \,\ \,\ \,\ \,\ \,\ \vdots \,\ \,\ \,\ \,\ \,\ \,\ \,\ \,\ \,\ \,\ \,\ \,\ \,\ \,\ \,\ \,\ \,\ \,\ \,\ \,\ \,\ \vdots \,\ \,\ \,\ \,\ \,\ \,\ \,\ \,\ \,\ \,\ \,\ \,\ \,\ \,\ \,\ \,\ \,\ \,\ \,\ \,\ \,\ \vdots \\
Continuing this process, we replace $\varphi$ by $\varphi+m-2$ in first and second rules of (\ref{const9}), we get
\begin{eqnarray*}\label{const19}
g_{2\varphi+2m-5}&=&\beta_{2-4m} \,\ g_{\varphi-m-2}+\beta_{4-4m} \,\ g_{\varphi-m-1}+\beta_{6-4m} \,\ g_{\varphi-m}+\beta_{8-4m} \,\ g_{\varphi-m+1}+\\&& \beta_{10-4m} \,\ g_{\varphi-m+2}+\beta_{12-4m} \,\ g_{\varphi-m+3}+\ldots+\beta_{4m-8} \,\ g_{\varphi+3m-7}+\beta_{4m-6} \\&& \times g_{\varphi+3m-6}+\beta_{4m-4} \,\ g_{\varphi+3m-5}+\beta_{4m-2} \,\ g_{\varphi+3m-4}+\beta_{4m} \,\ g_{\varphi+3m-3},\\
	g_{2\varphi+2m-4}&=&\beta_{4m} \,\ g_{\varphi-m-1}+\beta_{4m-2} \,\ g_{\varphi-m}+\beta_{4m-4} \,\ g_{\varphi-m+1}+\beta_{4m-6} \,\ g_{\varphi-m+2}+\\&& \beta_{4m-8} \,\ g_{\varphi-m+3}+\beta_{4m-10} \,\ g_{\varphi-m+4}+ \ldots +\beta_{10-4m} \,\ g_{\varphi+3m-6}+\beta_{8-4m} \\&& \times g_{\varphi+3m-5}+\beta_{6-4m} \,\ g_{\varphi+3m-4}+\beta_{4-4m} \,\ g_{\varphi+3m-3}+\beta_{2-4m} \,\ g_{\varphi+3m-2}.
\end{eqnarray*}
Now we replace $\varphi$ by $\varphi+m-1$ in the first and second rules of (\ref{const9}), we get
\begin{eqnarray*}\label{const21}
g_{2\varphi+2m-3}&=& \beta_{2-4m} \,\ g_{\varphi-m-1}+\beta_{4-4m} \,\ g_{\varphi-m}+\beta_{6-4m} \,\ g_{\varphi-m+1} +\beta_{8-4m} \,\ g_{\varphi-m+2}+\\&& \beta_{10-4m} \,\ g_{\varphi-m+3}+\beta_{12-4m} \,\ g_{\varphi-m+4}+ \ldots +\beta_{4m-8} \,\ g_{\varphi+3m-6}+\beta_{4m-6}\\&& \times g_{\varphi+3m-5}+\beta_{4m-4} \,\ g_{\varphi+3m-4} +\beta_{4m-2} \,\ g_{\varphi+3m-3}+\beta_{4m} \,\ g_{\varphi+3m-2},\\
	g_{2\varphi+2m-2}&=& \beta_{4m} \,\ g_{\varphi-m}+\beta_{4m-2} \,\ g_{\varphi-m+1}+\beta_{4m-4} \,\ g_{\varphi-m+2}+\beta_{4m-6} \,\ g_{\varphi-m+3}+\\&& \beta_{4m-8} \,\ g_{\varphi-m+4}+\beta_{4m-10} \,\ g_{\varphi-m+5}+\ldots +\beta_{12-4m} \,\ g_{\varphi+3m-6}+\beta_{10-4m}\\&& \times g_{\varphi+3m-5}+\beta_{8-4m} g_{\varphi+3m-4}+\beta_{6-4m} \,\ g_{\varphi+3m-3}+\beta_{4-4m} \,\ g_{\varphi+3m-2}+\\&&\beta_{2-4m} \,\ g_{\varphi+3m-1}.
\end{eqnarray*}
Now we replace  $\varphi$ by $\varphi+m$ in first and second rules of (\ref{const9}), we get
\begin{eqnarray*}\label{const23}
g_{2\varphi+2m-1}&=&\beta_{2-4m} \,\ g_{\varphi-m}+\beta_{4-4m} \,\ g_{\varphi-m+1}+\beta_{6-4m} \,\ g_{\varphi-m+2} +\beta_{8-4m} \,\ g_{\varphi-m+3}+\\&& \beta_{10-4m} \,\ g_{\varphi-m+4}+\beta_{12-4m} \,\ g_{\varphi-m+5}+\ldots +\beta_{4m-10} \,\ g_{\varphi+3m-6}+\beta_{4m-8}\\&& \times g_{\varphi+3m-5}+\beta_{4m-6} \,\ g_{\varphi+3m-4}+\beta_{4m-4} \,\ g_{\varphi+3m-3}+\beta_{4m-2} \,\ g_{\varphi+3m-2}+\\&&\beta_{4m} \,\ g_{\varphi+3m-1},\\
	g_{2\varphi+2m}&=& \beta_{4m} \,\ g_{\varphi-m+1}+\beta_{4m-2} \,\ g_{\varphi-m+2}+\beta_{4m-4} \,\ g_{\varphi-m+3}+\beta_{4m-6} \,\ g_{\varphi-m+4}+\\&& \beta_{4m-8} \,\ g_{\varphi-m+5}+\beta_{4m-10} \,\ g_{\varphi-m+6}+\ldots +\beta_{14-4m} \,\ g_{\varphi+3m-6}+\beta_{12-4m} \\&&\times g_{\varphi+3m-5}+\beta_{10-4m} \,\ g_{\varphi+3m-4}+\beta_{8-4m}g_{\varphi+3m-3}+\beta_{6-4m} \,\ g_{\varphi+3m-2}+\\&&\beta_{4-4m} \,\ g_{\varphi+3m-1}+\beta_{2-4m} \,\ g_{\varphi+3m}.
\end{eqnarray*}
We get all the unknowns $\,\ g_{2\varphi-2m},\,\  g_{2\varphi-2m+1}, \,\ g_{2\varphi-2m+2},\ldots, g_{2\varphi+2m-2},\,\ g_{2\varphi+2m-1},\,\ g_{2\varphi+2m} $. Now we substitute all these values in the four equations given in (\ref{const7}):
\begin{eqnarray*}\label{const30}
g_{4\varphi-2}&=&(\beta_{4m}^{2}+\beta_{4m-2}\beta_{2-4m})g_{\varphi-3m+1}
+(\beta_{4m}\beta_{4m-2}+\beta_{4m-2}\beta_{4-4m}+\beta_{4m-4}\beta_{4m}+\beta_{4m-6}\\&& \times \beta_{2-4m}) g_{\varphi-3m+2}+(\beta_{4m}\beta_{4m-4}+\beta_{4m-2} \beta_{6-4m}+\beta_{4m-4}\beta_{4m-2}+\beta_{4m-6}\beta_{4-4m}+\\&&\beta_{4m-8} \beta_{4m}+\beta_{4m-10}\beta_{2-4m})g_{\varphi-3m+3}+(\beta_{4m} \beta_{4m-6}+\beta_{4m-2}\beta_{8-4m}+\beta_{4m-4}^{2}+\beta_{4m-6}\\&&\times \beta_{6-4m}+\beta_{4m-8} \beta_{4m-2}+\beta_{4m-10}\beta_{4-4m}+\beta_{4m-12}\beta_{4m}+\beta_{4m-14}\beta_{2-4m}) g_{\varphi-3m+4}+\\&&(\beta_{4m}\beta_{4m-8}+\beta_{4m-2}\beta_{10-4m}+\beta_{4m-4}\beta_{4m-6} +\beta_{4m-6}\beta_{8-4m}+\beta_{4m-8}\beta_{4m-4}+\beta_{4m-10}\\&&\times \beta_{6-4m}+\beta_{4m-12}  \beta_{4m-2}+\beta_{4m-14}\beta_{4-4m}+\beta_{4m-16}\beta_{4m}+\beta_{4m-18}\beta_{2-4m}) g_{\varphi-3m+5}+\\&&(\beta_{4m}\beta_{4m-10}+\beta_{4m-2}\beta_{12-4m}+\beta_{4m-4}\beta_{4m-8}+\beta_{4m-6}
\beta_{10-4m}+\beta_{4m-8}\beta_{4m-6}+\beta_{4m-10}\\&&\times \beta_{8-4m}+\beta_{4m-12} \beta_{4m-4}+\beta_{4m-14}\beta_{6-4m}+\beta_{4m-16}\beta_{4m-2}+\beta_{4m-18} \beta_{4-4m}+\beta_{4m-20}\times\\&&\beta_{4m}+\beta_{4m-22}\beta_{2-4m})g_{\varphi-3m+6}
+\ldots+(\beta_{24-4m}\beta_{2-4m}+\beta_{22-4m}\beta_{4m}+\beta_{20-4m}\beta_{4-4m}+\\&&\beta_{18-4m}\beta_{4m-2}+
\beta_{16-4m} \beta_{6-4m}+\beta_{14-4m}\beta_{4m-4}+\beta_{12-4m}\beta_{8-4m}+\beta_{10-4m} \beta_{4m-6}+\beta_{8-4m}\\&&\times\beta_{10-4m}+\beta_{6-4m}\beta_{4m-8}+\beta_{4-4m} \beta_{12-4m}+\beta_{2-4m}\beta_{4m-10})g_{\varphi+3m-6}+(\beta_{20-4m}\beta_{2-4m}+\\&&\beta_{18-4m}  \beta_{4m}+\beta_{16-4m}\beta_{4-4m}+\beta_{14-4m}\beta_{4m-2}+\beta_{12-4m}\beta_{6-4m}+\beta_{10-4m} \beta_{4m-4}+\beta_{8-4m}^{2}\\&&+\beta_{6-4m}\beta_{4m-6}+\beta_{4-4m}\beta_{10-4m}+\beta_{2-4m} \beta_{4m-8})g_{\varphi+3m-5}+(\beta_{16-4m} \beta_{2-4m}+\beta_{14-4m}\beta_{4m}\\&&+\beta_{12-4m} \beta_{4-4m}+\beta_{10-4m}\beta_{4m-2}+\beta_{8-4m}\beta_{6-4m}+\beta_{6-4m}\beta_{4m-4}+\beta_{4-4m} \beta_{8-4m}+\beta_{2-4m}\\&&\times\beta_{4m-6})g_{\varphi+3m-4}
+(\beta_{12-4m}\beta_{2-4m}+\beta_{10-4m}\beta_{4m}+\beta_{8-4m}\beta_{4-4m}+\beta_{6-4m} \beta_{4m-2}+\beta_{4-4m}\\&&\times\beta_{6-4m}+\beta_{2-4m}\beta_{4m-4}) g_{\varphi+3m-3}+(\beta_{8-4m}   \beta_{2-4m}+\beta_{6-4m}\beta_{4m}+\beta_{4-4m}^{2}+\beta_{2-4m} \beta_{4m-2})\\&&\times g_{\varphi+3m-2}+(\beta_{4-4m}\beta_{2-4m}+\beta_{2-4m}\beta_{4m})g_{\varphi+3m-1},
\end{eqnarray*}
\begin{eqnarray*}
g_{4\varphi-1}&=& (\beta_{2-4m}\beta_{4m}+\beta_{4-4m}\beta_{2-4m})g_{\varphi-3m+1}
+(\beta_{2-4m}\beta_{4m-2}+\beta_{4-4m}^{2}+\beta_{6-4m}\beta_{4m}\\&&+\beta_{8-4m}\beta_{2-4m})g_{\varphi-3m+2}+
(\beta_{2-4m}\beta_{4m-4}+\beta_{4-4m}\beta_{6-4m}+\beta_{6-4m}\beta_{4m-2}+\\&&\beta_{8-4m} \beta_{4-4m}+\beta_{10-4m} \beta_{4m}+\beta_{12-4m}\beta_{2-4m})g_{\varphi-3m+3}+(\beta_{2-4m} \beta_{4m-6}+\beta_{4-4m}\\&&\times\beta_{8-4m}+\beta_{6-4m}\beta_{4m-4}+\beta_{8-4m} \beta_{6-4m}+\beta_{10-4m}\beta_{4m-2}+\beta_{12-4m}\beta_{4-4m}+\\&&\beta_{14-4m}\beta_{4m}+\beta_{16-4m} \beta_{2-4m})g_{\varphi-3m+4}+(\beta_{2-4m}\beta_{4m-8}+\beta_{4-4m} \beta_{10-4m}+\beta_{6-4m}\\&& \times \beta_{4m-6}+\beta_{8-4m}^{2}+\beta_{10-4m}\beta_{4m-4}+\beta_{12-4m} \beta_{6-4m}+\beta_{14-4m} \beta_{4m-2}+\\&&\beta_{16-4m}\beta_{4-4m}+\beta_{18-4m}\beta_{4m}+\beta_{20-4m}\beta_{2-4m}) g_{\varphi-3m+5}+(\beta_{2-4m}\beta_{4m-10}+\beta_{4-4m}\\&&\times\beta_{12-4m}+\beta_{6-4m}\beta_{4m-8}+\beta_{8-4m}
\beta_{10-4m}+\beta_{10-4m}\beta_{4m-6}+\beta_{12-4m}\beta_{8-4m}+\\&&\beta_{14-4m} \beta_{4m-4}+\beta_{16-4m} \beta_{6-4m}+\beta_{18-4m}\beta_{4m-2}+\beta_{20-4m} \beta_{4-4m}+\beta_{22-4m}\beta_{4m}+\\&&\beta_{24-4m}\beta_{2-4m})g_{\varphi-3m+6}+\ldots +(\beta_{4m-22} \beta_{2-4m}+\beta_{4m-20} \beta_{4m}+\beta_{4m-18} \beta_{4-4m}+\\&&\beta_{4m-16}\beta_{4m-2}+\beta_{4m-14}\beta_{6-4m} +\beta_{4m-12}\beta_{4m-4}+\beta_{4m-10}\beta_{8-4m}+\beta_{4m-8}\beta_{4m-6}+\\&&\beta_{4m-6} \beta_{10-4m}+\beta_{4m-4}\beta_{4m-8}+\beta_{4m-2}\beta_{12-4m}+\beta_{4m}\beta_{4m-10})g_{\varphi+3m-6}+
(\beta_{4m-18}\\&&\times\beta_{2-4m}+\beta_{4m-16}\beta_{4m}+\beta_{4m-14}\beta_{4-4m} +\beta_{4m-12} \beta_{4m-2}+\beta_{4m-10}\beta_{6-4m}+\beta_{4m-8}\\&&\times\beta_{4m-4}+\beta_{4m-6}\beta_{8-4m}+\beta_{4m-4}\beta_{4m-6} +\beta_{4m-2}\beta_{10-4m}+\beta_{4m}\beta_{4m-8})g_{\varphi+3m-5}+\\&&(\beta_{4m-14}\beta_{2-4m} +\beta_{4m-12}\beta_{4m}+\beta_{4m-10}\beta_{4-4m}+\beta_{4m-8}\beta_{4m-2}+\beta_{4m-6}\beta_{6-4m}+
\\&&\beta_{4m-4}^{2} +\beta_{4m-2}\beta_{8-4m}+\beta_{4m}\beta_{4m-6})g_{\varphi+3m-4}+
(\beta_{4m-10}\beta_{2-4m}+\beta_{4m-8}\beta_{4m}+\\&&\beta_{4m-6}\beta_{4-4m}+\beta_{4m-4}\beta_{4m-2} +\beta_{4m-2}\beta_{6-4m}+\beta_{4m}\beta_{4m-4}) g_{\varphi+3m-3}+(\beta_{4m-6}\times\\&&\beta_{2-4m}+\beta_{4m-4}\beta_{4m}+\beta_{4m-2} \beta_{4-4m}+\beta_{4m}\beta_{4m-2})g_{\varphi+3m-2}+(\beta_{4m-2}\beta_{2-4m}+\\&&\beta_{4m}^{2})g_{\varphi+3m-1},
\end{eqnarray*}
\begin{eqnarray*}
g_{4\varphi}&=&(\beta_{4m}\beta_{2-4m})g_{\varphi-3m+1}+(\beta_{4m}\beta_{4-4m}+\beta_{4m-2}\beta_{4m}+\beta_{4m -4}\beta_{2-4m})g_{\varphi-3m+2}+(\beta_{4m}\beta_{6-4m}\\&&+\beta_{4m-2}^{2}+\beta_{4m-4} \beta_{4-4m}+\beta_{4m-6}\beta_{4m}+\beta_{4m-8}\beta_{2-4m})g_{\varphi-3m+3}+(\beta_{4m}\beta_{8-4m}+
\beta_{4m-2}\times\\&&\beta_{4m-4}+\beta_{4m-4}\beta_{6-4m}+\beta_{4m-6}\beta_{4m-2}+\beta_{4m-8} \beta_{4-4m}+\beta_{4m-10}\beta_{4m}+\beta_{4m-12}\beta_{2-4m})\times\\&&g_{\varphi-3m+4}+(\beta_{4m}\beta_{10-4m}+
\beta_{4m-2}\beta_{4m-6}+\beta_{4m-4}\beta_{8-4m}+\beta_{4m-6}\beta_{4m-4}+\beta_{4m-8} \beta_{6-4m}\\&&+\beta_{4m-10}\beta_{4m-2}+\beta_{4m-12}\beta_{4-4m}+\beta_{4m-14}\beta_{4m}+\beta_{4m-16}
\beta_{2-4m})g_{\varphi-3m+5}+(\beta_{4m}\beta_{12-4m}\\&&+\beta_{4m-2}\beta_{4m-8}+\beta_{4m-4}\beta_{10-4m}+
\beta_{4m-6}^{2}+\beta_{4m-8}\beta_{8-4m}+\beta_{4m-10}\beta_{4m-4}+\beta_{4m-12}\beta_{6-4m}
\\&&+\beta_{4m-14}\beta_{4m-2}+\beta_{4m-16}\beta_{4-4m}+\beta_{4m-18}\beta_{4m}+\beta_{4m-20}\beta_{2-4m}) g_{\varphi-3m+6}+\ldots+(\beta_{24-4m}\\&&\times\beta_{4m}+\beta_{22-4m}\beta_{4-4m}+\beta_{20-4m}\beta_{4m-2}+
\beta_{18-4m}\beta_{6-4m} +\beta_{16-4m}\beta_{4-4m} +\beta_{14-4m}\beta_{8-4m}\\&& +\beta_{12-4m}\beta_{4m-6} +\beta_{10-4m}^{2} +\beta_{8-4m}\beta_{4m-8} +\beta_{6-4m}\beta_{12-4m}+\beta_{4-4m}\beta_{4m-10} +\beta_{2-4m}\beta_{14-4m})\\&&\times g_{\varphi+3m-6}+(\beta_{22-4m}\beta_{2-4m}+\beta_{20-4m}\beta_{4m}+
\beta_{18-4m}\beta_{4-4m} +\beta_{16-4m}\beta_{4m-2} +\beta_{14-4m}\beta_{6-4m} \\&&+\beta_{12-4m}\beta_{4m-4} +\beta_{10-4m}\beta_{8-4m}+\beta_{8-4m}\beta_{4m-6} +\beta_{6-4m}\beta_{10-4m}+\beta_{4-4m}\beta_{4m-8} +\beta_{2-4m}\times\\&&\beta_{12-4m})g_{\varphi+3m-5}+(\beta_{18-4m}
\beta_{2-4m} +\beta_{16-4m}\beta_{4m} +\beta_{14-4m}\beta_{4-4m} +\beta_{12-4m}\beta_{4m-2} +\beta_{10-4m}\\&&\times\beta_{6-4m}+\beta_{8-4m}\beta_{4m-4} +\beta_{6-4m}\beta_{8-4m}+\beta_{4-4m}\beta_{4m-6} +\beta_{2-4m}\beta_{10-4m})g_{\varphi+3m-4}+(\beta_{14-4m}\\&&\times\beta_{2-4m} +\beta_{12-4m}\beta_{4m} +\beta_{10-4m}\beta_{4-4m}+\beta_{8-4m}\beta_{4m-2} +\beta_{6-4m}^{2}+\beta_{4-4m}\beta_{4m-4} +\beta_{2-4m}\times\\&&\beta_{8-4m})g_{\varphi+3m-3}+(\beta_{10-4m} \beta_{2-4m}+\beta_{8-4m}\beta_{4m}+\beta_{6-4m} \beta_{4-4m}+\beta_{4-4m}\beta_{4m-2}+\beta_{2-4m}\times\\&&\beta_{6-4m})g_{\varphi+3m-2}+(\beta_{6-4m}\beta_{2-4m}+ \beta_{4-4m} \beta_{4m}+\beta_{2-4m}\beta_{4-4m})g_{\varphi+3m-1}+(\beta_{2-4m}^{2})g_{\varphi+3m}
\end{eqnarray*}
and
\begin{eqnarray*}
g_{4\varphi+1}&=&(\beta_{2-4m}^{2})g_{\varphi-3m+1}+(\beta_{2-4m}\beta_{4-4m}+\beta_{4-4m} \beta_{4m}+\beta_{6-4m}\beta_{2-4m})g_{\varphi-3m+2}+(\beta_{2-4m}\\&&\times \beta_{6-4m}+\beta_{4-4m} \beta_{4m-2}+\beta_{6-4m}\beta_{4-4m}+\beta_{8-4m}\beta_{4m}+\beta_{10-4m}\beta_{2-4m}) g_{\varphi-3m+3}+\\&&(\beta_{2-4m}\beta_{8-4m}+\beta_{4-4m}\beta_{4m-4}+\beta_{6-4m}^{2}+\beta_{8-4m} \beta_{4m-2}+\beta_{10-4m}\beta_{4-4m}+\beta_{12-4m}\\&&\times\beta_{4m}+\beta_{14-4m}\beta_{2-4m}) g_{\varphi-3m+4}+(\beta_{2-4m}\beta_{10-4m}+\beta_{4-4m}\beta_{4m-6}+\beta_{6-4m}\beta_{8-4m}\\&&+\beta_{8-4m} \beta_{4m-4}+\beta_{10-4m}\beta_{6-4m}+\beta_{12-4m}\beta_{4m-2}+\beta_{14-4m}\beta_{4-4m}+\beta_{16-4m}
\beta_{4m}+\\&&\beta_{18-4m}\beta_{2-4m})g_{\varphi-3m+5}+(\beta_{2-4m}\beta_{12-4m}+\beta_{4-4m}\beta_{4m-8}+
\beta_{6-4m}\beta_{10-4m}+\beta_{8-4m}\\&&\beta_{4m-6}+\beta_{10-4m}\beta_{8-4m}+\beta_{12-4m}\beta_{4m-4}
+\beta_{14-4m}\beta_{6-4m}+\beta_{16-4m}\beta_{4m-2}+\beta_{18-4m}\\&&\times\beta_{4-4m}+\beta_{20-4m}\beta_{4m}+
\beta_{22-4m}\beta_{2-4m})g_{\varphi-3m+6}+\ldots+(\beta_{4m-22}\beta_{4m}+\beta_{4m-20}\times\\&&\beta_{4-4m}+\beta_{4m-18} \beta_{4m-2}+\beta_{4m-16}\beta_{6-4m}+\beta_{4m-14}\beta_{4m-4}+\beta_{4m-12}\beta_{8-4m}+\beta_{4m-10}
\\&&\times\beta_{4m-6}+\beta_{4m-8}\beta_{10-4m}+\beta_{4m-6}\beta_{4m-8}+\beta_{4m-4}\beta_{12-4m}+\beta_{4m-2}\beta_{4m-10}
+\beta_{4m}\times\\&&\beta_{14-4m})g_{\varphi+3m-6}+(\beta_{4m-20}\beta_{2-4m}+\beta_{4m-18} \beta_{4m}+\beta_{4m-16}\beta_{4-4m}+\beta_{4m-14}\beta_{4m-2}+\\&&\beta_{4m-12}\beta_{6-4m}+\beta_{4m-10}
\beta_{4m-4}+\beta_{4m-8}\beta_{8-4m}+\beta_{4m-6}^{2}+\beta_{4m-4}\beta_{10-4m}+\beta_{4m-2}\times\\&&\beta_{4m-8}
+\beta_{4m}\beta_{12-4m}) g_{\varphi+3m-5}+(\beta_{4m-16}\beta_{2-4m}+\beta_{4m-14}\beta_{4m}+\beta_{4m-12}
\beta_{4-4m}+\\&&\beta_{4m-10}\beta_{4m-2}+\beta_{4m-8} \beta_{6-4m}+\beta_{4m-6}\beta_{4m-4}+\beta_{4m-4}\beta_{8-4m}+
\beta_{4m-2}\beta_{4m-6}+\beta_{4m}\\&&\times\beta_{10-4m})g_{\varphi+3m-4}+(\beta_{4m-12}
\beta_{2-4m}+\beta_{4m-10}\beta_{4m}+\beta_{4m-8}\beta_{4-4m}+\beta_{4m-6}\beta_{4m-2}+\\&&\beta_{4m-4}\beta_{6-4m}+
\beta_{4m-2}\beta_{4m-4}+\beta_{4m}\beta_{8-4m})g_{\varphi+3m-3}+(\beta_{4m-8}\beta_{2-4m}+\beta_{4m-6}
\beta_{4m}+\\&&\beta_{4m-4}\beta_{4-4m}+\beta_{4m-2}^{2}+\beta_{4m}\beta_{6-4m})g_{\varphi+3m-2}+(\beta_{4m-4}
\beta_{2-4m}+\beta_{4m-2}\beta_{4m}+\beta_{4m}\times\\&&\beta_{4-4m})g_{\varphi+3m-1}+(\beta_{4m}\beta_{2-4m})g_{\varphi+3m}.
\end{eqnarray*}
Which can be written in the following compact form
\begin{eqnarray}\label{const31}
	\left\{\begin{array}{cccc}
		g_{4\varphi-2}&=&\sum\limits_{\lambda=-m+1}^{m}\beta_{4-4\lambda}\left(\sum\limits_{\alpha=-2m}^{2m-1}\beta_{-2\alpha} \,\ g_{\varphi+\alpha+\lambda}\right)+\sum\limits_{\lambda=-m+1}^{m}\beta_{2-4\lambda} \left(\sum\limits_{\alpha=-2m}^{2m-1}\beta_{2+2\alpha} \,\ g_{\varphi+\alpha+\lambda}\right),\\ \\
		g_{4\varphi-1}&=&\sum\limits_{\lambda=-m+1}^{m}\beta_{4\lambda-2}\left(\sum\limits_{\alpha=-2m}^{2m-1}\beta_{-2\alpha} \,\ g_{\varphi+\alpha+\lambda}\right)+\sum\limits_{\lambda=-m+1}^{m}\beta_{4\lambda}
		\left(\sum\limits_{\alpha=-2m}^{2m-1}\beta_{2+2\alpha} \,\ g_{\varphi+\alpha+\lambda}\right),\\ \\
		g_{4\varphi}&=&\sum\limits_{\lambda=-m+1}^{m}\beta_{4-4\lambda}\left(\sum\limits_{\alpha=-2m}^{2m-1}\beta_{2+2\alpha} \,\ g_{\varphi+\alpha+\lambda}\right)+\sum\limits_{\lambda=-m+1}^{m}\beta_{2-4\lambda}
		\left(\sum\limits_{\alpha=-2m}^{2m-1}\beta_{-2\alpha} \,\ g_{\varphi+\alpha+\lambda+1}\right),\\ \\
		g_{4\varphi+1}&=&\sum\limits_{\lambda=-m+1}^{m}\beta_{4\lambda-2}\left(\sum\limits_{\alpha=-2m}^{2m-1}
		\beta_{2+2\alpha} \,\ g_{\varphi+\alpha+\lambda}\right)+\sum\limits_{\lambda=-m+1}^{m}\beta_{4\lambda}
		\left(\sum\limits_{\alpha=-2m}^{2m-1}\beta_{-2\alpha} \,\ g_{\varphi+\alpha+\lambda+1}\right).
	\end{array}\right.
\end{eqnarray}
This completes the proof.
\end{proof}
The following theorem presents a connection between the $4m$-point binary and the $(6m-1)$-point relaxed quaternary subdivision schemes.
\begin{thm}\label{thm-even-2}
If $n=2m$, then the subdivision equations given in (\ref{const31}) gives the four subdivision rules of the $(6m-1)$-point relaxed quaternary subdivision scheme whose  coefficients of the control points in the subdivision rules are the non-linear combination of the coefficients of the control points of the $4m$-point binary subdivision scheme.
\end{thm}
\begin{proof}
Now we add the subdivision level on the subdivision rules given in (\ref{const31}), Hence we get the following $(6m-1)$-point relaxed quaternary subdivision scheme
\begin{eqnarray}\label{1q}
\left\{\begin{array}{cccc}
g_{4\varphi-2}^{k+1}&=&\sum\limits_{\lambda=-m+1}^{m}\sum\limits_{\alpha=-2m}^{2m-1}\beta_{4-4\lambda}\beta_{-2\alpha} \,\ g_{\varphi+\alpha+\lambda}^{k}+\sum\limits_{\lambda=-m+1}^{m}\sum\limits_{\alpha=-2m}^{2m-1}\beta_{2-4\lambda} \beta_{2+2\alpha} \,\ g_{\varphi+\alpha+\lambda}^{k},\\ \\

g_{4\varphi-1}^{k+1}&=&\sum\limits_{\lambda=-m+1}^{m}\sum\limits_{\alpha=-2m}^{2m-1}\beta_{4\lambda-2}\beta_{-2\alpha} \,\ g_{\varphi+\alpha+\lambda}^{k}+\sum\limits_{\lambda=-m+1}^{m}\sum\limits_{\alpha=-2m}^{2m-1}\beta_{4\lambda}
\beta_{2+2\alpha} \,\ g_{\varphi+\alpha+\lambda}^{k},\\ \\

g_{4\varphi}^{k+1}&=&\sum\limits_{\lambda=-m+1}^{m}\sum\limits_{\alpha=-2m}^{2m-1}\beta_{4-4\lambda}\beta_{2+2\alpha} \,\ g_{\varphi+\alpha+\lambda}^{k}+\sum\limits_{\lambda=-m+1}^{m}\sum\limits_{\alpha=-2m}^{2m-1}\beta_{2-4\lambda}
\beta_{-2\alpha} \,\ g_{\varphi+\alpha+\lambda+1}^{k},\\ \\

g_{4\varphi+1}^{k+1}&=&\sum\limits_{\lambda=-m+1}^{m}\sum\limits_{\alpha=-2m}^{2m-1}\beta_{4\lambda-2}
\beta_{2+2\alpha} \,\ g_{\varphi+\alpha+\lambda}^{k}+\sum\limits_{\lambda=-m+1}^{m}\sum\limits_{\alpha=-2m}^{2m-1}\beta_{4\lambda}
\beta_{-2\alpha} \,\ g_{\varphi+\alpha+\lambda+1}^{k}.
\end{array}\right.
\end{eqnarray}
The mask coefficients of quaternary subdivision scheme (\ref{1q}) is the non-linear combination of the mask of the following $4m$-point binary subdivision scheme which we get by using $n=2m$ in (\ref{const1}).
\begin{equation}\label{1b}
\left\{\begin{array}{c}
{g}_{2\varphi-1}^{k+1}=\sum\limits_{\lambda=-2m+1}^{2m}\beta_{2\lambda}\,\ g_{\varphi+\lambda-1}^{k},\\ \\
{g}_{2\varphi}^{k+1}=\sum\limits_{\lambda=-2m+1}^{2m}\beta_{2-2\lambda} \,\ g_{\varphi+\lambda}^{k},
\end{array}\right.
\end{equation}
\end{proof}
The given theorems prove the generlized results about the odd $n$. 
\begin{thm}\label{thm-odd-1}
	If $n$ is odd, that is $n=2m+1: m\in \mathbb{N}$, then the subdivision rules $g_{4\varphi-2}$ and $g_{4\varphi-1}$ in (\ref{const6}) are the linear combination of $6m+3$ control points $g_{\varphi-3m-1}\ldots g_{\varphi+3m+1}$, while the subdivision rules $g_{4\varphi}$ and $g_{4\varphi+1}$ in (\ref{const6}) are the linear combination of $6m+2$ control points $g_{\varphi-3m}\ldots g_{\varphi+3m+1}$. 
\end{thm}
\begin{proof}
When $n$ is odd, we put $n=2m+1$ in (\ref{const6}). That is
\begin{equation}\label{const51}
\left\{\begin{array}{cccc}
g_{4\varphi-2}=\sum\limits_{\lambda=-2m}^{2m+1}\beta_{2-2\lambda} \,\ g_{2\varphi+\lambda-1},\,\ \,\ \,\
g_{4\varphi-1}=\sum\limits_{\lambda=-2m}^{2m+1}\beta_{2\lambda} \,\ g_{2\varphi+\lambda-1},\\\\
g_{4\varphi}=\sum\limits_{\lambda=-2m}^{2m+1}\beta_{2-2\lambda} \,\ g_{2\varphi+\lambda},\,\ \,\ \,\
g_{4\varphi+1}=\sum\limits_{\lambda=-2m}^{2m+1}\beta_{2\lambda} \,\ g_{2\varphi+\lambda}.
\end{array}\right.
\end{equation}
Now we have to find out values of $g_{2\varphi-2m-1}$, $g_{2\varphi-2m}$, $g_{2\varphi-2m+1}$, $\ldots$, $g_{2\varphi+2m}$, $g_{2\varphi+2m+1}$. For this, first we evaluate (\ref{aa2}) for $n=2m+1$ and then by increasing or decreasing the subscript, we get the required unknowns.
\begin{eqnarray}\label{const53}
\left\{\begin{array}{cccc}
 g_{2\varphi-1}&=&\beta_{-4m} \,\ g_{\varphi-2m-1}+ \beta_{2-4m} \,\ g_{\varphi-2m}+\beta_{4-4m} \,\ g_{\varphi-2m+1}+\beta_{6-4m} \,\ g_{\varphi-2m+2}+\beta_{8-4m}\\&&\times g_{\varphi-2m+3} +\beta_{10-4m} \,\ g_{\varphi-2m+4}+\beta_{12-4m} \,\ g_{\varphi-2m+5}+\ldots +\beta_{4m-8} \,\ g_{\varphi+2m-5}+\\&& \beta_{4m-6} \,\ g_{\varphi+2m-4}+\beta_{4m-4} \,\ g_{\varphi+2m-3} +\beta_{4m-2} \,\ g_{\varphi+2m-2}+\beta_{4m} \,\ g_{\varphi+2m-1}+\beta_{4m+2}\\&&\times g_{\varphi+2m},\\
 g_{2\varphi}&=&\beta_{4m+2} \,\ g_{\varphi-2m}+\beta_{4m} \,\ g_{\varphi-2m+1}+\beta_{4m-2} \,\ g_{\varphi-2m+2}+\beta_{4m-4} \,\ g_{\varphi-2m+3}+\beta_{4m-6}\\&&\times g_{\varphi-2m+4} +\beta_{4m-8} \,\ g_{\varphi-2m+5}+\beta_{4m-10} \,\ g_{\varphi-2m+6}+\ldots+\beta_{10-4m} \,\ g_{\varphi+2m-4}+ \\&&\beta_{8-4m} \,\ g_{\varphi+2m-3}+\beta_{6-4m} \,\ g_{\varphi+2m-2}+\beta_{4-4m} \,\ g_{\varphi+2m-1}+\beta_{2-4m}\,\ g_{\varphi+2m}+\beta_{-4m}\\&&\times g_{\varphi+2m+1}.
 \end{array}\right.
\end{eqnarray}
Now we replace $\varphi$ by $\varphi-m$ in the first and second rules of (\ref{const53}), we get
\begin{eqnarray*}\label{const55}
g_{2\varphi-2m-1}&=& \beta_{-4m} \,\ g_{\varphi-3m-1}+\beta_{2-4m} \,\ g_{\varphi-3m}+\beta_{4-4m} \,\ g_{\varphi-3m+1}+\beta_{6-4m} \,\ g_{\varphi-3m+2}+\beta_{8-4m}\\&&\times g_{\varphi-3m+3}+ \beta_{10-4m} \,\ g_{\varphi-3m+4}+\beta_{12-4m} \,\ g_{\varphi-3m+5}+\beta_{14-4m} \,\ g_{\varphi-3m+6}+\ldots +\\&&\beta_{4m-8} \,\ g_{\varphi+m-5}+\beta_{4m-6}\,\ g_{\varphi+m-4}+\beta_{4m-4} \,\ g_{\varphi+m-3}+\beta_{4m-2} \,\ g_{\varphi+m-2}+\beta_{4m}\times\\&& g_{\varphi+m-1}+\beta_{4m+2}g_{\varphi+m},\\
g_{2\varphi-2m}&=&\beta_{4m+2} \,\ g_{\varphi-3m}+\beta_{4m} \,\ g_{\varphi-3m+1}+\beta_{4m-2} \,\ g_{\varphi-3m+2}+\beta_{4m-4} \,\ g_{\varphi-3m+3}+\beta_{4m-6}\\&&\times g_{\varphi-3m+4} +\beta_{4m-8} \,\ g_{\varphi-3m+5}+\beta_{4m-10} \,\ g_{\varphi-3m+6}+\ldots +\beta_{10-4m} \,\ g_{\varphi+m-4}+\\&& \beta_{8-4m} \,\ g_{\varphi+m-3}+\beta_{6-4m} \,\ g_{\varphi+m-2} +\beta_{4-4m} \,\ g_{\varphi+m-1}+\beta_{2-4m} \,\ g_{\varphi+m}+\beta_{-4m} \\&&\times g_{\varphi+m+1}.
\end{eqnarray*}
Now we replace  $\varphi$ by $\varphi-m+1$ in the first and second rules of (\ref{const53}), we get
\begin{eqnarray*}\label{const57}
g_{2\varphi-2m+1}&=&\beta_{-4m} \,\ g_{\varphi-3m}+\beta_{2-4m} \,\ g_{\varphi-3m+1}+\beta_{4-4m} \,\ g_{\varphi-3m+2}+\beta_{6-4m} \,\ g_{\varphi-3m+3}+\beta_{8-4m}\\&&\times g_{\varphi-3m+4}+\beta_{10-4m} \,\ g_{\varphi-3m+5}+\beta_{12-4m} \,\ g_{\varphi-3m+6}+\ldots +\beta_{4m-8} \,\ g_{\varphi+m-4}+\\&& \beta_{4m-6} g_{\varphi+m-3}+\beta_{4m-4} \,\ g_{\varphi+m-2}+\beta_{4m-2} \,\ g_{\varphi+m-1}+\beta_{4m} \,\ g_{\varphi+m}+\beta_{4m+2}\\&&\times g_{\varphi+m+1},\\
g_{2\varphi-2m+2}&=& \beta_{4m+2} \,\ g_{\varphi-3m+1}+\beta_{4m} \,\ g_{\varphi-3m+2}+\beta_{4m-2} \,\ g_{\varphi-3m+3}+\beta_{4m-4} \,\ g_{\varphi-3m+4}+\beta_{4m-6} \\&&\times g_{\varphi-3m+5} +\beta_{4m-8} \,\ g_{\varphi-3m+6}+\beta_{4m-10} \,\ g_{\varphi-3m+7}+\ldots +\beta_{10-4m} \,\ g_{\varphi+m-3}+\\&& \beta_{8-4m} \,\ g_{\varphi+m-2}+\beta_{6-4m} \,\ g_{\varphi+m-1} +\beta_{4-4m} \,\ g_{\varphi+m}+\beta_{2-4m} \,\ g_{\varphi+m+1}+\beta_{-4m}\\&&\times g_{\varphi+m+2}.
\end{eqnarray*}
Now we replace $\varphi$ by $\varphi-m+2$ in the first and second rules of (\ref{const53}), we get
\begin{eqnarray*}\label{const59}
 g_{2\varphi-2m+3}&=& \beta_{-4m} \,\ g_{\varphi-3m+1}+\beta_{2-4m} \,\ g_{\varphi-3m+2}+\beta_{4-4m} \,\ g_{\varphi-3m+3}+\beta_{6-4m} \,\ g_{\varphi-3m+4}+\\&& \beta_{8-4m}\,\ g_{\varphi-3m+5} +\beta_{10-4m} \,\ g_{\varphi-3m+6}+\beta_{12-4m} \,\ g_{\varphi-3m+7}+\ldots +\beta_{4m-8} \times \\&& g_{\varphi+m-3}+ \beta_{4m-6} \,\ g_{\varphi+m-2}+\beta_{4m-4} \,\ g_{\varphi+m-1} +\beta_{4m-2} \,\ g_{\varphi+m}+\beta_{4m}\,\ g_{\varphi+m+1}\\&&+\beta_{4m+2}\,\ g_{\varphi+m+2},\\
 g_{2\varphi-2m+4}&=& \beta_{4m+2} \,\ g_{\varphi-3m+2}+\beta_{4m} \,\ g_{\varphi-3m+3}+\beta_{4m-2} \,\ g_{\varphi-3m+4}+\beta_{4m-4} \,\ g_{\varphi-3m+5}+\\&&\beta_{4m-6}\,\ g_{\varphi-3m+6} +\beta_{4m-8} \,\ g_{\varphi-3m+7}+\beta_{4m-10} \,\ g_{\varphi-3m+8}+\ldots +\beta_{10-4m}\\&&\times g_{\varphi+m-2}+ \beta_{8-4m} \,\ g_{\varphi+m-1}+\beta_{6-4m} \,\ g_{\varphi+m} +\beta_{4-4m} \,\ g_{\varphi+m+1}+\beta_{2-4m} \\&&\times g_{\varphi+m+2}+\beta_{-4m} \,\ g_{\varphi+m+3}.
\end{eqnarray*}
  \,\ \,\ \,\ \,\ \,\ \,\ \,\ \,\ \,\ \,\ \,\ \,\ \,\ \,\ \,\ \,\ \,\ \,\ \,\ \,\ \,\ \vdots \,\ \,\ \,\ \,\ \,\ \,\ \,\ \,\ \,\ \,\ \,\ \,\ \,\ \,\ \,\ \,\ \,\ \,\ \,\ \,\ \,\ \vdots \,\ \,\ \,\ \,\ \,\ \,\ \,\ \,\ \,\ \,\ \,\ \,\ \,\ \,\ \,\ \,\ \,\ \,\ \,\ \,\ \,\ \vdots \\
continuiting this process, we replace $\varphi$ by $\varphi+m-2$ in first and second rules of (\ref{const53}), we get
\begin{eqnarray*}\label{const63}
g_{2\varphi+2m-5}&=&\beta_{-4m} \,\ g_{\varphi-m-3}+\beta_{2-4m} \,\ g_{\varphi-m-2}+\beta_{4-4m} \,\ g_{\varphi-m-1}+\beta_{6-4m} \,\ g_{\varphi-m}+\beta_{8-4m} \\&&\times g_{\varphi-m+1}+ \beta_{10-4m} \,\ g_{\varphi-m+2}+\beta_{12-4m} \,\ g_{\varphi-m+3}+\ldots +\beta_{4m-8} \,\ g_{\varphi+3m-7}+\\&&\beta_{4m-6} \,\ g_{\varphi+3m-6}+\beta_{4m-4} \,\ g_{\varphi+3m-5}+\beta_{4m-2} \,\ g_{\varphi+3m-4}+\beta_{4m} \,\ g_{\varphi+3m-3}+\\&& \beta_{4m+2}\,\ g_{\varphi+3m-2},\\
g_{2\varphi+2m-4}&=&\beta_{4m+2} \,\ g_{\varphi-m-2}+\beta_{4m} \,\ g_{\varphi-m-1}+\beta_{4m-2} \,\ g_{\varphi-m}+\beta_{4m-4} \,\ g_{\varphi-m+1}+\beta_{4m-6}\times\\&& g_{\varphi-m+2}+ \beta_{4m-8} \,\ g_{\varphi-m+3}+\beta_{4m-10} \,\ g_{\varphi-m+4}+\ldots  +\beta_{10-4m} \,\ g_{\varphi+3m-6}+\beta_{8-4m}\\&& \times g_{\varphi+3m-5}+\beta_{6-4m} \,\ g_{\varphi+3m-4}+\beta_{4-4m} \,\ g_{\varphi+3m-3}+\beta_{2-4m} \,\ g_{\varphi+3m-2}+\beta_{-4m}\\&&\times g_{\varphi+3m-1}.
\end{eqnarray*}
Now we replace  $\varphi$ by $\varphi+m-1$ in the first and second rules of (\ref{const53}), we get
\begin{eqnarray*}\label{const65}
g_{2\varphi+2m-3}&=& \beta_{-4m} \,\ g_{\varphi-m-2}+\beta_{2-4m} \,\ g_{\varphi-m-1}+\beta_{4-4m} \,\ g_{\varphi-m}+\beta_{6-4m} \,\ g_{\varphi-m+1} +\beta_{8-4m}\\&&\times g_{\varphi-m+2}+ \beta_{10-4m} \,\ g_{\varphi-m+3}+\beta_{12-4m} \,\ g_{\varphi-m+4}+\ldots+\beta_{4m-8} \,\ g_{\varphi+3m-6}+\\&&\beta_{4m-6}\,\ g_{\varphi+3m-5}+\beta_{4m-4} \,\ g_{\varphi+3m-4} +\beta_{4m-2} \,\ g_{\varphi+3m-3}+\beta_{4m} \,\ g_{\varphi+3m-2}+\\&& \beta_{4m+2}\,\ g_{\varphi+3m-1},\\
g_{2\varphi+2m-2}&=& \beta_{4m+2} \,\ g_{\varphi-m-1}+\beta_{4m} \,\ g_{\varphi-m}+\beta_{4m-2} \,\ g_{\varphi-m+1}+\beta_{4m-4} \,\ g_{\varphi-m+2}+\beta_{4m-6}\\&&\times g_{\varphi-m+3}+\beta_{4m-8} \,\ g_{\varphi-m+4}+\beta_{4m-10} \,\ g_{\varphi-m+5}+\ldots +\beta_{10-4m} \,\ g_{\varphi+3m-5}+\\&& \beta_{8-4m}\,\ g_{\varphi+3m-4}+\beta_{6-4m} \,\ g_{\varphi+3m-3}+\beta_{4-4m} \,\ g_{\varphi+3m-2}+\beta_{2-4m} \,\ g_{\varphi+3m-1}+\\&& \beta_{-4m}\,\ g_{\varphi+3m}.
\end{eqnarray*}
Now we replace $\varphi$ by $\varphi+m$ in the first and second rules of (\ref{const53}), we get
\begin{eqnarray*}\label{const67}
g_{2\varphi+2m-1}&=&\beta_{-4m} \,\ g_{\varphi-m-1}+\beta_{2-4m} \,\ g_{\varphi-m}+\beta_{4-4m} \,\ g_{\varphi-m+1}+\beta_{6-4m} \,\ g_{\varphi-m+2} +\beta_{8-4m}\\&&\times g_{\varphi-m+3}+ \beta_{10-4m} \,\ g_{\varphi-m+4}+\beta_{12-4m} \,\ g_{\varphi-m+5}+\ldots+\beta_{4m-8} \,\ g_{\varphi+3m-5}+\\&&\beta_{4m-6}\,\ g_{\varphi+3m-4}+\beta_{4m-4} \,\ g_{\varphi+3m-3}+\beta_{4m-2} \,\ g_{\varphi+3m-2}+\beta_{4m} \,\ g_{\varphi+3m-1}+\\&&\beta_{4m+2} \,\ g_{\varphi+3m},\\
g_{2\varphi+2m}&=& \beta_{4m+2} \,\ g_{\varphi-m}+\beta_{4m} \,\ g_{\varphi-m+1}+\beta_{4m-2} \,\ g_{\varphi-m+2}+\beta_{4m-4} \,\ g_{\varphi-m+3}+\beta_{4m-6}\\&&\times g_{\varphi-m+4}+\beta_{4m-8} \,\ g_{\varphi-m+5}+\beta_{4m-10} \,\ g_{\varphi-m+6}+\ldots +\beta_{10-4m} \,\ g_{\varphi+3m-4}+\\&&\beta_{8-4m}\,\ g_{\varphi+3m-3}+\beta_{6-4m} \,\ g_{\varphi+3m-2}+\beta_{4-4m} \,\ g_{\varphi+3m-1}+\beta_{2-4m} \,\ g_{\varphi+3m}+\\&&\beta_{-4m} \,\ g_{\varphi+3m+1}.
\end{eqnarray*}
Now we replace $\varphi$ by $\varphi+m+1$ in the first rule of (\ref{const53}), we get
\begin{eqnarray*}\label{const69}
g_{2\varphi+2m+1}&=& \beta_{-4m} \,\ g_{\varphi-m}+\beta_{2-4m} \,\ g_{\varphi-m+1}+\beta_{4-4m} \,\ g_{\varphi-m+2}+\beta_{6-4m} \,\ g_{\varphi-m+3}+\beta_{8-4m}\\&&\times g_{\varphi-m+4}+ \beta_{10-4m} \,\ g_{\varphi-m+5}+\beta_{12-4m} \,\ g_{\varphi-m+6}+\ldots+\beta_{4m-8} \,\ g_{\varphi+3m-4}+\\&&\beta_{4m-6}\,\ g_{\varphi+3m-3}+\beta_{4m-4} \,\ g_{\varphi+3m-2}+\beta_{4m-2} \,\ g_{\varphi+3m-1}+\beta_{4m} \,\ g_{\varphi+3m}+\beta_{4m+2}\\&&\times g_{\varphi+3m+1}.
\end{eqnarray*}
Now, we get all the unknowns $g_{2\varphi-2m-1}$, $g_{2\varphi-2m}$, $g_{2\varphi-2m+1}$, $\ldots$, $g_{2\varphi+2m}$, $g_{2\varphi+2m+1}$. Further, we substitute these in the set of equations (\ref{const51}). Hence, we get
\begin{eqnarray*}\label{const74}
g_{4\varphi-2}&=&(\beta_{4m+2}\beta_{-4m}) g_{\varphi-3m-1}+(\beta_{4m+2}\beta_{2-4m}+\beta_{4m}\beta_{4m+2}+\beta_{4m-2}\beta_{-4m}) g_{\varphi-3m}+\\&&(\beta_{4m+2}\beta_{4-4m}+\beta_{4m}^{2}+\beta_{4m-2}\beta_{2-4m}+\beta_{4m-4}
\beta_{4m+2}+\beta_{4m-6}\beta_{-4m}) g_{\varphi-3m+1}
 \\&&+(\beta_{4m+2}\beta_{6-4m}+\beta_{4m}\beta_{4m-2}+\beta_{4m-2}\beta_{4-4m}+\beta_{4m-4}
\beta_{4m}+\beta_{4m-6}\beta_{2-4m}+\\&&\beta_{4m-8}\beta_{4m+2}+\beta_{4m-10}\beta_{-4m})g_{\varphi-3m+2}+
(\beta_{4m+2}\beta_{8-4m}+\beta_{4m}\beta_{4m-4}+\beta_{4m-2}\\&&\times \beta_{6-4m}+\beta_{4m-4}
\beta_{4m-2}+\beta_{4m-6}\beta_{4-4m}+\beta_{4m-8}\beta_{4m}+\beta_{4m-10}\beta_{2-4m}+\beta_{4m-12}\\&&\times
\beta_{4m+2}+\beta_{4m-14}\beta_{-4m})g_{\varphi-3m+3}+\ldots+(\beta_{16-4m}\beta_{-4m}+\beta_{14-4m}\beta_{4m+2}
+\\&&\beta_{12-4m}\beta_{2-4m}+
\beta_{10-4m}\beta_{4m}+\beta_{8-4m}\beta_{4-4m}+\beta_{6-4m}\beta_{4m-2}+\beta_{4-4m}\beta_{6-4m}+
\\&&\beta_{2-4m}\beta_{4m-4}+\beta_{-4m}\beta_{8-4m})g_{\varphi+3m-3}+(\beta_{12-4m}\beta_{-4m}+
\beta_{10-4m}\beta_{4m+2}+\beta_{8-4m}\\&&\times \beta_{2-4m}+\beta_{6-4m}\beta_{4m}+\beta_{4-4m}^{2}+
\beta_{2-4m}\beta_{4m-2}+\beta_{-4m}\beta_{6-4m})g_{\varphi+3m-2}+(\beta_{8-4m}\\&&\times \beta_{-4m}+\beta_{6-4m}\beta_{4m+2}+\beta_{4-4m}\beta_{2-4m}+\beta_{2-4m}\beta_{4m}+\beta_{-4m}\beta_{4-4m}) g_{\varphi+3m-1} +\\&&(\beta_{4-4m}\beta_{-4m}+\beta_{2-4m}\beta_{4m+2}+\beta_{-4m}\beta_{2-4m})g_{\varphi+3m} +(\beta_{-4m}^{2}) g_{\varphi+3m+1},
\end{eqnarray*}
\begin{eqnarray*}
g_{4\varphi-1}&=& (\beta_{-4m}^{2}) g_{\varphi-3m-1}+ (\beta_{-4m}\beta_{2-4m}+\beta_{2-4m}\beta_{4m+2}+\beta_{4-4m}\beta_{-4m}) g_{\varphi-3m}+(\beta_{-4m}\beta_{4-4m}\\&&+\beta_{2-4m}
\beta_{4m}+\beta_{4-4m}\beta_{2-4m}+\beta_{6-4m}\beta_{4m+2}+\beta_{8-4m}\beta_{-4m}) g_{\varphi-3m+1}
+(\beta_{-4m}\beta_{6-4m}+\\&&\beta_{2-4m}\beta_{4m-2}+\beta_{4-4m}^{2}+\beta_{6-4m}\beta_{4m}
+\beta_{8-4m}\beta_{2-4m}+\beta_{10-4m}\beta_{4m+2}+\beta_{12-4m}\beta_{-4m})\times\\&&g_{\varphi-3m+2}+(\beta_{-4m}
\beta_{8-4m}+\beta_{2-4m}\beta_{4m-4}+\beta_{4-4m}\beta_{6-4m}+\beta_{6-4m}\beta_{4m-2}
+\beta_{8-4m}\beta_{4-4m}\\&&+\beta_{10-4m}\beta_{4m}+\beta_{12-4m}\beta_{2-4m}+\beta_{14-4m}\beta_{4m+2}
+\beta_{16-4m}\beta_{-4m})g_{\varphi-3m+3}
+\ldots+
(\beta_{4m-14}\\&&\times \beta_{-4m}+\beta_{4m-12}\beta_{4m+2}+
\beta_{4m-10}\beta_{2-4m}+\beta_{4m-8}\beta_{4m}+\beta_{4m-6}\beta_{4-4m}+\beta_{4m-4}\beta_{4m-2}+\\&&\beta_{4m-2}
\beta_{6-4m}+\beta_{4m}\beta_{4m-4}+\beta_{4m+2}\beta_{8-4m})g_{\varphi+3m-3}+(\beta_{4m-10}\beta_{-4m}+\beta_{4m-8}
\beta_{4m+2}+\\&&\beta_{4m-6}\beta_{2-4m}+\beta_{4m-4}\beta_{4m}+\beta_{4m-2}\beta_{4-4m}+\beta_{4m}\beta_{4m-2}+
\beta_{4m+2}\beta_{6-4m})g_{\varphi+3m-2}+\\&&(\beta_{4m-6}\beta_{-4m}+\beta_{4m-4}\beta_{4m+2}+\beta_{4m-2}
\beta_{2-4m}+\beta_{4m}^{2}+\beta_{4m+2}\beta_{4-4m})g_{\varphi+3m-1}+\\&&(\beta_{4m-2}\beta_{-4m}+\beta_{4m}
\beta_{4m+2}+\beta_{4m+2}\beta_{2-4m})g_{\varphi-3m}+ (\beta_{4m+2}\beta_{-4m}) g_{\varphi-3m+1},
\end{eqnarray*}
\begin{eqnarray*}
g_{4\varphi}&=& (\beta_{4m+2}^{2}+\beta_{4m}\beta_{-4m}) g_{\varphi-3m}+ (\beta_{4m+2}\beta_{4m}+\beta_{4m}\beta_{2-4m}+\beta_{4m-2}\beta_{4m+2}+\beta_{4m-4}\beta_{-4m}) \\&&\times g_{\varphi-3m+1}+(\beta_{4m+2}\beta_{4m-2}+\beta_{4m}\beta_{4-4m}+\beta_{4m-2}\beta_{4m}+\beta_{4m-4}
\beta_{2-4m}+\beta_{4m-6}\beta_{4m+2}\\&&+\beta_{4m-8}\beta_{-4m})g_{\varphi-3m+2}+(\beta_{4m+2}\beta_{4m-4}+
\beta_{4m}\beta_{6-4m}+\beta_{4m-2}^{2}+\beta_{4m-4}\beta_{4-4m}+\beta_{4m-6}\\&&\times\beta_{4m}+\beta_{4m-8}
\beta_{2-4m}+\beta_{4m-10}\beta_{4m+2}+\beta_{4m-12}\beta_{-4m})g_{\varphi-3m+3}
+\ldots+
(\beta_{18-4m}\beta_{-4m}\\&&+\beta_{16-4m}\beta_{4m+2}+\beta_{14-4m}\beta_{2-4m}+\beta_{12-4m}
\beta_{4m}+\beta_{10-4m}\beta_{4-4m}+\beta_{8-4m}\beta_{4m-2}+\beta_{6-4m}^{2}\\&&+\beta_{4-4m}\beta_{4m-4}
+\beta_{2-4m}\beta_{8-4m}+\beta_{-4m}\beta_{4m-6})g_{\varphi+3m-3}+(\beta_{14-4m}\beta_{-4m}+\beta_{12-4m}
\beta_{4m+2}\\&&+\beta_{10-4m}\beta_{2-4m}+\beta_{8-4m}\beta_{4m}+\beta_{6-4m}\beta_{4-4m}+\beta_{4-4m}\beta_{4m-2}
+\beta_{2-4m}\beta_{6-4m}+\beta_{-4m}\times\\&&\beta_{4m-4}) g_{\varphi+3m-2}(\beta_{10-4m}\beta_{-4m} +\beta_{8-4m}\beta_{4m+2}+\beta_{6-4m}\beta_{2-4m}+\beta_{4-4m}\beta_{4m}+\beta_{2-4m}\beta_{4-4m}\\&& +\beta_{-4m}\beta_{4m-2})g_{\varphi+3m-1}+(\beta_{6-4m}\beta_{-4m}+\beta_{4-4m}\beta_{4m+2}+\beta_{2-4m}^{2}+
\beta_{-4m}\beta_{4m})g_{\varphi+3m}+\\&&(\beta_{2-4m}\beta_{-4m}+\beta_{-4m}\beta_{4m+2}) g_{\varphi+3m+1}
\end{eqnarray*}
and
\begin{eqnarray*}
 g_{4\varphi+1}&=&(\beta_{-4m}\beta_{4m+2}+\beta_{2-4m}\beta_{-4m}) g_{\varphi-3m}+(\beta_{-4m}\beta_{4m}+\beta_{2-4m}^{2}+\beta_{4-4m}\beta_{4m+2}+\beta_{6-4m}\\&&\times\beta_{-4m})
 g_{\varphi-3m+1}+(\beta_{-4m}\beta_{4m-2}+\beta_{2-4m}\beta_{4-4m}+\beta_{4-4m}\beta_{4m}+\beta_{6-4m}
 \beta_{2-4m}+\\&&\beta_{8-4m}\beta_{4m+2}+\beta_{10-4m}\beta_{-4m})g_{\varphi-3m+2}+(\beta_{-4m}\beta_{4m-4}
 +\beta_{2-4m}\beta_{6-4m}+\beta_{4-4m}\times\\&&\beta_{4m-2}+\beta_{6-4m}\beta_{4-4m}+\beta_{8-4m}\beta_{4m}+
\beta_{10-4m}\beta_{2-4m}+\beta_{12-4m}\beta_{4m+2}+\beta_{14-4m}\times\\&&\beta_{-4m})g_{\varphi-3m+3}
+\ldots+
(\beta_{4m-16}\beta_{-4m}+\beta_{4m-14}\beta_{4m+2}
+\beta_{4m-12}\beta_{2-4m}+\beta_{4m-10}\beta_{4m}\\&&+\beta_{4m-8}\beta_{4-4m}+\beta_{4m-6}\beta_{4m-2}+\beta_{4m-4}
\beta_{6-4m}+\beta_{4m-2}\beta_{4m-4}+\beta_{4m}\beta_{8-4m}+\beta_{4m+2}\\&&\times \beta_{4m-6})g_{\varphi+3m-3}+(\beta_{4m-12}\beta_{-4m}+\beta_{4m-10}\beta_{4m+2}+\beta_{4m-8}\beta_{2-4m}+
\beta_{4m-6}\beta_{4m}+\\&&\beta_{4m-4}
\beta_{4-4m}+\beta_{4m-2}^{2}+\beta_{4m} \beta_{6-4m}+\beta_{4m+2}\beta_{4m-4})g_{\varphi+3m-2}+
(\beta_{4m-8}\beta_{-4m}+\beta_{4m-6}\\&&\times \beta_{4m+2}+\beta_{4m-4}\beta_{2-4m}+\beta_{4m-2}\beta_{4m}+ \beta_{4m} \beta_{4-4m}+\beta_{4m+2}\beta_{4m-2})g_{\varphi+3m-1}+ (\beta_{4m-4}\\&&\times \beta_{-4m}+\beta_{4m-2}\beta_{4m+2}+\beta_{4m}\beta_{2-4m}+\beta_{4m+2}\beta_{4m}) g_{\varphi+3m}+(\beta_{4m}\beta_{-4m}+\beta_{4m+2}^{2}) g_{\varphi+3m+1}.
\end{eqnarray*}
Which, in short, can be written as
\begin{eqnarray}\label{const75}
\left\{\begin{array}{cccc}
g_{4\varphi-2}&=& \sum\limits_{\lambda=-m}^{m}\beta_{2-4\lambda}\left(\sum\limits_{\alpha=-2m}^{2m+1}\beta_{2\alpha} \,\ g_{\varphi+\alpha+\lambda-1}\right)+\sum\limits_{\lambda=-m}^{m}\beta_{-4\lambda}
\left(\sum\limits_{\alpha=-2m}^{2m+1}\beta_{2-2\alpha} \,\ g_{\varphi+\alpha+\lambda}\right),\\ \\
g_{4\varphi-1}&=& \sum\limits_{\lambda=-m}^{m}\beta_{4\lambda}\left(\sum\limits_{\alpha=-2m}^{2m+1}\beta_{2\alpha} \,\ g_{\varphi+\alpha+\lambda-1}\right)+\sum\limits_{\lambda=-m}^{m}\beta_{2+4\lambda}
\left(\sum\limits_{\alpha=-2m}^{2m+1}\beta_{2-2\alpha} \,\ g_{\varphi+\alpha+\lambda}\right),\\ \\
g_{4\varphi}&=& \sum\limits_{\lambda=-m}^{m}\beta_{2-4\lambda}\left(\sum\limits_{\alpha=-2m}^{2m+1}\beta_{2-2\alpha} \,\ g_{\varphi+\alpha+\lambda}\right)+\sum\limits_{\lambda=-m}^{m}\beta_{-4\lambda}
\left(\sum\limits_{\alpha=-2m}^{2m+1}\beta_{2\alpha} \,\ g_{\varphi+\alpha+\lambda}\right),\\ \\
g_{4\varphi+1}&=& \sum\limits_{\lambda=-m}^{m}\beta_{4\lambda}\left(\sum\limits_{\alpha=-2m}^{2m+1}\beta_{2-2\alpha} \,\ g_{\varphi+\alpha+\lambda}\right)+\sum\limits_{\lambda=-m}^{m}\beta_{2+4\lambda}
\left(\sum\limits_{\alpha=-2m}^{2m+1}\beta_{2\alpha} \,\ g_{\varphi+\alpha+\lambda}\right).
\end{array}\right.
\end{eqnarray}
Which completes the required result.
\end{proof}
The following theorem derives a relation between the $(4m+2)$-point binary and the $(6m+2)$-point relaxed quaternary subdivision schemes.
\begin{thm}\label{thm-odd-2}
	If $n=2m+1$, then the subdivision equations given in (\ref{const75}) gives the four subdivision rules of the $(6m+2)$-point relaxed quaternary subdivision scheme whose  coefficients of the control points in the subdivision rules are the non-linear combination of the coefficients of the control points of the $(4m+2)$-point binary subdivision scheme.
\end{thm}
\begin{proof}
	Now we add the subdivision level on the subdivision rules given in (\ref{const75}), Hence we get the following $(6m+2)$-point relaxed quaternary subdivision scheme
\begin{eqnarray}\label{const75a}
\left\{\begin{array}{cccc}
g_{4\varphi-2}^{k+1}&=& \sum\limits_{\lambda=-m}^{m}\sum\limits_{\alpha=-2m}^{2m+1}\beta_{2-4\lambda}\beta_{2\alpha} \,\ g_{\varphi+\alpha+\lambda-1}^{k}+\sum\limits_{\lambda=-m}^{m}\sum\limits_{\alpha=-2m}^{2m+1}\beta_{-4\lambda}
\beta_{2-2\alpha} \,\ g_{\varphi+\alpha+\lambda}^{k},\\ \\

g_{4\varphi-1}^{k+1}&=& \sum\limits_{\lambda=-m}^{m}\sum\limits_{\alpha=-2m}^{2m+1}\beta_{4\lambda}\beta_{2\alpha} \,\ g_{\varphi+\alpha+\lambda-1}^{k}+\sum\limits_{\lambda=-m}^{m}\sum\limits_{\alpha=-2m}^{2m+1}\beta_{2+4\lambda}
\beta_{2-2\alpha} \,\ g_{\varphi+\alpha+\lambda}^{k},\\ \\

g_{4\varphi}^{k+1}&=& \sum\limits_{\lambda=-m}^{m}\sum\limits_{\alpha=-2m}^{2m+1}\beta_{2-4\lambda}\beta_{2-2\alpha} \,\ g_{\varphi+\alpha+\lambda}^{k}+\sum\limits_{\lambda=-m}^{m}\sum\limits_{\alpha=-2m}^{2m+1}\beta_{-4\lambda}
\beta_{2\alpha} \,\ g_{\varphi+\alpha+\lambda}^{k},\\ \\

g_{4\varphi+1}^{k+1}&=& \sum\limits_{\lambda=-m}^{m}\sum\limits_{\alpha=-2m}^{2m+1}\beta_{4\lambda}\beta_{2-2\alpha} \,\ g_{\varphi+\alpha+\lambda}^{k}+\sum\limits_{\lambda=-m}^{m}\sum\limits_{\alpha=-2m}^{2m+1}\beta_{2+4\lambda}
\beta_{2\alpha} \,\ g_{\varphi+\alpha+\lambda}^{k}.
\end{array}\right.
\end{eqnarray}
The mask coefficients of quaternary subdivision scheme (\ref{const75a}) is the non-linear combination of the mask of the following $(4m+2)$-point binary subdivision scheme which we get by using $n=2m+1$ in (\ref{const1}).
\begin{equation}\label{const1a}
\left\{\begin{array}{c}
{g}_{2\varphi-1}^{k+1}=\sum\limits_{\lambda=-2m}^{2m+1}\beta_{2\lambda}\,\ g_{\varphi+\lambda-1}^{k},\\ \\
{g}_{2\varphi}^{k+1}=\sum\limits_{\lambda=-2m}^{2m+1}\beta_{2-2\lambda} \,\ g_{\varphi+\lambda}^{k},
\end{array}\right.
\end{equation}
Hence proved.
\end{proof}
In the next section, we will validate the results of Theorem \ref{thm-even-2} and Theorem \ref{thm-odd-2}.
\section{Applications of the presented techniques}
In this section, we implement and validate the results, which are proved in Theorem \ref{thm-even-2} and Theorem \ref{thm-odd-2} of Section 2,  to the known even-point  binary approximating subdivision schemes. We also inspects the graphical results of both type of schemes using the same initial data.We use non-parametric binary subdivision schemes in the Corollaries \ref{cor-1}-\ref{cor-7}, but the  given procedure can be applied on all the parametric as well as the non-parametric linear even-point binary subdivision schemes. The first three corollaries are the applications of Theorem \ref{thm-even-2} whereas the next four are the applications of Theorem \ref{thm-odd-2}.
\begin{cor}\label{cor-1}
We expand the binary subdivision scheme which is defined in (\ref{1b}) for $m=1$, thus we get
\begin{eqnarray}\label{constaa1}
\left\{\begin{array}{cccc}
{g}_{2\varphi-1}^{k+1}&=&\beta_{-2}\,\ g_{\varphi-2}^{k}+\beta_{0}\,\ g_{\varphi-1}^{k}+\beta_{2}\,\ g_{\varphi}^{k}+\beta_{4}\,\ g_{\varphi+1}^{k},\\
{g}_{2\varphi}^{k+1}&=&\beta_{4}\,\ g_{\varphi-1}^{k}+\beta_{2}\,\ g_{\varphi}^{k}+\beta_{0}\,\ g_{\varphi+1}^{k}+\beta_{-2}\,\ g_{\varphi+2}^{k}.
\end{array}\right.
\end{eqnarray}
To get  the values of $\beta_{-2}, \beta_{0}, \beta_{2}, \beta_{4}$, we compare the general form of the $4$-point binary scheme  (\ref{constaa1}) with the $4$-point scheme defined by \cite{Siddiqi2}, hence we get
\begin{eqnarray}\label{a41}
\beta_{-2}={\frac {1}{384}},\,\ \,\ \beta_{0}={\frac {121}{384}},\,\ \,\ \beta_{2}={\frac {235}{384}}, \,\ \,\ \beta_{4}={\frac {9}{128}}.
\end{eqnarray}
Now by using the mask  (\ref{a41}), we can get the mask/coefficients of the quaternary subdivision scheme. Hence by expanding (\ref{1q}) for $m=1$, we get
\begin{eqnarray*}
{g}_{4\varphi-2}^{k+1}&=&(\beta_{4}^{2}+\beta_{2}\beta_{-2})g_{\varphi-2}^{k}+(\beta_{4}\beta_{2}+\beta_{0}\beta_{4}+
\beta_{2}\beta_{0}+\beta_{-2}^{2})g_{\varphi-1}^{k}+(\beta_{4}\beta_{0}+\beta_{0}\beta_{2}+\beta_{2}^{2}+\\&& \beta_{-2}\beta_{0})
g_{\varphi}^{k}+(\beta_{4}\beta_{-2}+\beta_{0}^{2}+\beta_{2}\beta_{4}+\beta_{-2}\beta_{2})g_{\varphi+1}^{k}+(\beta_{0}\beta_{-2}
+\beta_{-2}\beta_{4})g_{\varphi+2}^{k},\\
{g}_{4\varphi-1}^{k+1}&=&(\beta_{-2}\beta_{4}+\beta_{0}\beta_{-2})g_{\varphi-2}^{k}+(\beta_{-2}\beta_{2}+\beta_{2}\beta_{4}+
\beta_{0}^{2}+\beta_{4}\beta_{-2})g_{\varphi-1}^{k}+(\beta_{-2}\beta_{0}+\beta_{2}^{2}\\&&+\beta_{0}\beta_{2}+\beta_{4}\beta_{0})
g_{\varphi}^{k}+(\beta_{-2}^{2}+\beta_{2}\beta_{0}+\beta_{0}\beta_{4}+\beta_{4}\beta_{2})g_{\varphi+1}^{k}+(\beta_{2}\beta_{-2}
+\beta_{4}^{2})g_{\varphi+2}^{k},\\
{g}_{4\varphi}^{k+1}&=&(\beta_{4}\beta_{-2})g_{\varphi-2}^{k}+(\beta_{4}\beta_{0}+\beta_{0}\beta_{-2}+\beta_{2}\beta_{4})
g_{\varphi-1}^{k}+(\beta_{4}\beta_{2}+\beta_{0}^{2}+\beta_{2}^{2}+\beta_{-2}\beta_{4})g_{\varphi}^{k}+\\&&(\beta_{4}^{2}+
\beta_{0}\beta_{2}+\beta_{2}\beta_{0}+\beta_{-2}\beta_{2})g_{\varphi+1}^{k}+(\beta_{0}\beta_{4}+\beta_{2}\beta_{-2}+
\beta_{-2}\beta_{0})g_{\varphi+2}^{k}+(\beta_{-2}^{2})g_{\varphi+3}^{k},\\
{g}_{4\varphi+1}^{k+1}&=&(\beta_{-2}^{2})g_{\varphi-2}^{k}+(\beta_{-2}\beta_{0}+\beta_{2}\beta_{-2}+
\beta_{0}\beta_{4})g_{\varphi-1}^{k}+(\beta_{-2}\beta_{2}+\beta_{2}\beta_{0}+\beta_{0}\beta_{2}+\beta_{4}^{2})
g_{\varphi}^{k}\\&&+(\beta_{-2}\beta_{4}+\beta_{2}^{2}+\beta_{0}^{2}+\beta_{4}\beta_{2})g_{\varphi+1}^{k}+(\beta_{2}\beta_{4}
+\beta_{0}\beta_{-2}+\beta_{4}\beta_{0})g_{\varphi+2}^{k}+(\beta_{4}\beta_{-2})g_{\varphi+3}^{k}.
\end{eqnarray*}
By using the values of $\beta_{-2}, \beta_{0}, \beta_{2}$ and $\beta_{4}$ from (\ref{a41}) in above, we get
\begin{eqnarray}\label{constaa2}
\left\{\begin{array}{cccc}
{g}_{4\varphi-2}^{k+1}&=&\hat{A}_{1}g_{\varphi-2}^{k}+\hat{A}_{2}g_{\varphi-1}^{k}+\hat{A}_{3}g_{\varphi}^{k}+
\hat{A}_{4}g_{\varphi+1}^{k}+\hat{A}_{5}g_{\varphi+2}^{k},\\ \\
{g}_{4\varphi-1}^{k+1}&=&\hat{A}_{5}g_{\varphi-2}^{k}+\hat{A}_{4}g_{\varphi-1}^{k}+\hat{A}_{3}g_{\varphi}^{k}
+\hat{A}_{2}g_{\varphi+1}^{k}+\hat{A}_{1}g_{\varphi+2}^{k},\\ \\
{g}_{4\varphi}^{k+1}&=&\hat{B}_{1}g_{\varphi-2}^{k}+\hat{B}_{2}g_{\varphi-1}^{k}+\hat{B}_{3}g_{\varphi}^{k}
+\hat{B}_{4}g_{\varphi+1}^{k}+\hat{B}_{5}g_{\varphi+2}^{k}+\hat{B}_{6}g_{\varphi+3}^{k},\\ \\
{g}_{4\varphi+1}^{k+1}&=&\hat{B}_{6}g_{\varphi-2}^{k}+\hat{B}_{5}g_{\varphi-1}^{k}+\hat{B}_{4}g_{\varphi}^{k}
+\hat{B}_{3}g_{\varphi+1}^{k}+\hat{B}_{2}g_{\varphi+2}^{k}+\hat{B}_{1}g_{\varphi+3}^{k},
\end{array}\right.
\end{eqnarray}
where
\begin{eqnarray}\label{a42}
\left\{\begin{array}{cccc}
 \hat{A}_{1}&=&{\frac{241}{36864}}, \,\ \,\ \hat{A}_{2}={\frac{1189}{4608}}, \,\ \,\ \hat{A}_{3}={\frac{1209}{2048}},\,\ \,\ \hat{A}_{4}={\frac{83}{576}}, \,\ \,\ \hat{A}_{5}={\frac{37}{36864}},\,\ \,\ \hat{B}_{1}={\frac{3}{16384}}, \\ \\ \hat{B}_{2}&=&{\frac{9733}{147456}},\,\ \,\ \hat{B}_{3}={\frac{38119}{73728}}, \,\ \,\ \hat{B}_{4}={\frac{3213}{8192}},\,\ \,\ \hat{B}_{5}={\frac{3623}{147456}}, \,\ \,\ \hat{B}_{6}={\frac{1}{147456}}
 \end{array}\right.
\end{eqnarray}
Which is the $5$-point relaxed quaternary subdivision scheme. The mask/coefficients of this quaternary subdivision scheme (\ref{constaa2}) is just the  non-linear combination of the mask of the binary subdivision scheme  (\ref{constaa1}). 
\end{cor}
\begin{figure}[h]
\begin{tabular}{cccc}
\epsfig{file=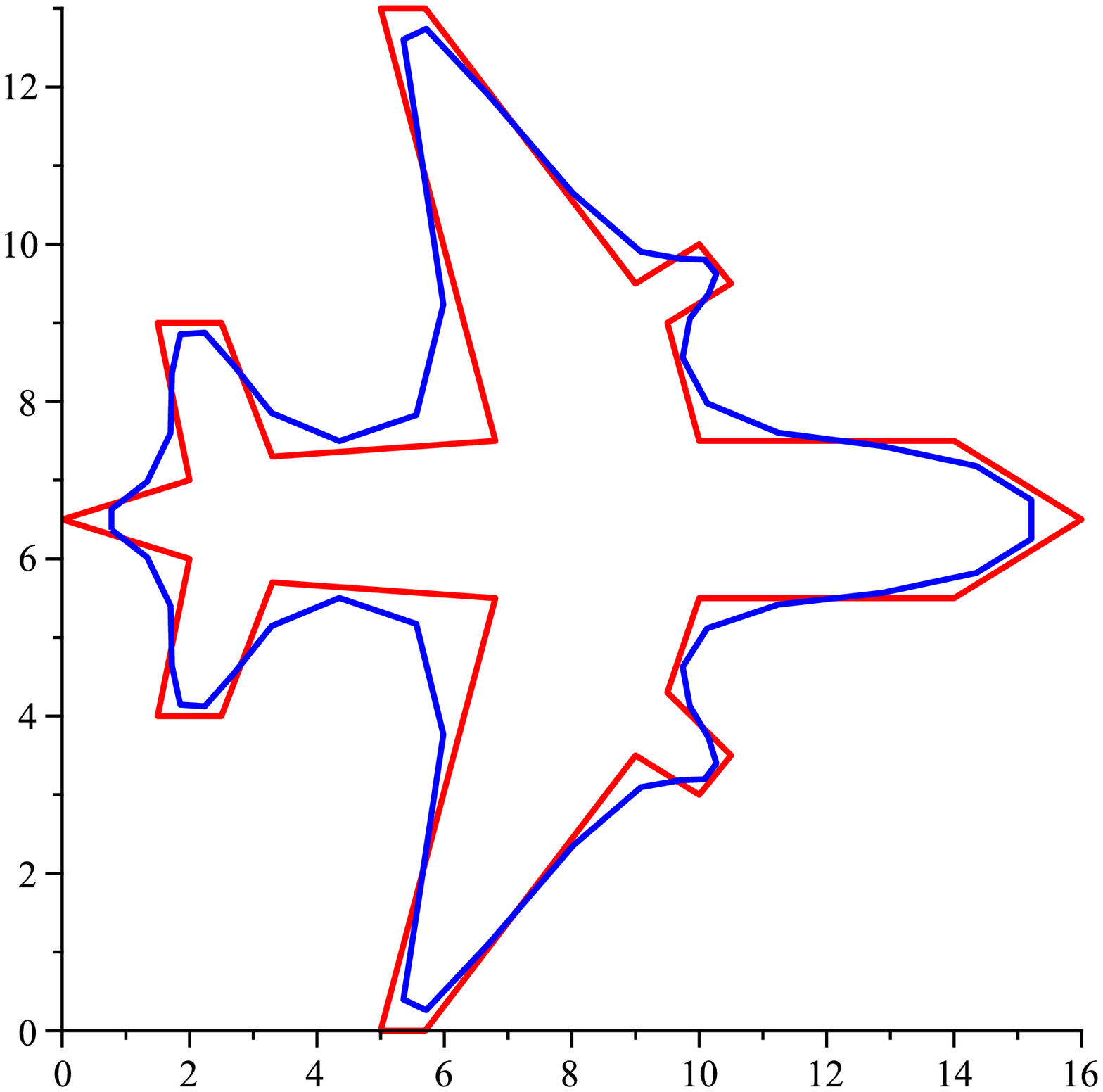, width=1.5 in} & \epsfig{file=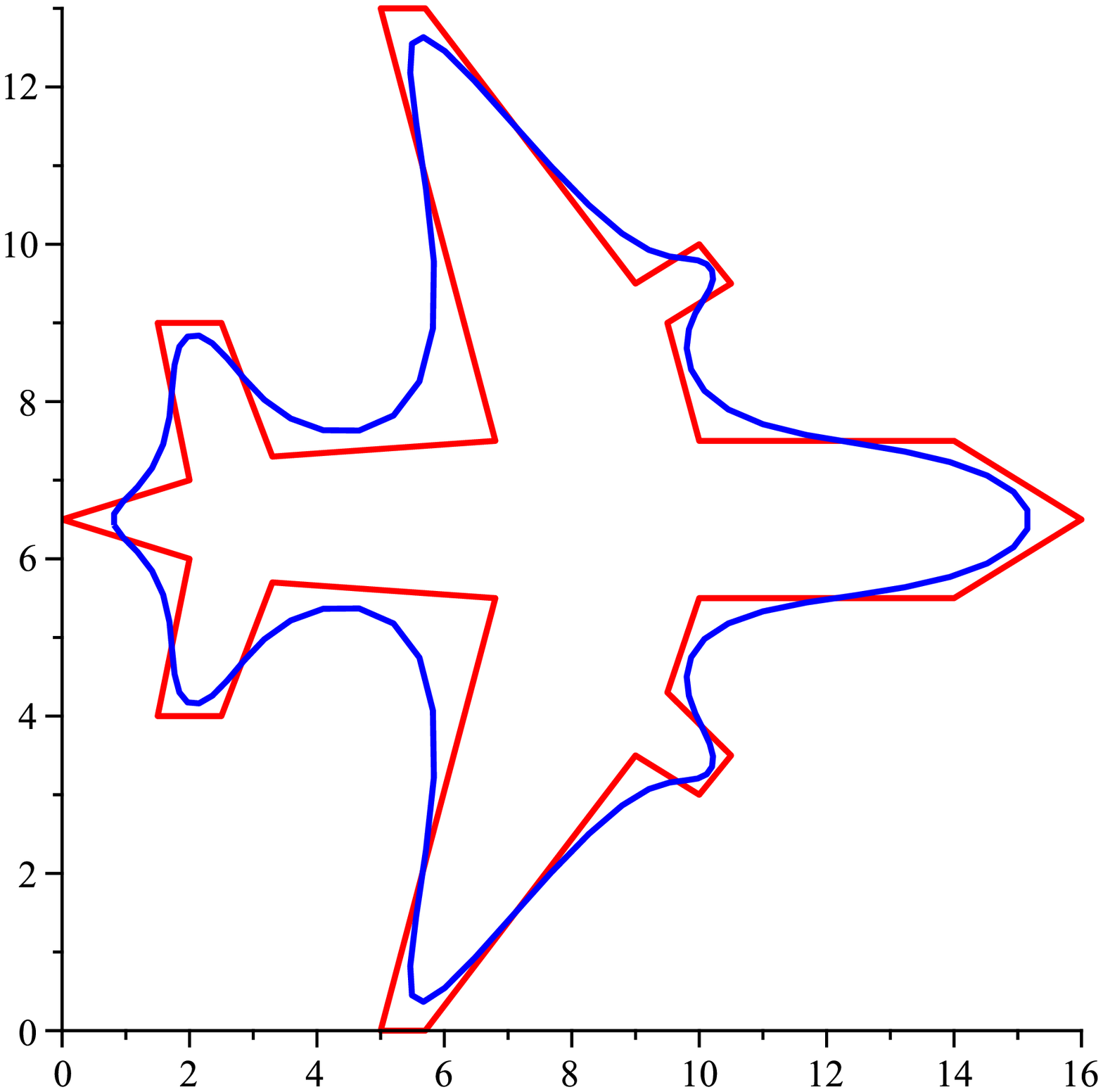, width=1.5 in} & \epsfig{file=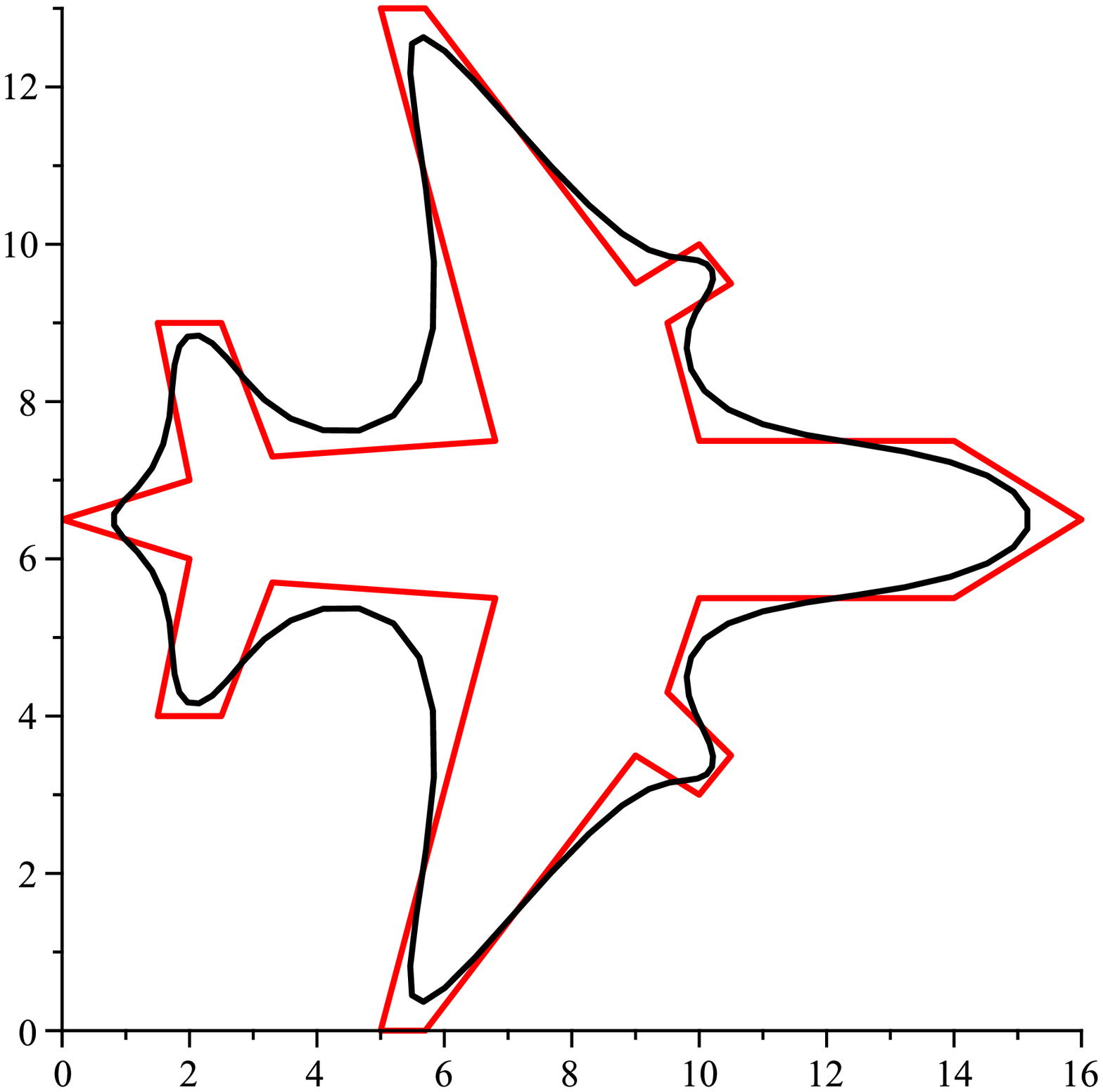, width=1.5 in} & \epsfig{file=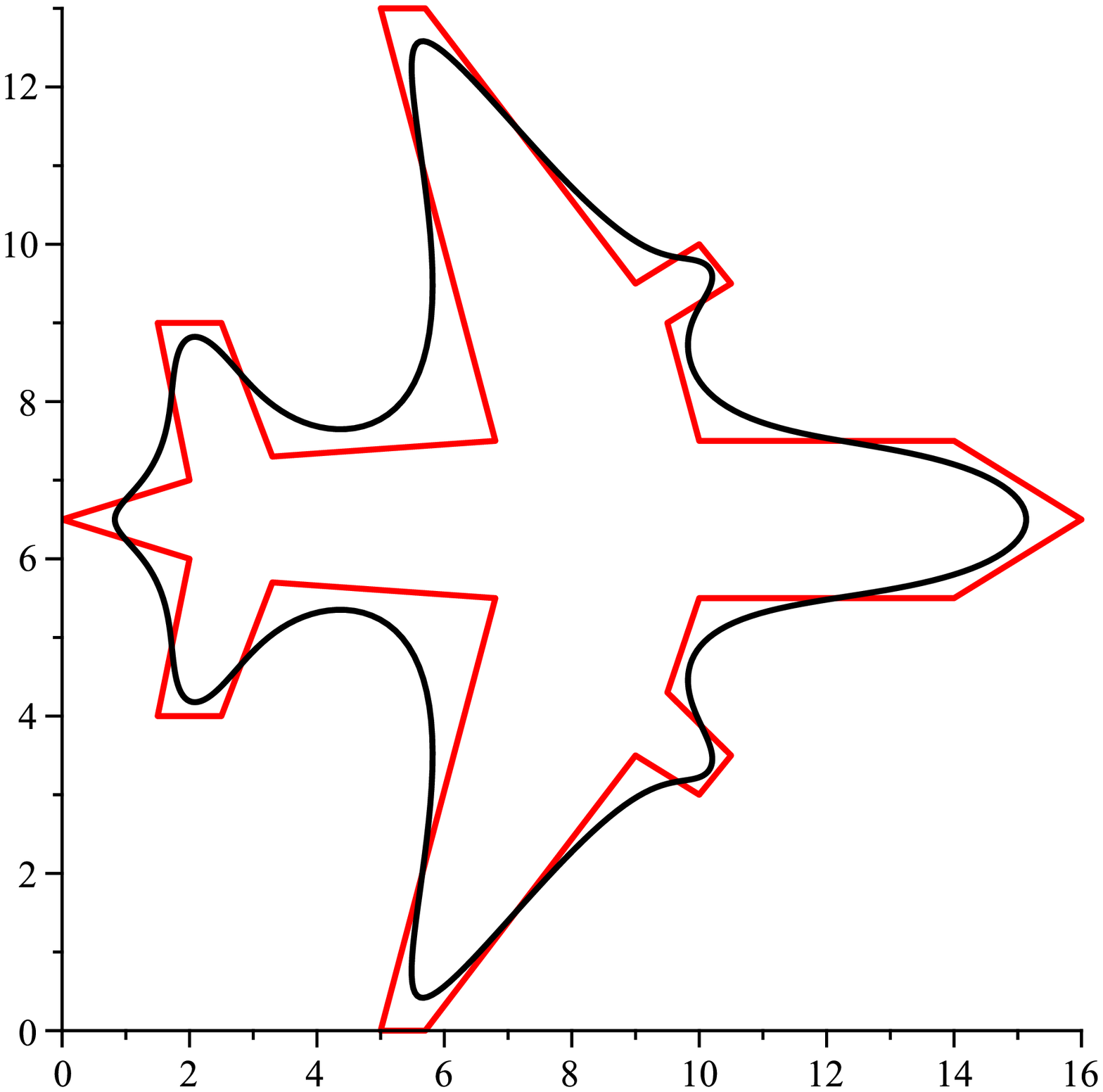, width=1.5 in} \\
(a) One SS by BSS &(b) One SS by QSS & (c) Two SSs by BSS & (d) Two SSs by QSS
\end{tabular}
\caption{\label{w1}\emph{Curves generated by the binary and quaternary subdivision schemes (\ref{constaa1}) and (\ref{constaa2}) respectively.}}
\end{figure}
The graphical inspection and comparison of the binary subdivision scheme  (\ref{constaa1}) and the quaternary subdivision scheme (\ref{constaa2}) after one and two subdivision steps is given in Figure \ref{w1}. This figure clearly shows that the quaternary subdivision scheme smooths the model more efficiently as compare to the binary subdivision scheme.

\begin{rem}
	In captions of the Figures \ref{w1}-\ref{w7}, SS, BSS and QSS denote the Subdivision Step, Binary Subdivision Scheme and the Quaternary Subdivision Scheme respectively. Moreover, in Figures \ref{w1}-\ref{w7} red solid lines represent the initial polygons, blue solid lines show the  curves fitted by the binary and the quaternary subdivision schemes after one subdivision level, while the black solid lines show the curves fitted by the binary and quaternary subdivision schemes after two subdivision steps.
\end{rem}
\begin{cor}
	This corollary is the application of Theorem \ref{thm-even-2} for $m=2$. The binary subdivision scheme (\ref{1b}) for $m=2$ is:
\begin{eqnarray}\label{constaa}
\left\{\begin{array}{cccc}
{g}_{2\varphi-1}^{k+1}&=&\beta_{-6}\,\ g_{\varphi-4}^{k}+\beta_{-4}\,\ g_{\varphi-3}^{k}+\beta_{-2}\,\ g_{{\varphi-2}}^{k}+\beta_{0}\,\ g_{{\varphi-1}}^{k}+\beta_{2}\,\ g_{{\varphi}}^{k}+\beta_{4}\,\ g_{\varphi+1}^{k}+\beta_{6}\,\ g_{\varphi+2}^{k}\\&&+\beta_{8}\,\ g_{\varphi+3}^{k},\\
{g}_{2\varphi}^{k+1}&=&\beta_{8}\,\ g_{{\varphi-3}}^{k}+\beta_{6}\,\ g_{{\varphi-2}}^{k}+\beta_{4}\,\ g_{{\varphi-1}}^{k}+\beta_{2}\,\ g_{{\varphi}}^{k}+\beta_{0}\,\ g_{\varphi+1}^{k}+\beta_{-2}\,\ g_{\varphi+2}^{k}+\beta_{-4}\,\ g_{\varphi+3}^{k}\\&&+\beta_{-6}\,\ g_{\varphi+4}^{k}.
\end{array}\right.
\end{eqnarray}
(\ref{constaa}) is the general form of $8$-point binary subdivision scheme, the coefficients $\beta_{-6},\beta_{-4},\ldots,\beta_{6}$, $\beta_{8}$ of which can be get by any of the  known $8$-point binary subdivision scheme. Therefore, in order to get coefficients we compare the scheme (\ref{constaa}) with the $8$-point scheme presented by \cite{Siddiqi2}, so we get
\begin{eqnarray}\label{a43}
\left\{\begin{array}{cccc}
\beta_{-6}&=&{\frac {1}{82575360}}, \,\ \,\ \beta_{-4}={\frac {26039}{27525120}}, \,\ \,\ \beta_{-2}={\frac {1385999}{27525120}}, \,\ \,\ \beta_{0}={\frac {26672209}{82575360}}, \\ \\ \beta_{2}&=&{\frac {4210971}{9175040}},\,\ \,\ \,\ \,\ \beta_{4}={\frac {1440007}{9175040}}, \,\ \,\ \,\ \beta_{6}={\frac {806047}{82575360}}, \,\ \,\ \,\ \beta_{8}={\frac {243}{9175040}}.
\end{array}\right.
\end{eqnarray}
If we put $m=2$ in (\ref{1q}), we get the $12$-points relaxed quaternary subdivision scheme whose mask is:
\begin{eqnarray*}
\mu_{8}&=&\{\beta_{-6}^{2}, \,\ \beta_{8}\beta_{-6}, \,\ \beta_{-4}\beta_{-6}+\beta_{-6}\beta_{8},\,\  \beta_{8}^{2}+\beta_{6}\beta_{-6}, \,\ \beta_{-4}\beta_{8}+\beta_{-2}\beta_{-6}+\beta_{-6}\beta_{-4},\,\ \beta_{8}\beta_{-4}+\\&&\beta_{4}\beta_{-6}+\beta_{6}\beta_{8},\,\
\beta_{0}\beta_{-6}+\beta_{-4}^{2}+\beta_{-2}\beta_{8}+\beta_{-6}\beta_{6},\,\ \beta_{8}\beta_{6}+\beta_{4}\beta_{8}+
\beta_{6}\beta_{-4}+\beta_{2}\beta_{-6}, \,\ \beta_{0}\beta_{8}\\&&+\beta_{-4}\beta_{6}+\beta_{2}\beta_{-6}+
\beta_{-2}\beta_{-4}+\beta_{-6}\beta_{-2}, \,\ \beta_{8}\beta_{-2}+\beta_{4}\beta_{-4}+\beta_{0}\beta_{-6}+\beta_{6}^{2}+\beta_{2}\beta_{8},\,\ \beta_{4}\beta_{-6} \\&&+
\beta_{0}\beta_{-4}+\beta_{-4}\beta_{-2}+\beta_{2}\beta_{8}+\beta_{-2}\beta_{6}+\beta_{-6}\beta_{4},\,\  \beta_{8}\beta_{4}+\beta_{4}\beta_{6}+\beta_{0}\beta_{8}+
\beta_{6}\beta_{-2}+\beta_{2}\beta_{-4}+\\&&\beta_{-2}\beta_{-6}, \,\ \beta_{4}\beta_{8}+\beta_{0}\beta_{6}+\beta_{-4}\beta_{4}+\beta_{6}\beta_{-6}+\beta_{2}\beta_{-4}
+\beta_{-2}^{2}+\beta_{-6}\beta_{0}, \,\ \beta_{8}\beta_{0}+\beta_{4}\beta_{-2}+\beta_{0}\\&&\times\beta_{-4}+\beta_{-4}\beta_{-6}+\beta_{6}\beta_{4}+
\beta_{2}\beta_{6}+\beta_{-2}\beta_{8}, \,\ \beta_{8}\beta_{-6}+\beta_{4}\beta_{-4}+\beta_{0}\beta_{-2}+\beta_{-4}\beta_{0}
+\beta_{6}\beta_{8}+\\&&\beta_{2}\beta_{6}+\beta_{-2}\beta_{4}+\beta_{-6}\beta_{2}, \,\ \beta_{8}\beta_{2}+\beta_{4}^{2}+
\beta_{0}\beta_{6}+\beta_{-4}\beta_{8}+\beta_{6}\beta_{0}+\beta_{2}\beta_{-2}+ \beta_{-2}\beta_{-4}+\beta_{-6}^{2},\\&& \beta_{8}^{2}+\beta_{4}\beta_{6}+\beta_{0}\beta_{4}+\beta_{-4}\beta_{2}
+\beta_{6}\beta_{-4}+\beta_{2}\beta_{-2}+\beta_{-2}\beta_{0}+\beta_{-6}\beta_{2}, \,\ \beta_{8}\beta_{2}+\beta_{4}\beta_{0}+
\beta_{0}\beta_{-2}\\&&+\beta_{-4}^{2}+\beta_{6}\beta_{2}+\beta_{2}\beta_{4}+\beta_{-2}\beta_{6}+\beta_{-6}\beta_{8},
 \beta_{8}\beta_{-4}+\beta_{4}\beta_{-2}+\beta_{0}^{2}+\beta_{-4}\beta_{2}+
\beta_{6}^{2}+\beta_{2}\beta_{4}+\\&&\beta_{-2}\beta_{2}+\beta_{-6}\beta_{0}, \,\ \beta_{8}\beta_{0}+\beta_{4}\beta_{2}+\beta_{0}\beta_{4}+\beta_{-4}\beta_{6}+\beta_{6}\beta_{2}+
\beta_{2}\beta_{0}+\beta_{-2}^{2}+\beta_{-6} \beta_{-4},\,\ \beta_{8} \beta_{6}\\&&+\beta_{4}^{2}+\beta_{0}\beta_{2}+\beta_{-4}\beta_{0}+\beta_{6}\beta_{-2}+\beta_{2}\beta_{0}+\beta_{-2}
\beta_{2}+\beta_{-6}\beta_{4}, \,\ \beta_{8}\beta_{4}+\beta_{4}\beta_{2}+\beta_{0}^{2}+\beta_{-4}\beta_{-2}\\&&+\beta_{6}\beta_{0}+
\beta_{2}^{2}+\beta_{-2}\beta_{4}+\beta_{-6}\beta_{6},\,\ \beta_{8}\beta_{-2}+\beta_{4}\beta_{0}+
\beta_{0}\beta_{2}+\beta_{-4}\beta_{4}+\beta_{6}\beta_{4}+\beta_{2}^{2}+\beta_{-2}\beta_{0}
+\\&&\beta_{-6}\beta_{-2}, \,\  \beta_{8}\beta_{-2}+\beta_{4}\beta_{0}+
\beta_{0}\beta_{2}+\beta_{-4}\beta_{4}+\beta_{6}\beta_{4}+\beta_{2}^{2}+\beta_{-2}\beta_{0}
+ \beta_{-6}\beta_{-2}, \,\ \beta_{8}\beta_{4}+\beta_{4}\beta_{2}
\end{eqnarray*}
\begin{eqnarray*}
&&+\beta_{0}^{2}+\beta_{-4}\beta_{-2}+\beta_{6}\beta_{0}+
\beta_{2}^{2}+\beta_{-2}\beta_{4}+\beta_{-6}\beta_{6}, \,\ \beta_{8} \beta_{6}+\beta_{4}^{2}+\beta_{0}\beta_{2}+\beta_{-4}\beta_{0}+\beta_{6}\beta_{-2}+\\&&\beta_{2}\beta_{0}+\beta_{-2}
\beta_{2}+\beta_{-6}\beta_{4} , \,\ \beta_{8}\beta_{0}+\beta_{4}\beta_{2}+\beta_{0}\beta_{4}+\beta_{-4}\beta_{6}+\beta_{6}\beta_{2}+
\beta_{2}\beta_{0}+\beta_{-2}^{2}+\beta_{-6} \beta_{-4},\\&& \beta_{8}\beta_{-4}+\beta_{4}\beta_{-2}+\beta_{0}^{2}+\beta_{-4}\beta_{2}+
\beta_{6}^{2}+\beta_{2}\beta_{4}+\beta_{-2}\beta_{2}+\beta_{-6}\beta_{0}, \,\ \beta_{8}\beta_{2}+\beta_{4}\beta_{0}+
\beta_{0}\beta_{-2}+\\&&\beta_{-4}^{2}+\beta_{6}\beta_{2}+\beta_{2}\beta_{4}+\beta_{-2}\beta_{6}+
\beta_{-6}\beta_{8},\,\ \beta_{8}^{2}+\beta_{4}\beta_{6}+\beta_{0}\beta_{4}+\beta_{-4}\beta_{2}
+\beta_{6}\beta_{-4}+\beta_{2}\beta_{-2}+\\&& \beta_{-2}\beta_{0}+\beta_{-6}\beta_{2},\,\ \beta_{8}\beta_{2}+\beta_{4}^{2}+
\beta_{0}\beta_{6}+\beta_{-4}\beta_{8}+\beta_{6}\beta_{0}+\beta_{2}\beta_{-2}+ \beta_{-2}\beta_{-4}+\beta_{-6}^{2},\,\ \beta_{8}\beta_{-6}\\&&+\beta_{4}\beta_{-4}+\beta_{0}\beta_{-2}+\beta_{-4}\beta_{0}
+\beta_{6}\beta_{8}+\beta_{2}\beta_{6}+\beta_{-2}\beta_{4}+\beta_{-6}\beta_{2}, \,\ \beta_{8}\beta_{0}+\beta_{4}\beta_{-2}+\beta_{0}\beta_{-4}+\\&&\beta_{-4}\beta_{-6}+\beta_{6}\beta_{4}+
\beta_{2}\beta_{6}+\beta_{-2}\beta_{8}, \,\ \beta_{4}\beta_{8}+\beta_{0}\beta_{6}+\beta_{-4}\beta_{4}+\beta_{6}\beta_{-6}+\beta_{2}\beta_{-4}
+\beta_{-2}^{2}+\beta_{-6}\\&&\times \beta_{0},\,\ \beta_{8}\beta_{4}+\beta_{4}\beta_{6}+\beta_{0}\beta_{8}+
\beta_{6}\beta_{-2}+\beta_{2}\beta_{-4}+\beta_{-2}\beta_{-6},\,\ \beta_{4}\beta_{-6}+
\beta_{0}\beta_{-4}+\beta_{-4}\beta_{-2}+\beta_{2}\\&&\times\beta_{8}+\beta_{-2}\beta_{6}+\beta_{-6}\beta_{4}, \,\ \beta_{8}\beta_{-2}+\beta_{4}\beta_{-4}+\beta_{0}\beta_{-6}+\beta_{6}^{2}+\beta_{2}\beta_{8},\,\ \beta_{0}\beta_{8}+\beta_{-4}\beta_{6}+\beta_{2}\beta_{-6}\\&&+
\beta_{-2}\beta_{-4}+\beta_{-6}\beta_{-2},\,\ \beta_{8}\beta_{6}+\beta_{4}\beta_{8}+
\beta_{6}\beta_{-4}+\beta_{2}\beta_{-6},\,\ \beta_{0}\beta_{-6}+
\beta_{-4}^{2}+\beta_{-2}\beta_{8}+\beta_{-6}\beta_{6}, \\&& \beta_{8}\beta_{-4}+\beta_{4}\beta_{-6}+\beta_{6}\beta_{8},\,\ \beta_{-4}\beta_{8}+\beta_{-2}\beta_{-6}+\beta_{-6}\beta_{-4},\,\ \beta_{8}^{2}+\beta_{6}\beta_{-6},\,\ \beta_{-4}\beta_{-6}+\beta_{-6}\beta_{8},\\&& \beta_{8}\beta_{-6},\,\ \beta_{-6}^{2}\}.
\end{eqnarray*}
The mask of the quternary subdivision scheme is the non-linear combination of the mask  $\beta_{-6},\beta_{-4},\ldots,\beta_{6}$, $\beta_{8}$ of the $8$-point binary subdivision scheme. Therefore, by using (\ref{a43}) in the mask $\mu_{8}$, we get the mask of the following quternary subdivision scheme
\begin{eqnarray}\label{const32}
\left\{\begin{array}{ccccc}
{g}_{4\varphi-2}^{k+1}&=&\check{A}_{1}g_{\varphi-5}^{k}+\check{A}_{2}g_{\varphi-4}^{k}+\check{A}_{3}g_{\varphi-3}^{k}+\check{A}_{4}
g_{\varphi-2}^{k}+\check{A}_{5}g_{\varphi-1}^{k}+\check{A}_{6}g_{\varphi}^{k}+\check{A}_{7}g_{\varphi+1}^{k}+\check{A}_{8}
g_{\varphi+2}^{k}\\&&+\check{A}_{9}g_{\varphi+3}^{k}+\check{A}_{10}g_{\varphi+4}^{k}+\check{A}_{11}g_{\varphi+5}^{k},\\ \\
{g}_{4\varphi-1}^{k+1}&=&\check{A}_{11}g_{\varphi-5}^{k}+\check{A}_{10}g_{\varphi-4}^{k}+\check{A}_{9}g_{\varphi-3}^{k}
+\check{A}_{8}
g_{\varphi-2}^{k}+\check{A}_{7}g_{\varphi-1}^{k}+\check{A}_{6}g_{\varphi}^{k}+\check{A}_{5}g_{\varphi+1}^{k}+\check{A}_{4}
g_{\varphi+2}^{k}\\&&+\check{A}_{3}g_{\varphi+3}^{k}+\check{A}_{2}g_{\varphi+4}^{k}+\check{A}_{1}g_{\varphi+5}^{k},\\ \\
{g}_{4\varphi}^{k+1}&=&\check{B}_{1}g_{\varphi-5}^{k}+\check{B}_{2}g_{\varphi-4}^{k}+\check{B}_{3}g_{\varphi-3}^{k}+\check{B}_{4}
g_{\varphi-2}^{k}+\check{B}_{5}g_{\varphi-1}^{k}+\check{B}_{6}g_{\varphi}^{k}+\check{B}_{7}g_{\varphi+1}^{k}+\check{B}_{8}
g_{\varphi+2}^{k}+\\&&\check{B}_{9}g_{\varphi+3}^{k}+\check{B}_{10}g_{\varphi+4}^{k}+\check{B}_{11}g_{\varphi+5}^{k}+\check{B}_{12}
g_{\varphi+6}^{k},\\ \\
{g}_{4\varphi+1}^{k+1}&=&\check{B}_{12}g_{\varphi-5}^{k}+\check{B}_{11}g_{\varphi-4}^{k}+\check{B}_{10}
g_{\varphi-3}^{k}+\check{B}_{9}
g_{\varphi-2}^{k}+\check{B}_{8}g_{\varphi-1}^{k}+\check{B}_{7}g_{\varphi}^{k}+\check{B}_{6}g_{\varphi+1}^{k}+\check{B}_{5}
g_{\varphi+2}^{k}\\&&+\check{B}_{4}g_{\varphi+3}^{k}+\check{B}_{3}g_{\varphi+4}^{k}+\check{B}_{2}g_{\varphi+5}^{k}+\check{B}_{1}
g_{\varphi+6}^{k},
\end{array}\right.
\end{eqnarray}
where
\begin{eqnarray}\label{a44}
\left\{\begin{array}{ccccc}
\check{A}_{1}&=&{\frac{698627}{852336259891200}}, \,\ \,\ \check{A}_{2}={\frac{3879598117}{284112086630400}}, \,\ \,\ \check{A}_{3}={\frac{350941180003}{142056043315200}}, \,\ \,\
\check{A}_{4}={\frac{36602385889}{676457349120}},\,\ \,\  \\ \\ \check{A}_{5}&=&{\frac{5641800724981}{20293720473600}}, \,\ \,\ \check{A}_{6}={\frac{247891317863}{579820584960}}, \,\ \,\  \check{A}_{7}={\frac{527402518309}{2536715059200}}, \,\ \,\ \check{A}_{8}={\frac{2067049873661}{71028021657600}},\,\ \,\ \\ \\ \check{A}_{9}&=&{\frac{48690122269}{56822417326080}}, \,\ \,\ \check{A}_{10}={\frac{1902910423}{852336259891200}}, \,\ \,\ \check{A}_{11}={\frac{239}{20293720473600}}\,\ \,\
\check{B}_{1}={\frac{27}{84181359001600}}, \\ \\ \check{B}_{2}&=&{\frac{30898837}{108233175859200}}, \,\ \,\ \check{B}_{3}={\frac{1754117761421}{6818690079129600}}, \,\ \,\ \check{B}_{4}={\frac{32344488846199}{2272896693043200}}, \,\ \,\ \check{B}_{5}={\frac{163622535291293}{1136448346521600}}, \\ \\ \check{B}_{6}&=&{\frac{12926815750607}{32469952757760}}, \,\ \,\  \check{B}_{7}={\frac{56015931444329}{162349763788800}}, \,\ \,\ \check{B}_{8}={\frac{104604880235159}{1136448346521600}}, \,\ \,\ \check{B}_{9}={\frac{75466139809}{12025908428800}}, \\ \\  \check{B}_{10}&=&{\frac{148716800989}{2272896693043200}}, \,\ \,\ \check{B}_{11}={\frac{58359331}{2272896693043200}}, \,\ \,\ \check{B}_{12}={\frac{1}{6818690079129600}}.\,\ \,\ \,\ \,\ \,\ \,\ \,\ \,\ \,\ \,\  \,\ \,\ \,\ \,\ \,\ \,\
\end{array}\right.
\end{eqnarray}
\end{cor}
\begin{figure}[h]
\begin{tabular}{cccc}
\epsfig{file=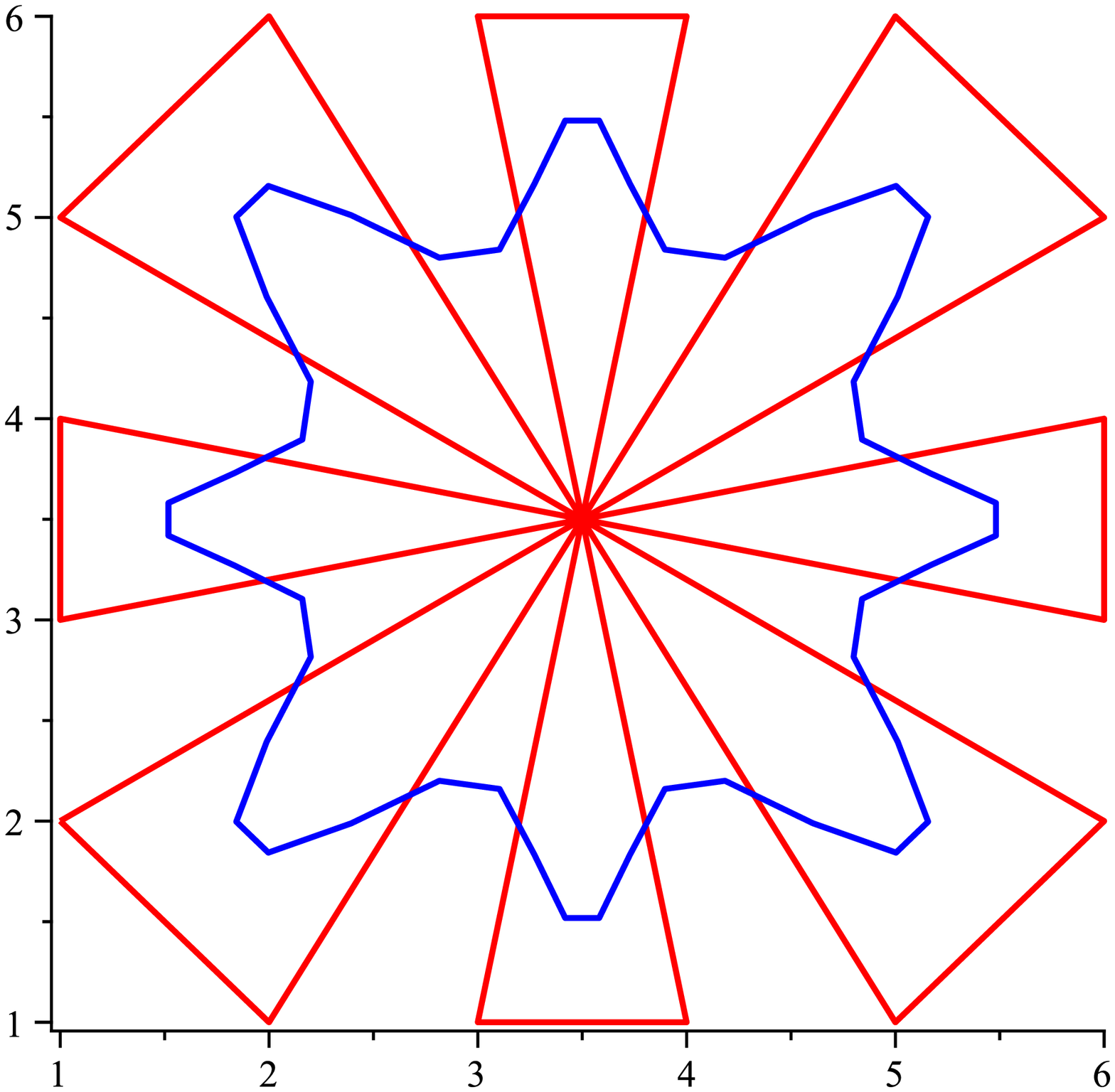, width=1.5 in} & \epsfig{file=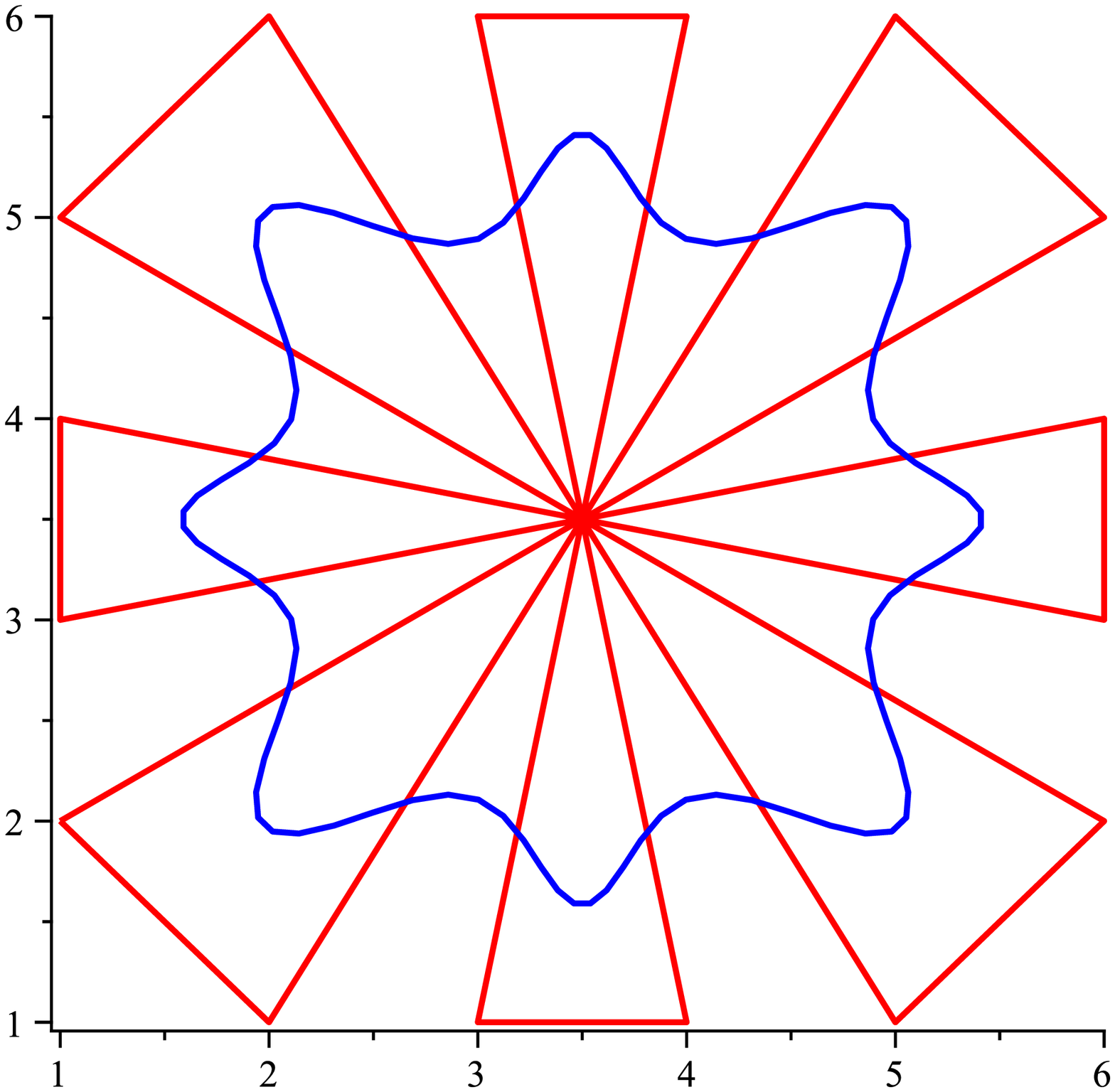, width=1.5 in} & \epsfig{file=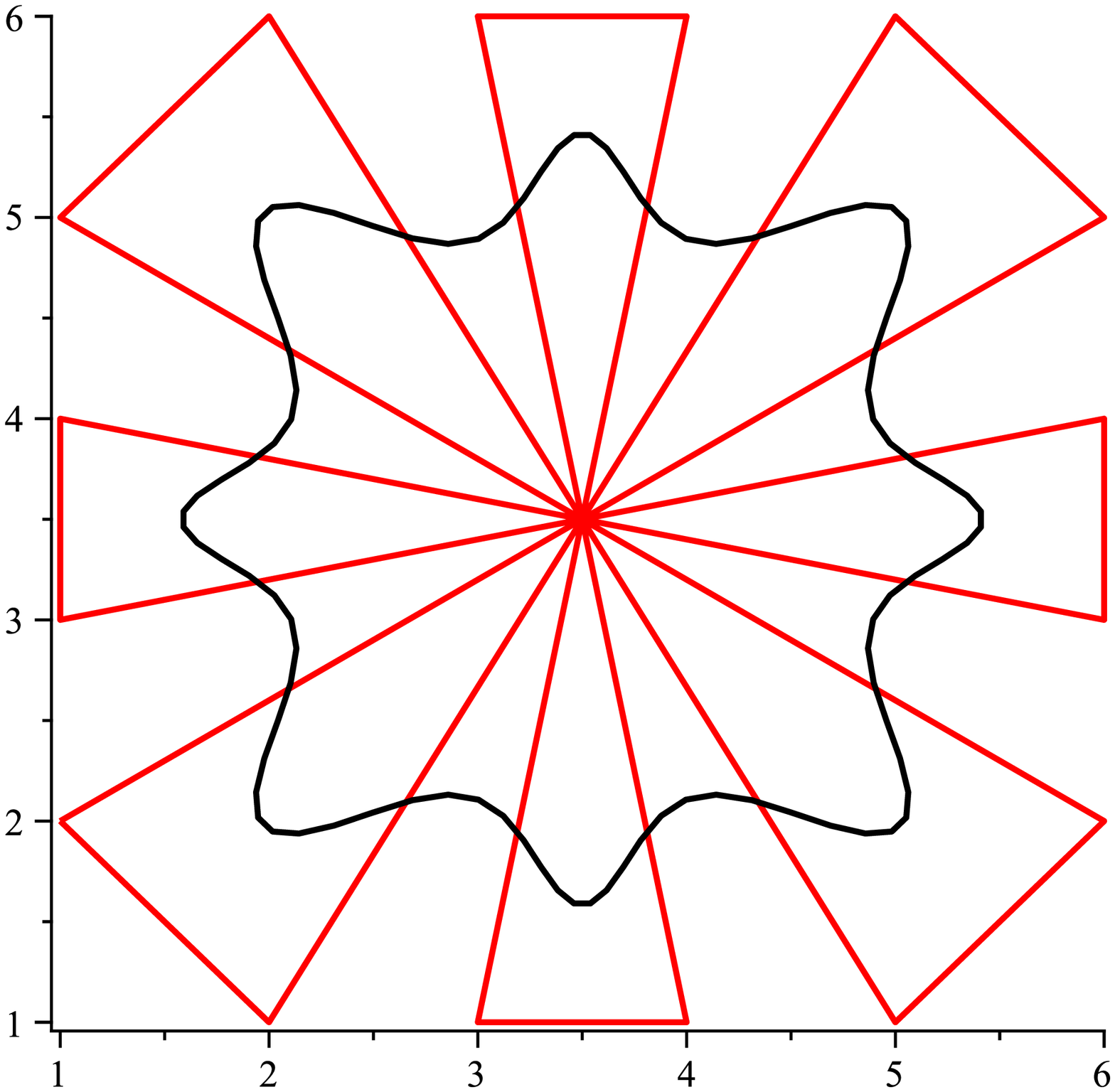, width=1.5 in} & \epsfig{file=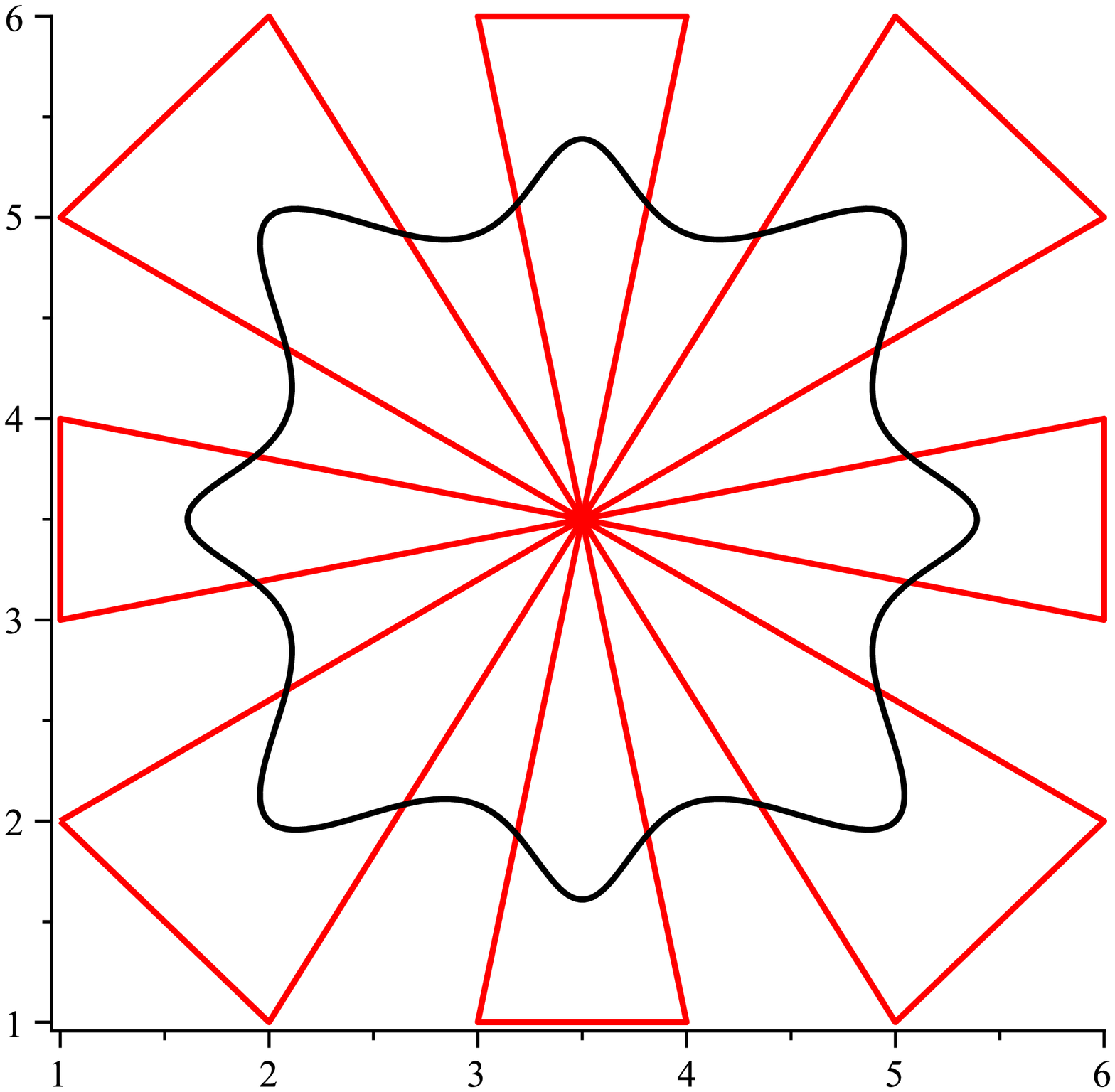, width=1.5 in} \\
(a) One SS by BSS &(b) One SS by QSS & (c) Two SSs by BSS & (d) Two SSs by QSS
\end{tabular}
\caption{\label{w2}\emph{Curves generated by the binary and quaternary subdivision schemes (\ref{constaa}) and (\ref{const32}) respectively.}}
\end{figure}
Figure \ref{w2} shows the models fitted by the binary subdivision scheme  (\ref{constaa}) and the quaternary subdivision scheme (\ref{const32}) after one and two subdivision steps. Again this graphical inspection shows the superiority of the quaternary subdivision scheme over the binary subdivision scheme.
\begin{cor}
	In this corollary, we use the even-points binary subdivision scheme (\ref{1b}) and expand it when $m=3$. As the result we get the following $12$-points binary subdivision scheme:
	\begin{eqnarray}\label{consta32}
		\left\{\begin{array}{ccccc}
			{g}_{2\varphi-1}^{k+1}&=&\beta_{-10}\,\ g_{\varphi-6}^{k}+\beta_{-8}\,\ g_{\varphi-5}^{k}+
			\beta_{-6}\,\ g_{\varphi-4}^{k}+\beta_{-4}\,\ g_{\varphi-3}^{k}+\beta_{-2}\,\ g_{\varphi-2}^{k}+
			\beta_{0}\,\ g_{\varphi-1}^{k}\\&&+\beta_{2}\,\ g_{\varphi}^{k}+ \beta_{4}\,\ g_{\varphi+1}^{k}+\beta_{6}\,\ g_{\varphi+2}^{k}+\beta_{8}\,\ g_{\varphi+3}^{k}+\beta_{10}\,\ g_{\varphi+4}^{k}+\beta_{12}\,\ g_{\varphi+5}^{k},\\
			{g}_{2\varphi}^{k+1}&=&\beta_{12}\,\ g_{\varphi-5}^{k}+\beta_{10}\,\ g_{\varphi-4}^{k}+\beta_{8}\,\ g_{\varphi-3}^{k}+\beta_{6}\,\ g_{\varphi-2}^{k}+\beta_{4}\,\ g_{\varphi-1}^{k}+\beta_{2}\,\ g_{\varphi}^{k}+\beta_{0}\,\ g_{\varphi+1}^{k}\\&&+\beta_{-2}\,\ g_{\varphi+2}^{k}+\beta_{-4}\,\ g_{\varphi+3}^{k}+\beta_{-6}\,\ g_{\varphi+4}^{k}+\beta_{-8}\,\ g_{\varphi+5}^{k}+\beta_{-10}\,\ g_{\varphi+6}^{k}.
		\end{array}\right.
	\end{eqnarray}
	A $12$-points binary subdivision scheme which we get by the  algorithm presented by \cite{Siddiqi2} gives the coefficients  $\beta_{-10}, \beta_{-8}, \ldots,  \beta_{10}$, $\beta_{12}$ of this scheme. Hence the coefficients are:
	\begin{eqnarray}\label{a45}
		\left\{\begin{array}{ccccc}
			\beta_{-10}&=&{\frac {1}{167423193907200}}, \,\ \,\ \beta_{-8}={\frac {48828113}{167423193907200}},  \,\ \,\ \beta_{-6}={\frac {410601629}{2232309252096}},\,\ \,\ \beta_{-4}={\frac {52548530917}{6200859033600}},  \\ \\ \,\ \beta_{-2}&=&{\frac {2471063221063}{27903865651200}}, \,\ \,\ \,\ \beta_{0}={\frac {1680588922139}{5580773130240}}, \,\ \,\  \beta_{2}={\frac {2134020225233}{5580773130240}}, \,\ \,\ \,\ \,\ \,\ \beta_{4}={\frac {69086299223}{372051542016}}, \\ \\  \beta_{6}&=&{\frac {1785493468247}{55807731302400}}, \,\ \,\  \beta_{8}={\frac {261595441397}{167423193907200}}, \,\ \,\ \beta_{10}={\frac {1975200979}{167423193907200}}, \,\ \,\  \beta_{12}={\frac {2187}{2066953011200}}.
		\end{array}\right.
	\end{eqnarray}
	Now first we simplify (\ref{1q}) for $m=3$ and then substitute these 12 values from (\ref{a45}) into (\ref{1q}),  thus we get the 17-points relaxed quaternary approximating subdivision scheme, that is
	\begin{eqnarray}\label{const36}
		\left\{\begin{array}{ccccc}
			{g}_{4\varphi-2}^{k+1}&=&C_{1}g_{\varphi-8}^{k}+C_{2}g_{\varphi-7}^{k}+C_{3} g_{\varphi-6}^{k}+C_{4}g_{\varphi-5}^{k} +C_{5} g_{\varphi-4}^{k}+C_{6} g_{\varphi-3}^{k} +C_{7}g_{\varphi-2}^{k}+C_{8}g_{\varphi-1}^{k}\\&& +C_{9}g_{\varphi}^{k}+C_{10}g_{\varphi+1}^{k} +C_{11}g_{\varphi+2}^{k}+C_{12}g_{\varphi+3}^{k}+C_{13}g_{\varphi+4}^{k}+C_{14}g_{\varphi+5}^{k} +C_{15}g_{\varphi+6}^{k}+C_{16}\\&&  \times g_{\varphi+7}^{k}+C_{17}g_{\varphi+8}^{k},\\
			{g}_{4\varphi-1}^{k+1}&=&C_{17}g_{\varphi-8}^{k}+C_{16}g_{\varphi-7}^{k}+C_{15} g_{\varphi-6}^{k}+C_{14} g_{\varphi-5}^{k}+C_{13} g_{\varphi-4}^{k}+C_{12} g_{\varphi-3}^{k}+C_{11}g_{\varphi-2}^{k}+C_{10}\\&& \times g_{\varphi-1}^{k}+C_{9}g_{\varphi}^{k}+
			C_{8}g_{\varphi+1}^{k}+C_{7}g_{\varphi+2}^{k}+C_{6}g_{\varphi+3}^{k}+C_{5}g_{\varphi+4}^{k}+C_{4}
			g_{\varphi+5}^{k}+C_{3}g_{\varphi+6}^{k}+\\&& C_{2}g_{\varphi+7}^{k}+C_{1}g_{\varphi+8}^{k},\\
			{g}_{4\varphi}^{k+1}&=&D_{1}g_{\varphi-8}^{k}+D_{2}g_{\varphi-7}^{k}+D_{3}g_{\varphi-6}^{k}+D_{4}
			g_{\varphi-5}^{k}+D_{5}g_{\varphi-4}^{k}+D_{6}g_{\varphi-3}^{k}+D_{7}g_{\varphi-2}^{k}+D_{8}
			g_{\varphi-1}^{k}\\&&+D_{9}g_{\varphi}^{k}+D_{10}g_{\varphi+1}^{k}+D_{11}g_{\varphi+2}^{k}+D_{12}
			g_{\varphi+3}^{k}+D_{13}g_{\varphi+4}^{k}+D_{14}g_{\varphi+5}^{k}+D_{15}g_{\varphi+6}^{k}+D_{16}\\&& \times g_{\varphi+7}^{k}+D_{17}g_{\varphi+8}^{k}+D_{18}g_{\varphi+9}^{k},\\
			{g}_{4\varphi+1}^{k+1}&=&D_{18}g_{\varphi-8}^{k}+D_{17}g_{\varphi-7}^{k}+D_{16}g_{\varphi-6}^{k}+D_{15}
			g_{\varphi-5}^{k}+D_{14}g_{\varphi-4}^{k}+D_{13}g_{\varphi-3}^{k}+D_{12}g_{\varphi-2}^{k}+\\&& D_{11}g_{\varphi-1}^{k}+D_{10}g_{\varphi}^{k}+D_{9}g_{\varphi+1}^{k}+D_{8}g_{\varphi+2}^{k}+
			D_{7}g_{\varphi+3}^{k}+D_{6}g_{\varphi+4}^{k}+D_{5}g_{\varphi+5}^{k} +D_{4}g_{\varphi+6}^{k}\\&&+D_{3}g_{\varphi+7}^{k}+D_{2}g_{\varphi+8}^{k}+D_{1}g_{\varphi+9}^{k},
		\end{array}\right.
	\end{eqnarray}
	where
	\begin{eqnarray}\label{a66-1}
		\left\{\begin{array}{ccccccc}
			C_{1}&=&{\frac{170185003}{143012887031060669399040000}}, \,\ \,\  C_{2}={\frac{11618623511827}{2275205020948692467712000}}, \\ \\  C_{3}&=&{\frac{2639470801758827761}{87595393306524660006912000}} , \,\ \,\ C_{4}={\frac{268509519026918532793}{25027255230435617144832000}}, \\ \\
			C_{5}&=&{\frac{16029066874929664797821}{23358771548406576001843200}}, \,\ \,\ C_{6}={\frac{11208054761546257498631183}{875953933065246600069120000}}, \\ \\ C_{7}&=&{\frac{32207048328244229783317}{361964435150928347136000}}, \,\ \,\ C_{8}={\frac{4187024370524574208415201}{15926435146640847273984000}},\\ \\
			C_{9}&=&{\frac{678810440255672301860419}{1930476987471617851392000}}, \,\ \,\  C_{10}={\frac{689073490261723216665251}{3185287029328169454796800}},\\ \\
			C_{11}&=&{\frac{338960562772283782004933}{5688012552371731169280000}}, \,\ \,\ C_{12}={\frac{1191634131193759734219757}{175190786613049320013824000}}, \\ \\
			C_{13}&=&{\frac{8792959297188237328469}{31852870293281694547968000}},\,\ \,\  C_{14}={\frac{19346249685396093203}{6488547652335160000512000}}, \\ \\
			C_{15}&=&{\frac{1048707451741153}{218988483266311650017280}}, \,\ \,\
			C_{16}={\frac{245027153519101}{875953933065246600069120000}},\\ \\
			C_{17}&=&{\frac{2450263}{1401526292904394560110592000}}, \,\ \,\ D_{1}={\frac{27}{4272294750508747325440000}}, 
					\end{array}\right.
		\end{eqnarray}
		\begin{eqnarray}\label{a66-2}
			\left\{\begin{array}{ccccccc}
			D_{2}&=&{\frac{358812277001921}{28030525858087891202211840000}},\,\ \,\    D_{3}={\frac{2203625183357673913}{3503815732260986400276480000}}, \\ \\ D_{4}&=&{\frac{103411470339072618307}{140152629290439456011059200}}, \,\ \,\ D_{5}={\frac{637615639085108558009}{6229005746241753600491520}}, \\ \\ D_{6}&=&{\frac{23767610644231508920682471}{7007631464521972800552960000}},\,\ \,\
			D_{7}={\frac{132394698240692877731080079}{3503815732260986400276480000}}, \\ \\ D_{8}&=&{\frac{7705315847678086375114273}{45504100418973849354240000}},\,\ \,\
			D_{9}={\frac{17075025982837018547243819}{50964592469250711276748800}},\\ \\  D_{10}&=&{\frac{1033280151597605241630239}{3397639497950047418449920}}, \,\ \,\
			D_{11}={\frac{40126522029222497908818157}{318528702932816945479680000}},\\ \\   D_{12}&=&{\frac{79334867894972096989478381}{3503815732260986400276480000}},\,\ \,\   D_{13}={\frac{11078992716719243713520413}{7007631464521972800552960000}}, \\ \\
			D_{14}&=&{\frac{3248016674151168146933}{93435086193626304007372800}}, \,\ \,\ D_{15}={\frac{4496423852919659533}{28030525858087891202211840}}, \\ \\
			D_{16}&=&{\frac{231433236694503691}{3503815732260986400276480000}},\,\  \,\ D_{17}={\frac{8680597683899}{28030525858087891202211840000}}, \\ \\
			D_{18}&=&{\frac{1}{28030525858087891202211840000}}.
		\end{array}\right.
	\end{eqnarray}
\end{cor}
\begin{figure}[h]
	\begin{tabular}{cccc}
		\epsfig{file=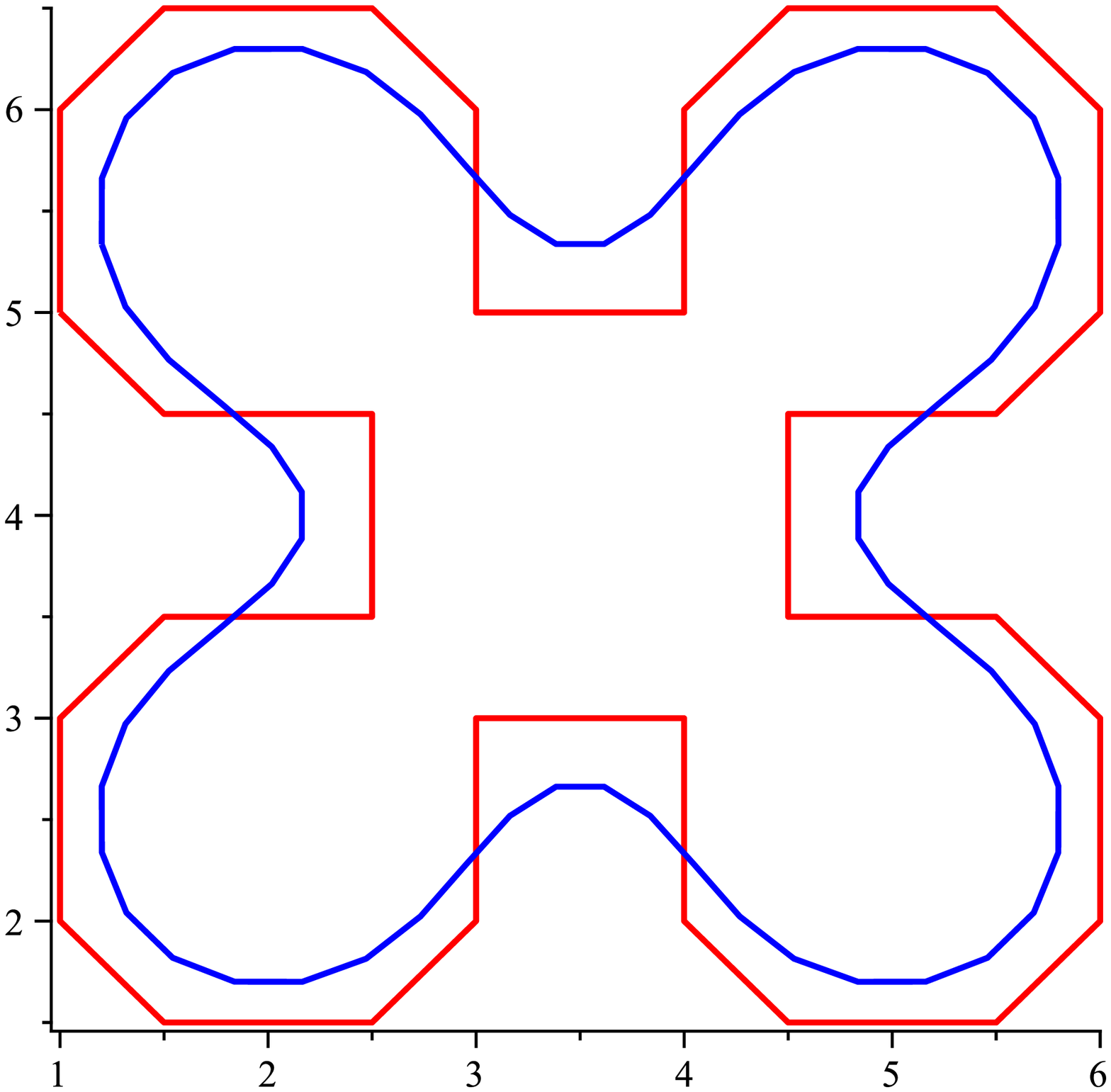, width=1.5 in} & \epsfig{file=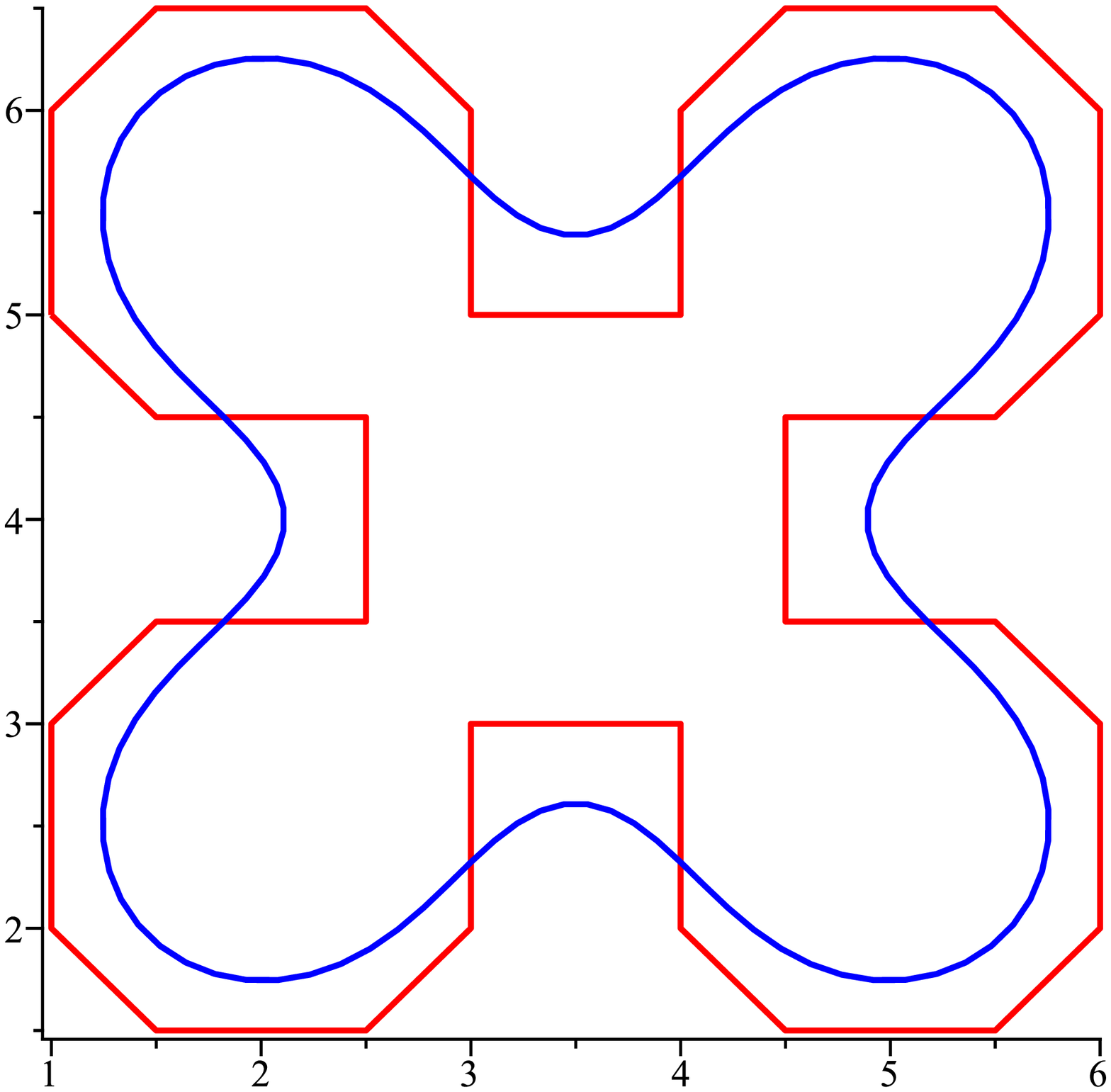, width=1.5 in} & \epsfig{file=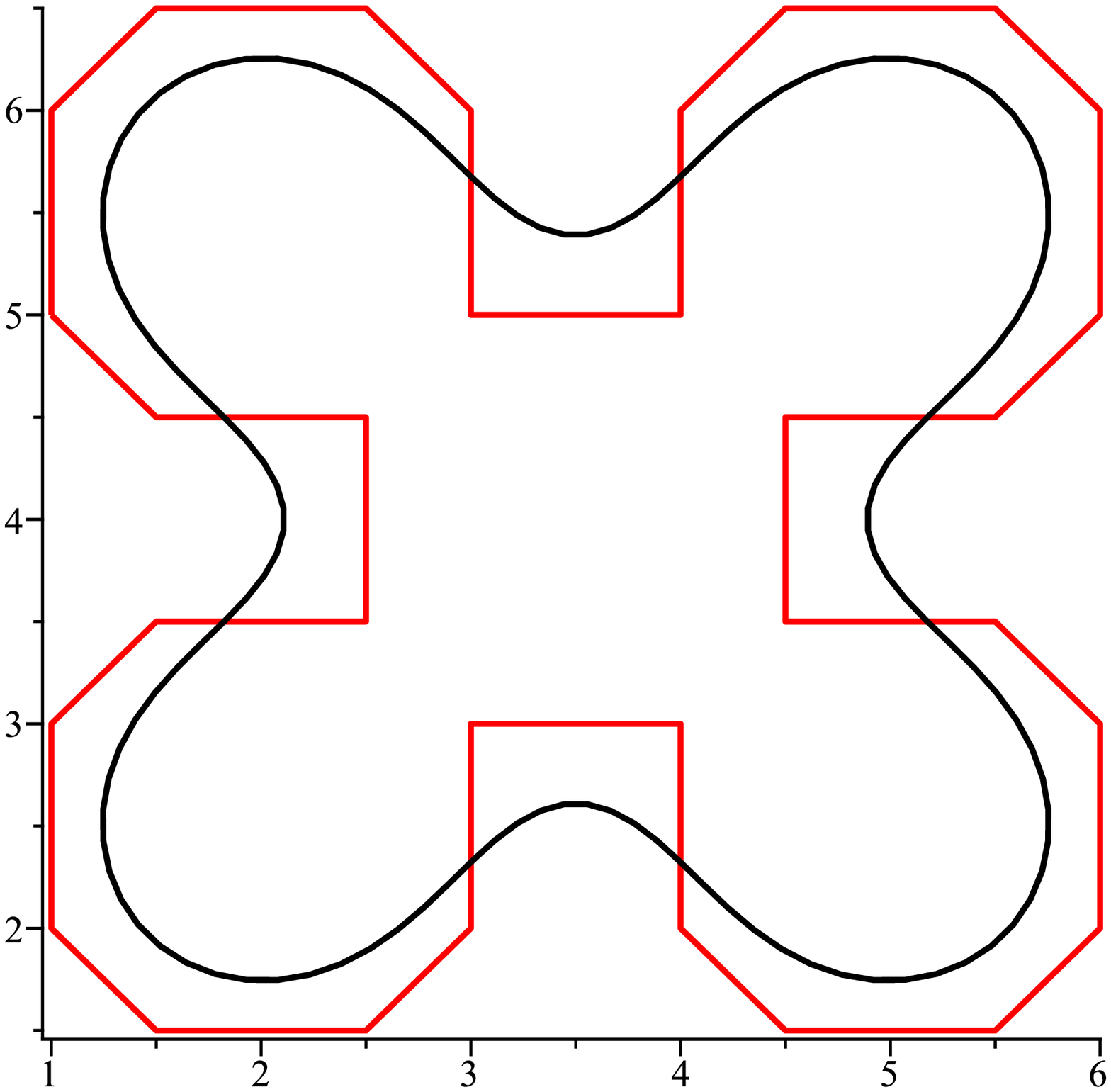, width=1.5 in} & \epsfig{file=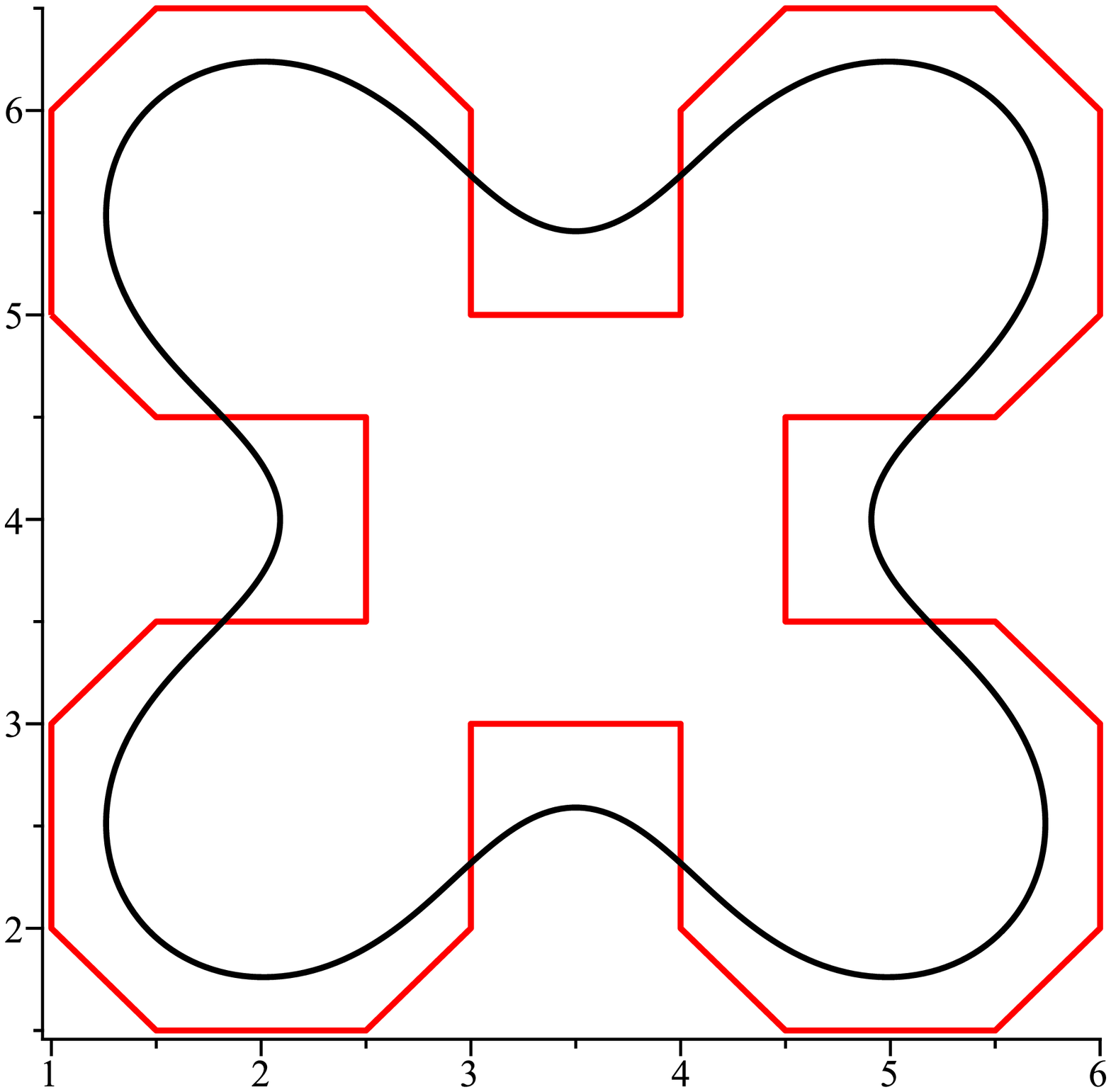, width=1.5 in} \\
		(a) One SS by BSS &(b) One SS by QSS & (c) Two SSs by BSS & (d)  Two SSs by QSS
	\end{tabular}
	\caption{\label{w3}\emph{Curves generated by the binary and quaternary subdivision schemes (\ref{consta32}) and (\ref{const36}) respectively.}}
\end{figure}
Figure \ref{w3} illustrates  that the $17$-points relaxed quaternary subdivision scheme (\ref{const36}) uses only one subdivision iterations to smooth the model while the $12$-points binary subdivision scheme (\ref{consta32}) uses two subdivision iterations to acheives that level of smoothness.

Now we present the applications of the Theorem \ref{thm-odd-2} in the next few results.
\begin{cor}
	Here we use the result which we get from Theorem \ref{thm-odd-2} for $m=0$, Thererfore we expand the $(4m+2)$-point binary subdivision scheme (\ref{const1a}) for $m=0$. Hence we get
	\begin{eqnarray}\label{consta40}
		\left\{\begin{array}{ccccc}
			{g}_{2\varphi-1}^{k+1}&=&\beta_{0}\,\ g_{{\varphi-1}}^{k}+\beta_{2}\,\ g_{{\varphi}}^{k},\\ \\
			{g}_{2\varphi}^{k+1}&=&\beta_{2}\,\ g_{{\varphi}}^{k}+\beta_{0}\,\ g_{\varphi+1}^{k}.
		\end{array}\right.
	\end{eqnarray}
	For the values of $\beta_{0}$ and $\beta_{2}$, we compare the $2$-points binary subdivision scheme (\ref{consta40}) with the well-known Chaikin's subdivision scheme, so we get:
	\begin{eqnarray}\label{beta-chaikin}
		\beta_{0}={\frac {1}{4}}, \,\ \,\ \beta_{2}={\frac {3}{4}}.
	\end{eqnarray}
We get the $2$-point relaxed quaternary subdivision scheme by putting $m=0$ in (\ref{const75a}) whose mask elements attained by the mask (\ref{beta-chaikin}) of  the scheme  (\ref{consta40}). The scheme is:
	\begin{eqnarray}\label{consta41}
		\left\{\begin{array}{ccccc}
			{g}_{4\varphi-2}^{k+1}&=&\hat{a}_{1}g_{\varphi-1}^{k}+\hat{a}_{2}g_{\varphi}^{k}+\hat{a}_{3}g_{\varphi+1}^{k},\\ \\
			{g}_{4\varphi-1}^{k+1}&=&\hat{a}_{3}g_{\varphi-1}^{k}+\hat{a}_{2}g_{\varphi}^{k}+\hat{a}_{1} g_{\varphi+1}^{k},\\ \\
			{g}_{4\varphi}^{k+1}&=&\hat{b}_{1}g_{\varphi}^{k}+\hat{b}_{2}g_{\varphi+1}^{k},\\ \\
			{g}_{4\varphi+1}^{k+1}&=&\hat{b}_{2}g_{\varphi}^{k}+\hat{b}_{1}g_{\varphi+1}^{k},
		\end{array}\right.
	\end{eqnarray}
	where
	\begin{eqnarray*}
		\hat{a}_{1}&=&{\frac{3}{16}}, \,\ \,\ \hat{a}_{2}={\frac{3}{4}}, \,\ \,\ \hat{a}_{3}={\frac{1}{16}}\,\ \,\
		\hat{b}_{1}={\frac{5}{8}}, \,\ \,\ \hat{b}_{2}={\frac{3}{8}}.
	\end{eqnarray*}
\end{cor}
\begin{figure}[h]
	\begin{tabular}{cccc}
		\epsfig{file=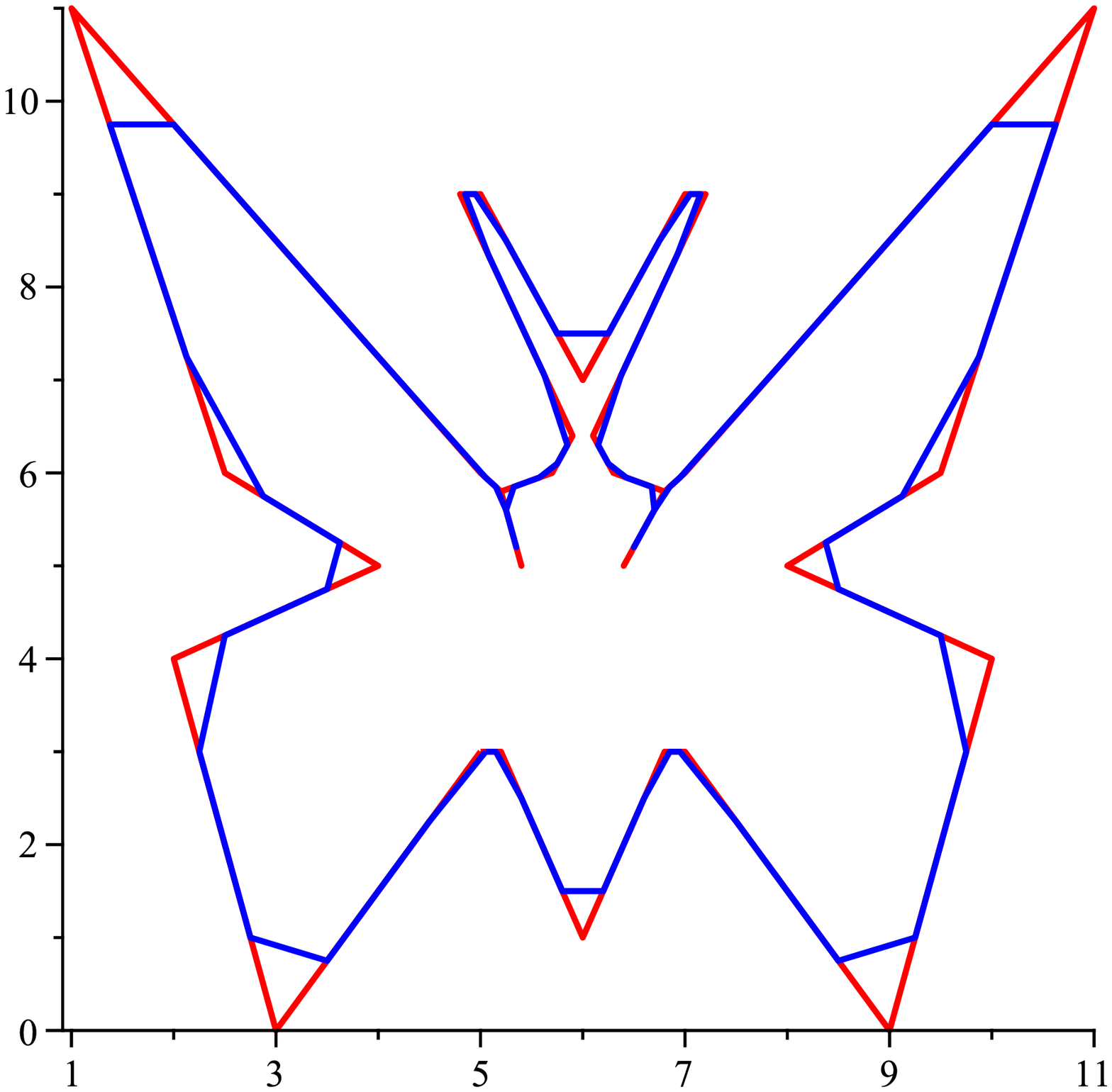, width=1.5 in} & \epsfig{file=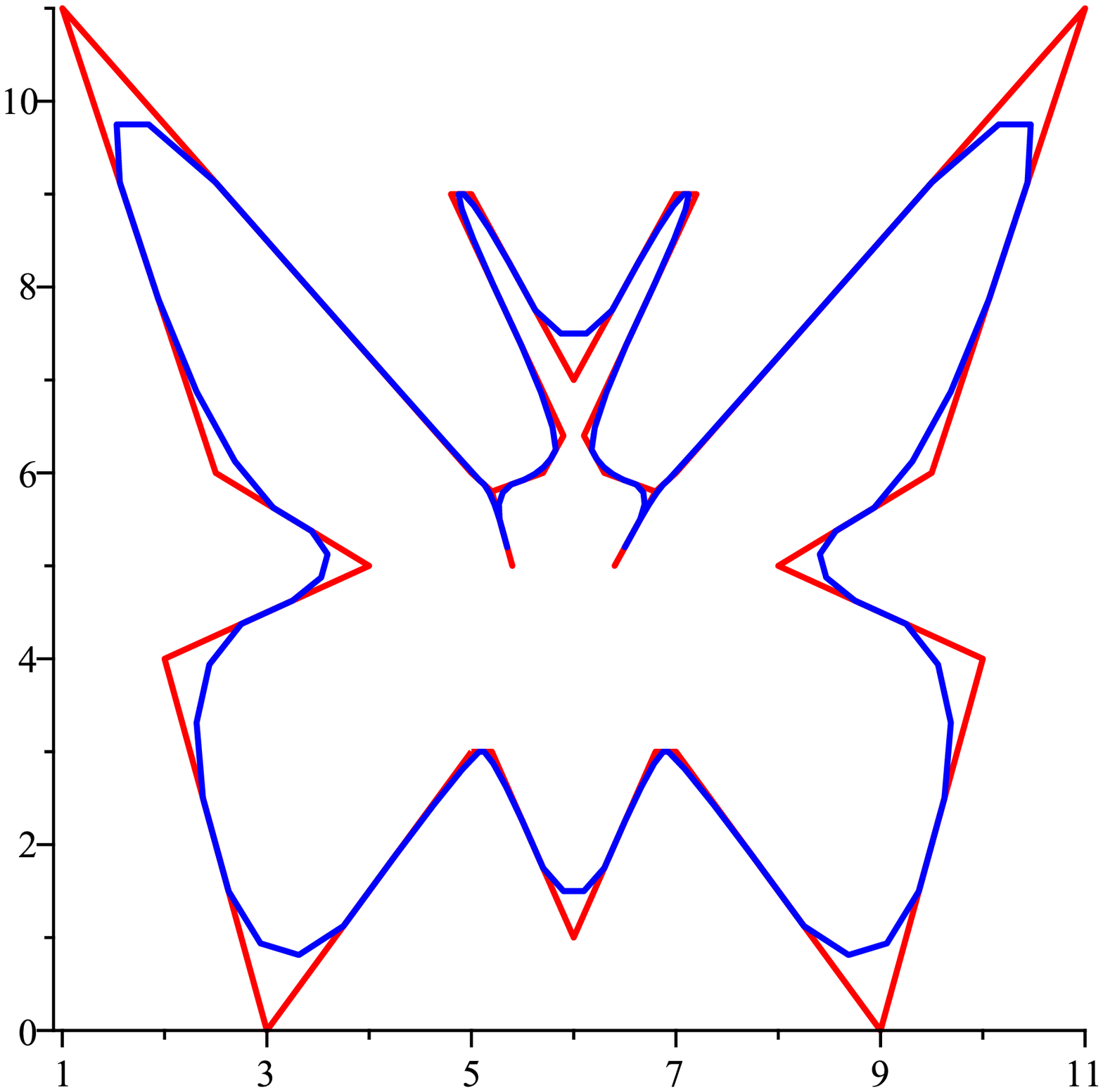, width=1.5 in} & \epsfig{file=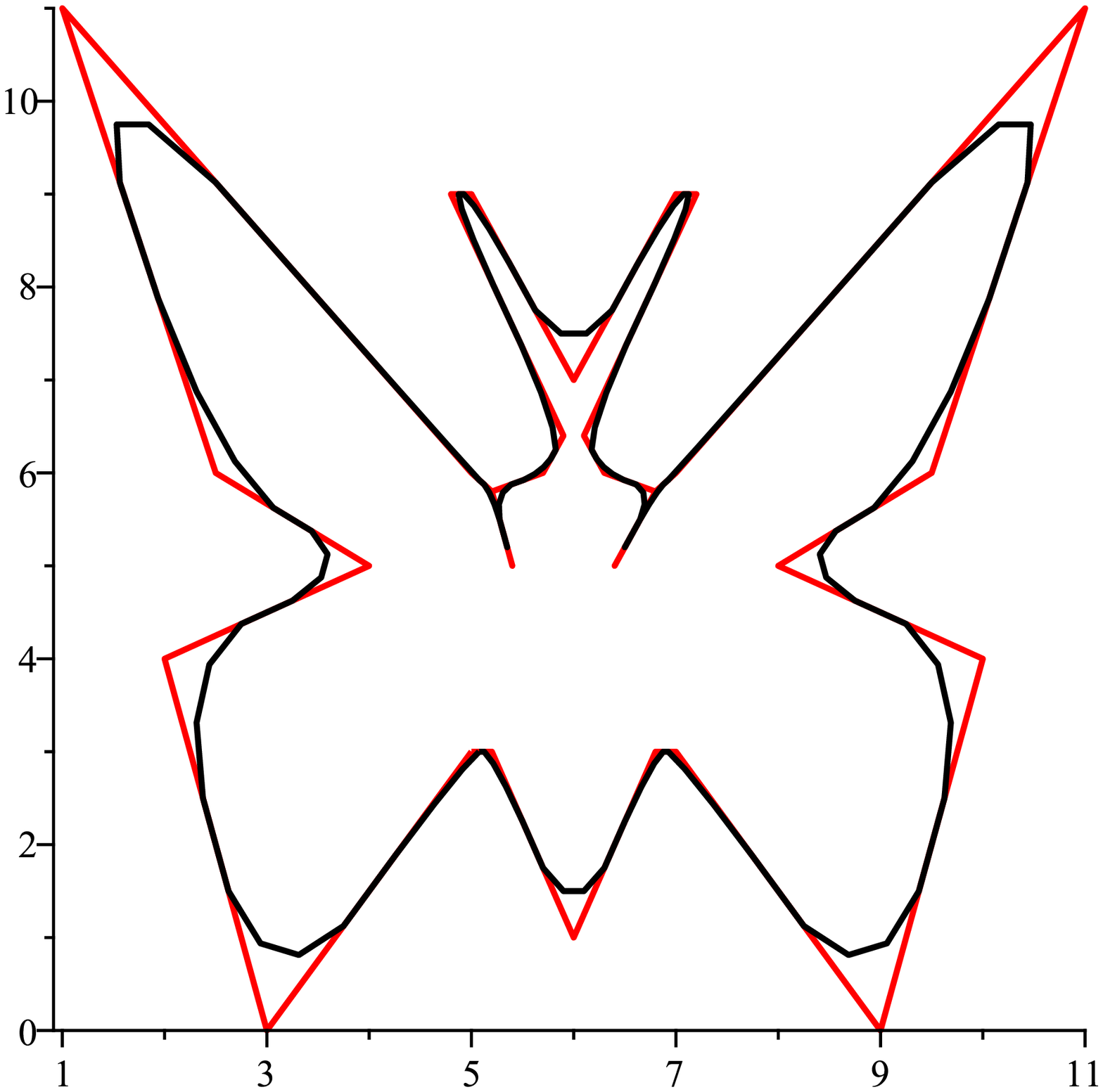, width=1.5 in} & \epsfig{file=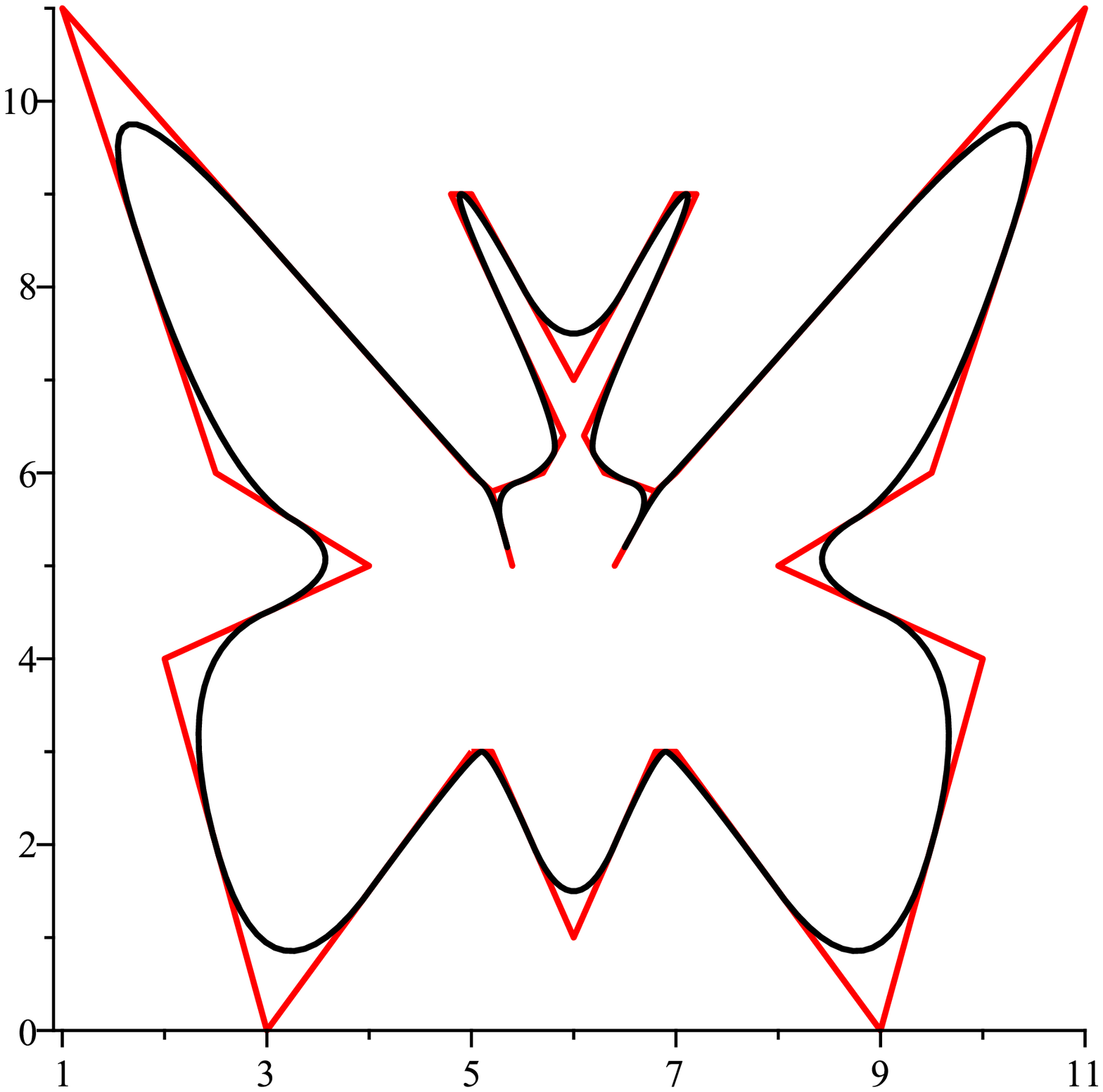, width=1.5 in} \\
		(a) One SS by BSS &(b) One SS by QSS & (c) Two SSs by BSS & (d) Two SSs by QSS
	\end{tabular}
	\caption{\label{w4}\emph{Curves generated by the binary and quaternary subdivision schemes (\ref{consta40}) and (\ref{consta41}) respectively.}}
\end{figure}
Figure \ref{w4} shows the two dimensional shapes fitted by the $2$-point binary subdivision scheme (\ref{consta40})  and the $2$-point relaxed quaternary subdivision scheme  (\ref{consta41}). This figure shows the speedy convergence of the quaternary subdivision scheme as compare to the binary subdivision scheme.
\begin{cor}
	Let ${g}_{\varphi}^{k}:$  $\varphi \in \mathbb{Z}$ be the control points at the $k$-th subdivision step and ${g}_{\varphi}^{k+1}:$  $\varphi \in \mathbb{Z}$ be the refined data/points at the $(k + 1)$-th refinement step. If we use $m=1$ in the $(2m+2)$-point binary subdivision scheme (\ref{const1a}), we get the generalized form of the following $6$-point binary subdivision scheme:
	\begin{eqnarray}\label{consta42}
		\left\{\begin{array}{ccccc}
			{g}_{2\varphi-1}^{k+1}&=&\beta_{-4}\,\ g_{\varphi-3}^{k}+\beta_{-2}\,\ g_{{\varphi-2}}^{k}+\beta_{0}\,\ g_{{\varphi-1}}^{k}+\beta_{2}\,\ g_{{\varphi}}^{k}+\beta_{4}\,\ g_{\varphi+1}^{k}+\beta_{6}\,\ g_{\varphi+2}^{k},\\ \\
			{g}_{2\varphi}^{k+1}&=&\beta_{6}\,\ g_{{\varphi-2}}^{k}+\beta_{4}\,\ g_{{\varphi-1}}^{k}+\beta_{2}\,\ g_{{\varphi}}^{k}+\beta_{0}\,\ g_{\varphi+1}^{k}+\beta_{-2}\,\ g_{\varphi+2}^{k}+\beta_{-4}\,\ g_{\varphi+3}^{k}.
		\end{array}\right.
	\end{eqnarray}
	After comparing $6$-point binary subdivision scheme (\ref{consta42}) with $6$-point approximating subdivision scheme  by \cite{Siddiqi2}, we get the values of following 6 unknowns:
	\begin{eqnarray}\label{a46}
		\left\{\begin{array}{ccccc}
			\beta_{-4}={\frac {1}{122880}}, \,\  \beta_{-2}={\frac {3119}{122880}}, \,\ \beta_{0}={\frac {6719}{20480}}, \,\ \beta_{2}={\frac {31927}{61440}},\,\
			\beta_{4}={\frac {15349}{122880}}, \,\ \beta_{6}={\frac {81}{40960}}.
		\end{array}\right.
	\end{eqnarray}
	Now if $m=1$, the coefficients of the control points in the $8$-point relaxed quaternary subdivision scheme (\ref{const75a}) are:
	\begin{eqnarray}\label{beta-r-1}
	\nonumber	\check{a}_{1}&=&\beta_{6}\beta_{-4}, \,\ \,\ \check{a}_{2}=\beta_{6}\beta_{-2}+\beta_{2}\beta_{-4}+\beta_{4}\beta_{6}, \,\ \,\ \check{a}_{3}=\beta_{6}\beta_{0}+\beta_{2}\beta_{-2}+\beta_{-2}\beta_{-4}+\beta_{4}^{2}+\beta_{0}\beta_{6},\\&& \nonumber\check{a}_{4}=\beta_{6}\beta_{2}+\beta_{2}\beta_{0}+\beta_{-2}^{2}+\beta_{4}\beta_{2}+\beta_{0}\beta_{4}+\beta_{-4}
	\nonumber	\beta_{6}, \,\ \,\ \check{a}_{5}=\beta_{6}\beta_{4}+\beta_{2}^{2}+\beta_{-2}\beta_{0}+\beta_{4}\beta_{0}\\&&+\beta_{0}\beta_{2}
		+\beta_{-4}\beta_{4}, \,\ \,\ \check{a}_{6}=\beta_{6}^{2}+\beta_{2}\beta_{4}+\beta_{-2}\beta_{2}+\beta_{4}\beta_{-2}+\beta_{0}^{2}
	\nonumber	+\beta_{-4}\beta_{2}, \,\ \,\ \check{a}_{7}=\beta_{2}\beta_{6}+\\&&\beta_{-2}\beta_{4}+\beta_{4}\beta_{-4}+\beta_{0}\beta_{-2}
	\nonumber	+\beta_{-4}\beta_{0}, \,\ \,\ \check{a}_{8}=\beta_{-2}\beta_{6}+\beta_{0}\beta_{-4}+\beta_{-4}\beta_{-2}, \,\ \,\ \check{a}_{9}=\beta_{-4}^{2}\\&&
		\check{b}_{1}=\beta_{6}^{2}+\beta_{4}\beta_{-4}, \,\ \,\ \check{b}_{2}=\beta_{6}\beta_{4}+\beta_{2}\beta_{6}+
		\beta_{4}\beta_{-2}+\beta_{0}\beta_{-4}, \,\ \,\ \check{b}_{3}=\beta_{6}\beta_{2}+\beta_{2}\beta_{4}+
	\nonumber	\beta_{-2}\beta_{6}\\&&+\beta_{4}\beta_{0}+\beta_{0}\beta_{-2}+\beta_{-4}^{2}, \,\ \,\ \check{b}_{4}=\beta_{6}\beta_{0}+
	\nonumber	\beta_{2}^{2}+\beta_{-2}\beta_{4}+\beta_{4}\beta_{2}+\beta_{0}^{2}+\beta_{-4}\beta_{-2}, \,\ \,\  \check{b}_{5}=\beta_{6}\\&&\times
		\beta_{-2}+\beta_{2}\beta_{0}+\beta_{-2}\beta_{2}+\beta_{4}^{2}+\beta_{0}\beta_{2}+\beta_{-4}\beta_{0}, \,\ \,\      \nonumber  \check{b}_{6}=\beta_{6}\beta_{-4}+\beta_{2}\beta_{-2}+\beta_{-2}\beta_{0}+\beta_{4}\\&&\times\beta_{6}+\beta_{0}\beta_{4}+\beta_{-4}
	\nonumber	\beta_{2}, \,\ \,\ \check{b}_{7}=\beta_{2}\beta_{-4}+\beta_{-2}^{2}+\beta_{0}\beta_{6}+\beta_{-4}\beta_{4}, \,\ \,\ \check{b}_{8}=\beta_{-2}\beta_{-4}+\beta_{-4}\beta_{6}.\\
	\end{eqnarray}
	By using the values of $\beta_{-4}, \beta_{-2}, \beta_{0}, \beta_{2}, \beta_{4}$, $\beta_{6}$ from (\ref{a46}) into the (\ref{beta-r-1})  and after simplification, we get the following quaternary approximating subdivision scheme:
	\begin{eqnarray}\label{consta43}
		\left\{\begin{array}{ccccc}
			{g}_{4\varphi-2}^{k+1}&=&\check{a}_{1}g_{\varphi-4}^{k}+\check{a}_{2}g_{\varphi-3}^{k}+\check{a}_{3}g_{\varphi-2}^{k}
			+\check{a}_{4}g_{\varphi-1}^{k}+\check{a}_{5}g_{\varphi}^{k}+\check{a}_{6}g_{\varphi+1}^{k}+\check{a}_{7}
			g_{\varphi+2}^{k}+\check{a}_{8}g_{\varphi+3}^{k}+\check{a}_{9}g_{\varphi+4}^{k},\\ \\
			{g}_{4\varphi-1}^{k+1}&=&\check{a}_{9}g_{\varphi-4}^{k}+\check{a}_{8}g_{\varphi-3}^{k}+\check{a}_{7}g_{\varphi-2}^{k}
			+\check{a}_{6}g_{\varphi-1}^{k}+\check{a}_{5}g_{\varphi}^{k}+\check{a}_{4}g_{\varphi+1}^{k}+\check{a}_{3}
			g_{\varphi+2}^{k}+\check{a}_{2}g_{\varphi+3}^{k}+\check{a}_{1}g_{\varphi+4}^{k},\\ \\
			{g}_{4\varphi}^{k+1}&=&\check{b}_{1}g_{\varphi-3}^{k}+\check{b}_{2}g_{\varphi-2}^{k}+\check{b}_{3}g_{\varphi-1}^{k}
			+\check{b}_{4}g_{\varphi}^{k}+\check{b}_{5}g_{\varphi+1}^{k}+\check{b}_{6}g_{\varphi+2}^{k}+\check{b}_{7}
			g_{\varphi+3}^{k}+\check{b}_{8}g_{\varphi+4}^{k},\\ \\
			{g}_{4\varphi+1}^{k+1}&=&\check{b}_{8}g_{\varphi-3}^{k}+\check{b}_{7}g_{\varphi-2}^{k}+\check{b}_{6}g_{\varphi-1}^{k}
			+\check{b}_{5}g_{\varphi}^{k}+\check{b}_{4}g_{\varphi+1}^{k}+\check{b}_{3}g_{\varphi+2}^{k}+\check{b}_{2}
			g_{\varphi+3}^{k}+\check{b}_{1}g_{\varphi+4}^{k},
		\end{array}\right.
	\end{eqnarray}
	where
	\begin{eqnarray}\label{a47}
		\left\{\begin{array}{ccccc}
			\check{a}_{1}&=&{\frac{27}{1677721600}}, \,\ \,\ \check{a}_{2}={\frac{2275789}{7549747200}}, \,\ \,\ \check{a}_{3}={\frac{9086963}{301989888}}, \,\ \,\
			\check{a}_{4}={\frac{699721619}{2516582400}}, \,\  \check{a}_{5}={\frac{369990379}{754974720}}, \\ \\ \check{a}_{6}&=&{\frac{1426235351}{7549747200}}, \,\ \,\ \check{a}_{7}={\frac{31530847}{2516582400}}, \,\ \,\ \check{a}_{8}={\frac{16027}{301989888}}, \,\  \check{a}_{9}={\frac{1}{15099494400}}
			\check{b}_{1}={\frac{37199}{7549747200}}, \\ \\ \check{b}_{2}&=&{\frac{33580087}{7549747200}}, \,\  \check{b}_{3}={\frac{290148073}{2516582400}}, \,\ \check{b}_{4}={\frac{674031991}{1509949440}}, \,\ \,\ \check{b}_{5}={\frac{558397097}{1509949440}}, \,\ \,\ \check{b}_{6}={\frac{157912247}{2516582400}}, \\ \\ \check{b}_{7}&=&{\frac{9801833}{7549747200}}, \,\ \,\ \check{b}_{8}={\frac{1681}{7549747200}}.
		\end{array}\right.
	\end{eqnarray}
\end{cor}
\begin{figure}[h]
	\begin{tabular}{cccc}
		\epsfig{file=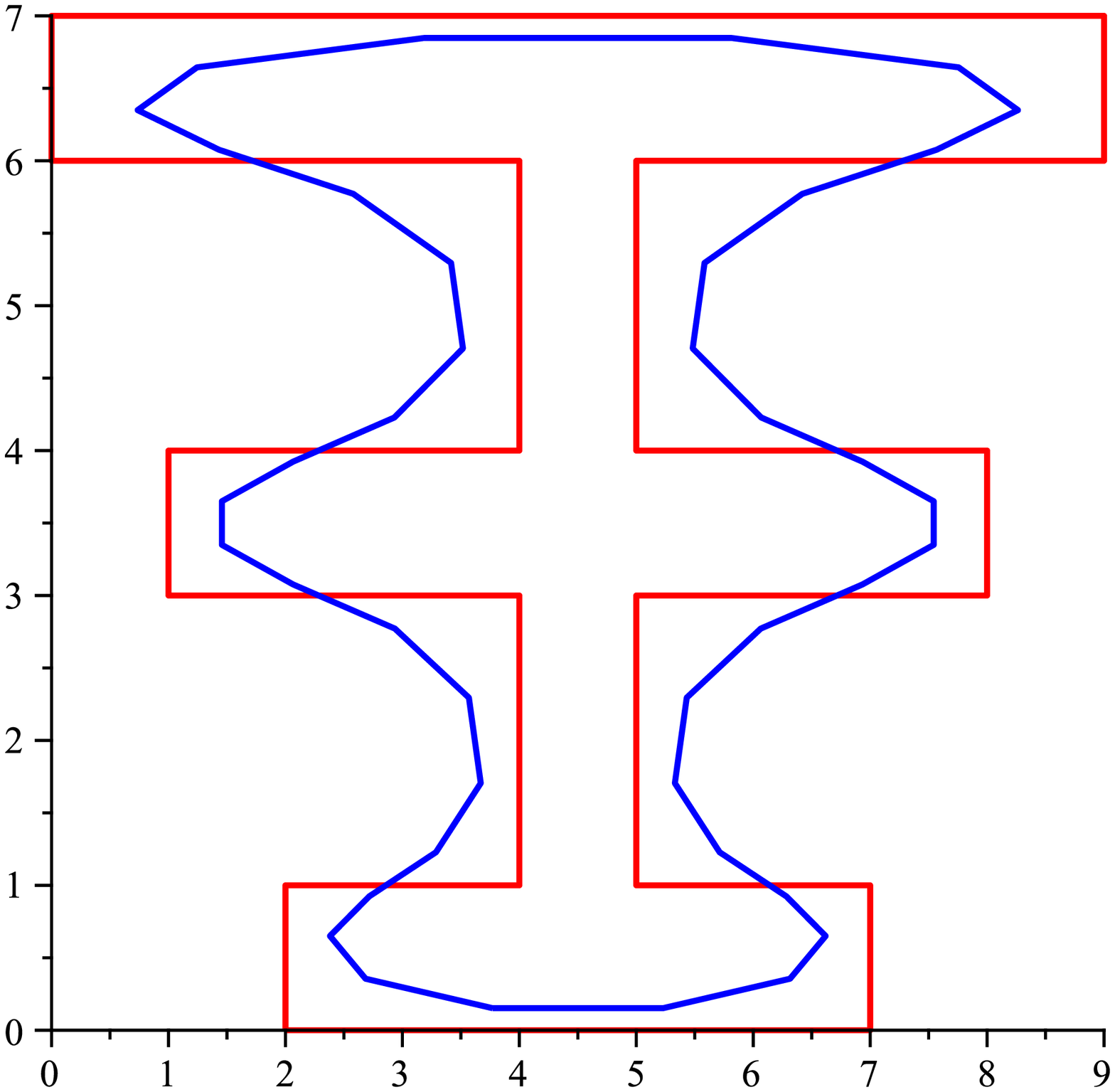, width=1.5 in} & \epsfig{file=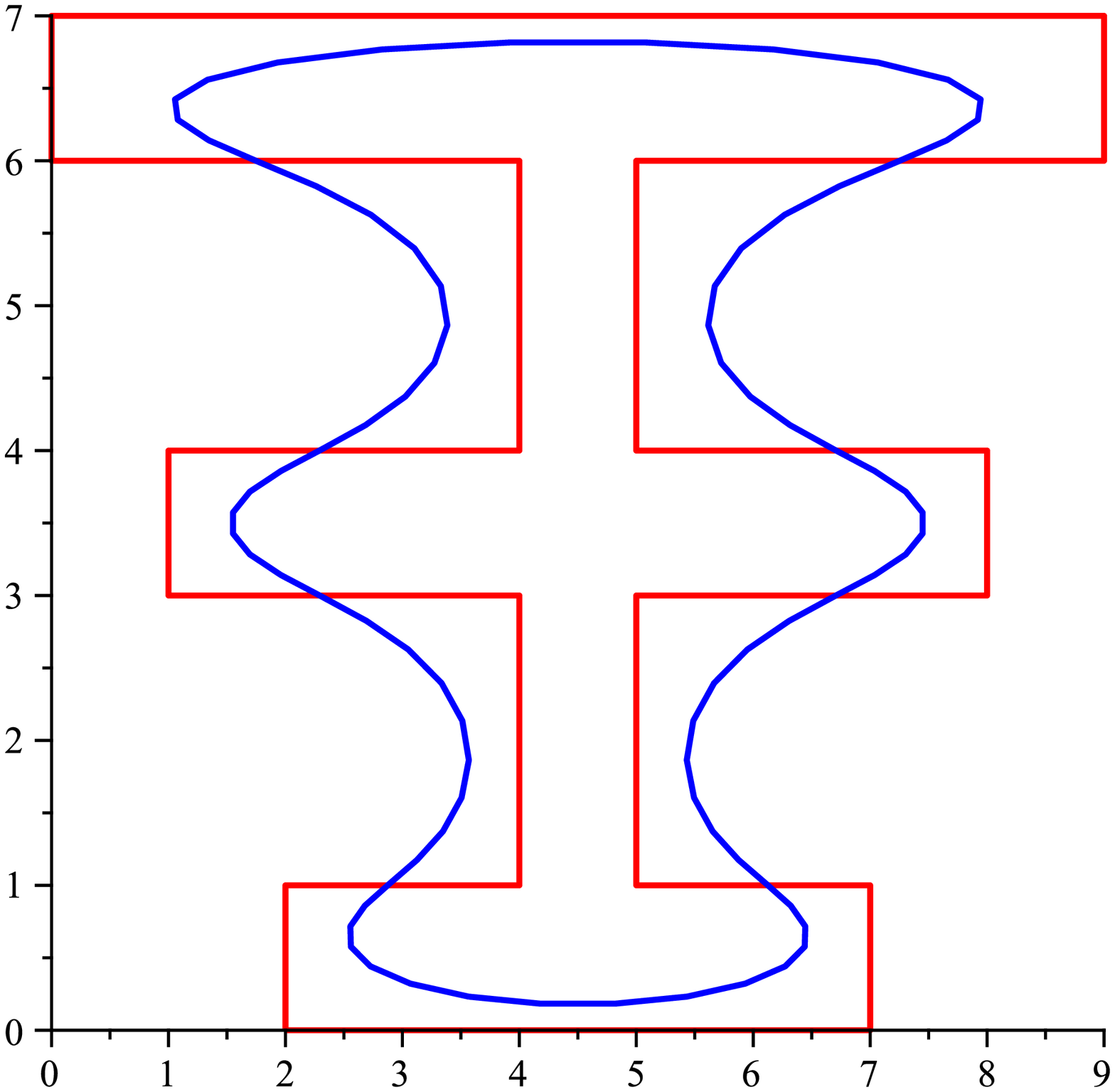, width=1.5 in} & \epsfig{file=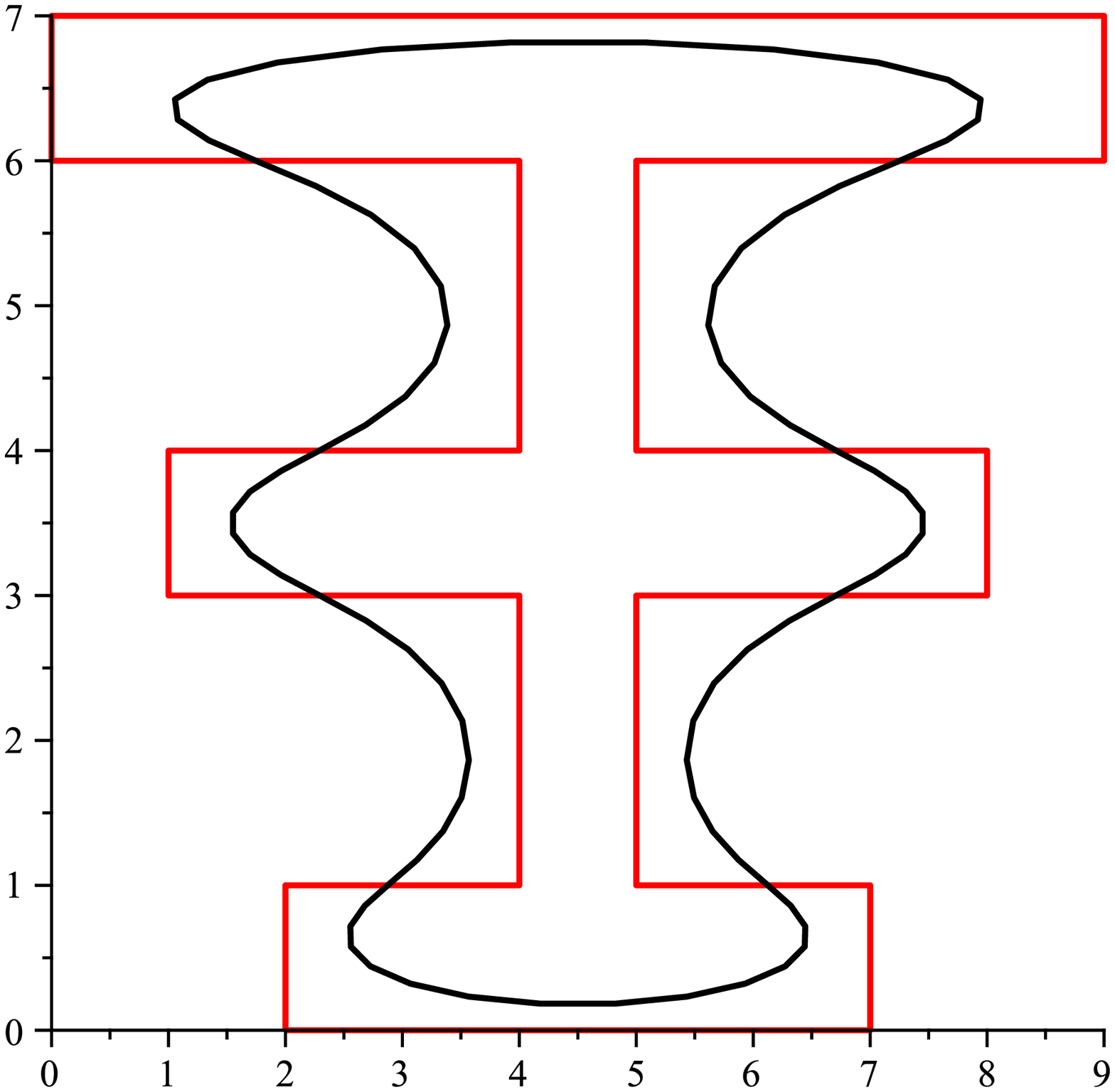, width=1.5 in} & \epsfig{file=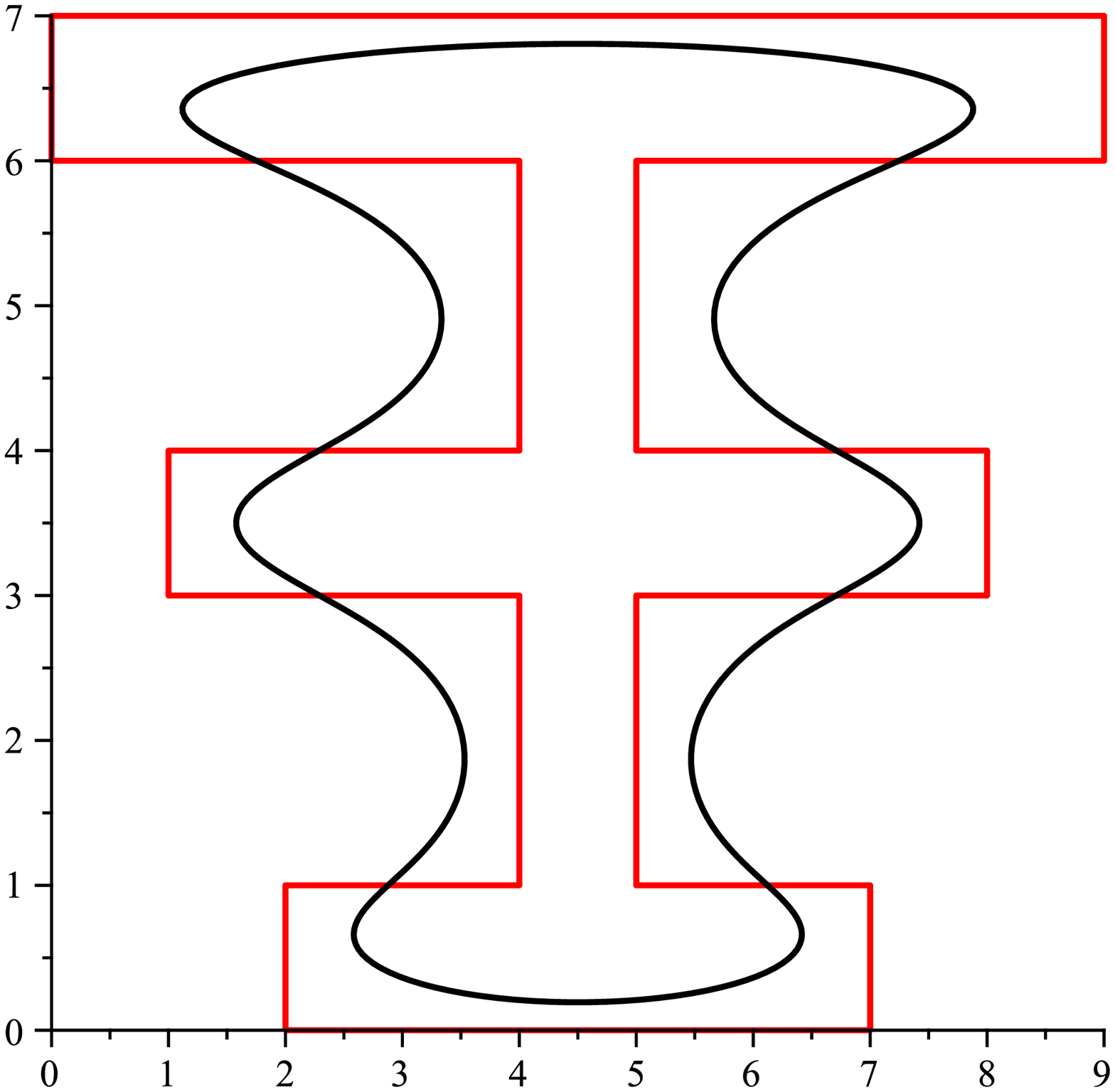, width=1.5 in} \\
		(a) One SS by BSS &(b) One SS by QSS & (c) Two SSs by BSS & (d) Two SSs by QSS
	\end{tabular}
	\caption{\label{w5}\emph{Curves generated by the binary and quaternary subdivision schemes (\ref{consta42}) and (\ref{consta43}) respectively.}}
\end{figure}
The graphical results of the $6$-point binary subdivision schemes (\ref{consta42}) and the $8$-point relaxed quaternary approximating subdivision schemes (\ref{consta43}) are reviewed in Figure \ref{w5}. It is clear that $8$-points relaxed quaternary approximating subdivision scheme uses less iterations for smoothness comparatively to its corresponding $6$-points binary subdivision scheme.
\begin{cor}\label{cor-6}
	The binary subdivision scheme (\ref{const1a}) for $m=2$ reduces to an approximating subdivsion scheme whose each subdivision rule use the linear combination of $10$ control points of subdivision level $k$ to get a control point at subdivision level $k+1$.  So the general form of this scheme is:
	\begin{eqnarray}\label{const76}
		\left\{\begin{array}{cccc}
			{g}_{2\varphi-1}^{k+1}&=&\beta_{-8}\,\ g_{\varphi-5}^{k}+\beta_{-6}\,\ g_{\varphi-4}^{k}+
			\beta_{-4}\,\ g_{\varphi-3}^{k}+\beta_{-2}\,\ g_{{\varphi-2}}^{k}+\beta_{0}\,\ g_{{\varphi-1}}^{k}\\&&+\beta_{2}\,\ g_{{\varphi}}^{k}+\beta_{4}\,\ g_{{\varphi+1}}^{k}+\beta_{6}\,\ g_{{\varphi+2}}^{k}+\beta_{8}\,\ g_{\varphi+3}^{k}+\beta_{10}\,\ g_{{\varphi+4}}^{k},\\
			{g}_{2\varphi}^{k+1}&=&\beta_{10}\,\ g_{\varphi-4}^{k}+\beta_{8}\,\ g_{\varphi-3}^{k}+
			\beta_{6}\,\ g_{\varphi-2}^{k}+\beta_{4}\,\ g_{{\varphi-1}}^{k}+\beta_{2}\,\ g_{{\varphi}}^{k}+\beta_{0}\,\ g_{\varphi+1}^{k}\\&&+\beta_{-2}\,\ g_{\varphi+2}^{k}+\beta_{-4}\,\ g_{\varphi+3}^{k}+\beta_{-6}\,\ g_{\varphi+4}^{k}+\beta_{-8}\,\ g_{\varphi+5}^{k}.
		\end{array}\right.
	\end{eqnarray}
	To get the values of 10 unknowns $\beta_{-8}, \beta_{-6} \ldots  \beta_{8}$, $\beta_{10}$, we  uses the $10$-point binary subdivision scheme presented by  \cite{Siddiqi2} so that the control points at level $k+1$ in (\ref{const76}) become the convex combination of the control points of level $k$. Hence the values of unknowns are:
	\begin{eqnarray}\label{const77}
		\left\{\begin{array}{cccc}
			\beta_{-8}={\frac {12155}{33554432}}, \,\ \,\ \beta_{-6}={\frac {-138567}{33554432}}, \,\ \,\
			\beta_{-4}={\frac {188955}{8388608}},  \,\ \,\ \beta_{-2}={\frac {-692835}{8388608}}, \,\ \,\ \beta_{0}={\frac {4849845}{16777216}}, \\ \\ \beta_{2}={\frac {14549535}{16777216}}, \,\ \,\ \beta_{4}={\frac {-969969}{8388608}}, \,\ \,\ \beta_{6}={\frac {230945}{8388608}}, \,\ \,\  \beta_{8}={\frac {-159885}{33554432}}, \,\ \,\ \beta_{10}={\frac {13585}{33554432}}.
		\end{array}\right.
	\end{eqnarray}
Now we simplify  (\ref{const75a}) for $m=2$ and get the general form of the $14$-point relaxed quaternary subdivision scheme. Now we use the unknowns from (\ref{const77})  into (\ref{const75a}) and get the following quaternary subdivision scheme which is relaxed bacasue its two subdivision rules are the convex combination of 14 control points of level $k$ whereas the remaining two subdivision  rules are the convex combination of 15 control points of level $k$.
	\begin{eqnarray}\label{const80}
		\left\{\begin{array}{cccc}
			{g}_{4\varphi-2}^{k+1}&=&\tilde{a}_{1}g_{\varphi-7}^{k}+\tilde{a}_{2}g_{\varphi-6}^{k}+\tilde{a}_{3}g_{\varphi-5}^{k}+\tilde{a}_{4} g_{\varphi-4}^{k}+\tilde{a}_{5}g_{\varphi-3}^{k}+\tilde{a}_{6}g_{\varphi-2}^{k}+\tilde{a}_{7}g_{\varphi-1}^{k}+\tilde{a}_{8}
			g_{\varphi}^{k}\\&&+\tilde{a}_{9}g_{\varphi+1}^{k}+\tilde{a}_{10}g_{\varphi+2}^{k}+\tilde{a}_{11}g_{\varphi+3}^{k}+\tilde{a}_{12}
			g_{\varphi+4}^{k}+\tilde{a}_{13}g_{\varphi+5}^{k}+\tilde{a}_{14}g_{\varphi+6}^{k}+\tilde{a}_{15}g_{\varphi+7}^{k},\\ \\
			{g}_{4\varphi-1}^{k+1}&=&\tilde{a}_{15}g_{\varphi-7}^{k}+\tilde{a}_{14}g_{\varphi-6}^{k}+\tilde{a}_{13} g_{\varphi-5}^{k}+\tilde{a}_{12}g_{\varphi-4}^{k}+\tilde{a}_{11} g_{\varphi-3}^{k}+\tilde{a}_{10}g_{\varphi-2}^{k}+ \tilde{a}_{9}g_{\varphi-1}^{k}+\\&&\tilde{a}_{8}g_{\varphi}^{k}+
			\tilde{a}_{7}g_{\varphi+1}^{k}+\tilde{a}_{6}g_{\varphi+2}^{k}+ \tilde{a}_{5}g_{\varphi+3}^{k}+\tilde{a}_{4}g_{\varphi+4}^{k}+
			\tilde{a}_{3}g_{\varphi+5}^{k}+\tilde{a}_{2}g_{\varphi+6}^{k}+\tilde{a}_{1}g_{\varphi+7}^{k},\\ \\
			{g}_{4\varphi}^{k+1}&=&\tilde{b}_{1}g_{\varphi-6}^{k}+\tilde{b}_{2}g_{\varphi-5}^{k}+\tilde{b}_{3}g_{\varphi-4}^{k}
			+\tilde{b}_{4}g_{\varphi-3}^{k}
			+\tilde{b}_{5}g_{\varphi-2}^{k}+\tilde{b}_{6}g_{\varphi-1}^{k}+\tilde{b}_{7}g_{\varphi}^{k}+\tilde{b}_{8}
			g_{\varphi+1}^{k}\\&& +
			\tilde{b}_{9}g_{\varphi+2}^{k}+\tilde{b}_{10}g_{\varphi+3}^{k}+\tilde{b}_{11}g_{\varphi+4}^{k}+
			\tilde{b}_{12}g_{\varphi+5}^{k}+\tilde{b}_{13}g_{\varphi+6}^{k}+\tilde{b}_{14}g_{\varphi+7}^{k},\\ \\
			{g}_{4\varphi+1}^{k+1}&=&\tilde{b}_{14}g_{\varphi-6}^{k}+\tilde{b}_{13}g_{\varphi-5}^{k}+\tilde{b}_{12}
			g_{\varphi-4}^{k}+\tilde{b}_{11}
			g_{\varphi-3}^{k}+\tilde{b}_{10}g_{\varphi-2}^{k}+\tilde{b}_{9}g_{\varphi-1}^{k}+\tilde{b}_{8}g_{\varphi}^{k}
			+\tilde{b}_{7}\\&&\times g_{\varphi+1}^{k}+\tilde{b}_{6}g_{\varphi+2}^{k}+\tilde{b}_{5}g_{\varphi+3}^{k}+\tilde{b}_{4}g_{\varphi+4}^{k}
			+\tilde{b}_{3}g_{\varphi+5}^{k}
			+\tilde{b}_{2}g_{\varphi+6}^{k}+\tilde{b}_{1}g_{\varphi+7}^{k},
		\end{array}\right.
	\end{eqnarray}
	where
	\begin{eqnarray}\label{a77}
		\left\{\begin{array}{cccc}
			\tilde{a}_{1}&=&{\frac{165125675}{1125899906842624}}, \,\ \,\ \tilde{a}_{2}={\frac{896759435}{140737488355328}}, \,\ \,\ \tilde{a}_{3}={\frac{208816685055}{1125899906842624}},\,\  \tilde{a}_{4}={\frac{-350111003385}{140737488355328}}, \\ \\ \tilde{a}_{5}&=&{\frac{15443267900775}{1125899906842624}}, \,\ \,\ \tilde{a}_{6}={\frac{-845664470365}{17592186044416}},
			\,\  \tilde{a}_{7}={\frac{165267038051115}{1125899906842624}}, \,\  \tilde{a}_{8}={\frac{66176709385215}{70368744177664}}
			, \\ \\ \tilde{a}_{9}&=&{\frac{-67374962898815}{1125899906842624}},\,\  \tilde{a}_{10}={\frac{1418419563275}{140737488355328}}
			, \,\ \,\ \tilde{a}_{11}={\frac{-772276338691}{1125899906842624}}, \,\ \,\ \tilde{a}_{12}={\frac{-17904682965}{140737488355328}},
			\\ \\  \tilde{a}_{13}&=&{\frac{6861145005}{1125899906842624}}, \,\ \,\ \tilde{a}_{14}={\frac{351267345}{70368744177664}}, \,\ \,\ \tilde{a}_{15}={\frac{147744025}{1125899906842624}}
			\,\ \tilde{b}_{1}={\frac{-879424975}{562949953421312}}, \\ \\ \tilde{b}_{2}&=&{\frac{-7313797205}{562949953421312}}, \,\  \tilde{b}_{3}={\frac{198710020223}{281474976710656}},\,\  \tilde{b}_{4}={\frac{-1928438555615}{281474976710656}}, \,\  \tilde{b}_{5}={\frac{20280883755275}{562949953421312}}, \\ \\ \tilde{b}_{6}&=&{\frac{-78511743180975}{562949953421312}},\,\ \tilde{b}_{7}={\frac{105752643189765}{140737488355328}}, \,\ \,\ \tilde{b}_{8}={\frac{63006444414771}{140737488355328}}, \,\ \,\ \tilde{b}_{9}={\frac{-64970510830665}{562949953421312}},\\ \\  \tilde{b}_{10}&=&{\frac{17600920061725}{562949953421312}}, \,\ \,\ \tilde{b}_{11}={\frac{-1671012940025}{281474976710656}}, \,\ \,\ \tilde{b}_{12}=
			{\frac{163784577385}{281474976710656}},\\ \\  \tilde{b}_{13}&=&{\frac{-3080205843}{562949953421312}}, \,\ \,\ \tilde{b}_{14}={\frac{-759578105}{562949953421312}}.
		\end{array}\right.
	\end{eqnarray}
\end{cor}
\begin{figure}[h]
	\begin{tabular}{cccc}
		\epsfig{file=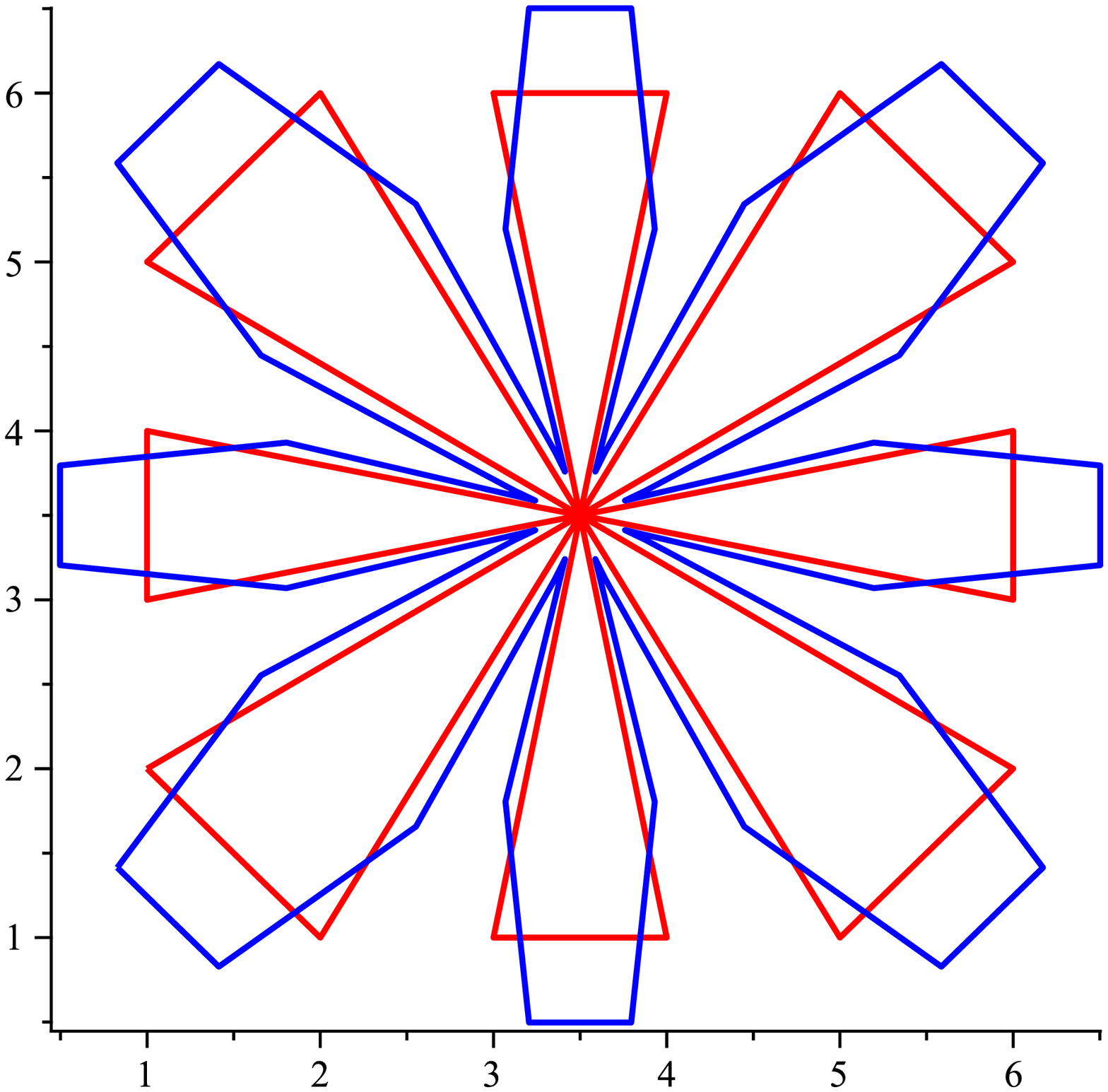, width=1.5 in} & \epsfig{file=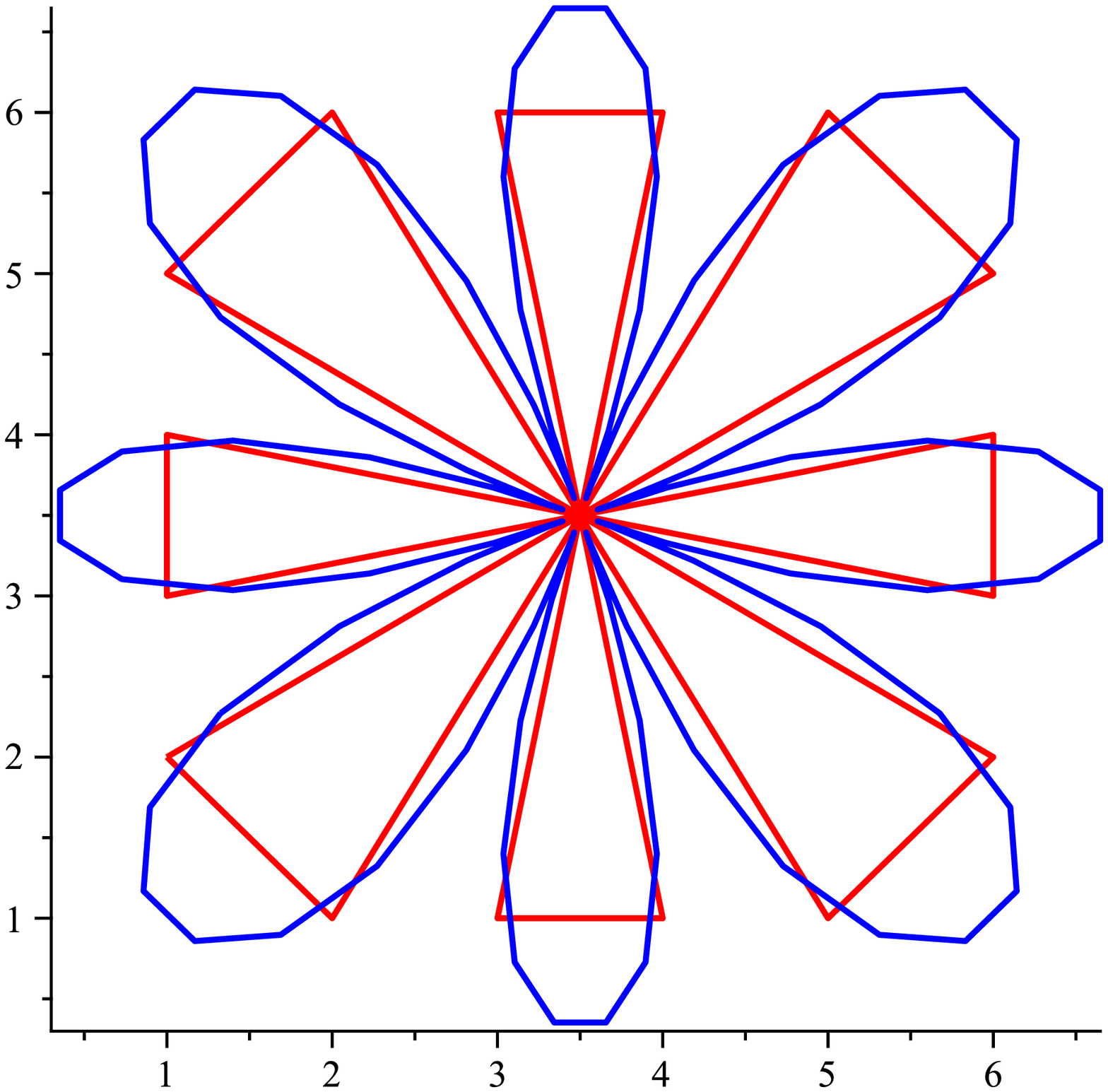, width=1.5 in} & \epsfig{file=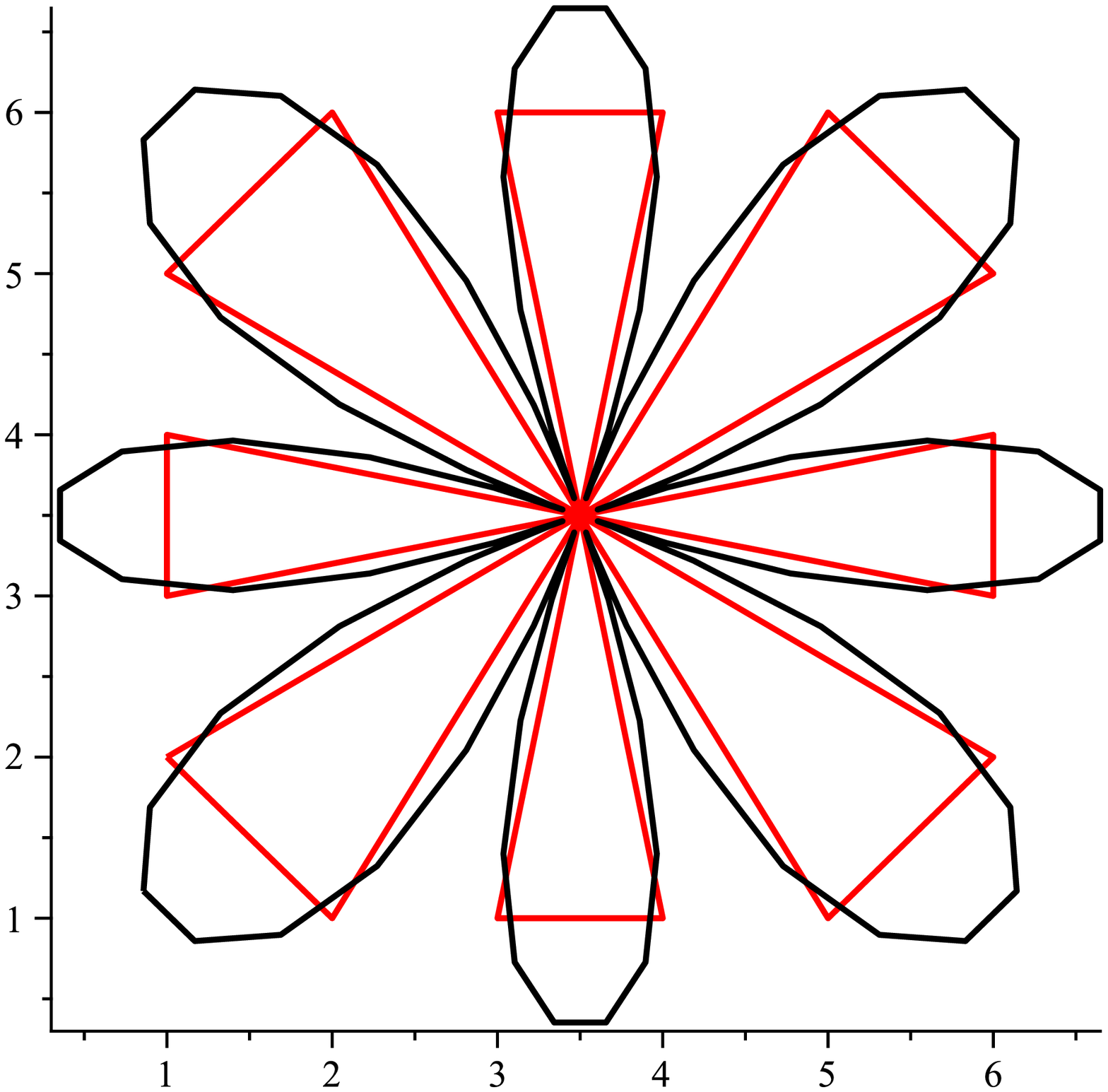, width=1.5 in} & \epsfig{file=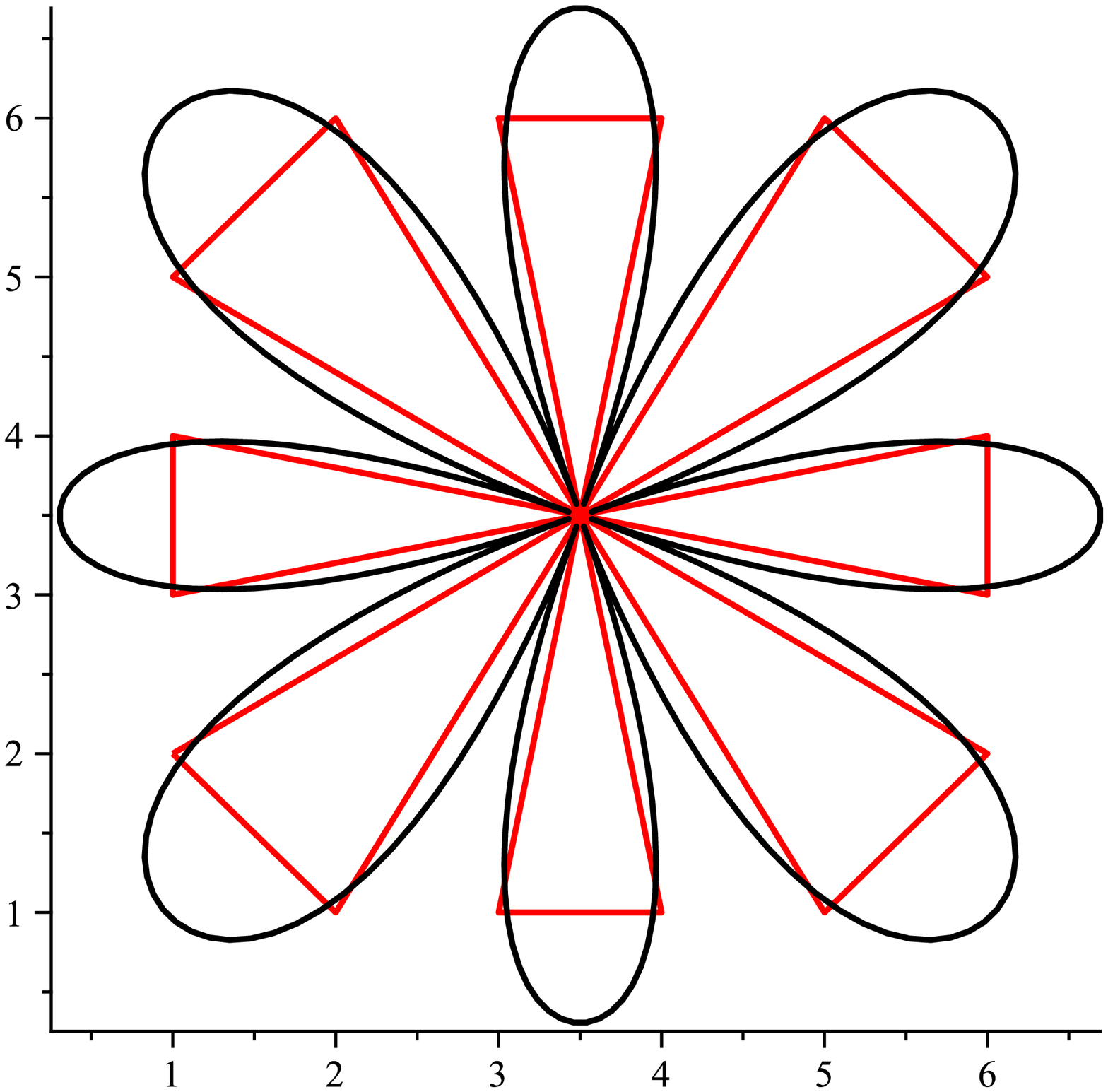, width=1.5 in} \\
		(a) One SS by BSS &(b) One SS by QSS & (c) Two SSs by BSS & (d) Two SSs by QSS
	\end{tabular}
	\caption{\label{w6}\emph{Curves generated by the binary and quaternary subdivision schemes (\ref{const76}) and (\ref{const80}) respectively.}}
\end{figure}
In Figure \ref{w6}, we present the graphical shapes generated by the binary subdivision scheme (\ref{const76}) and the quaternary approximating subdivision scheme (\ref{const80}).  The difference at first two subdivision levels can be visualized clearly from this figure.
\begin{cor}\label{cor-7}
	This corollary also shows the application of Theorem \ref{thm-odd-2} for $m=2$. So the general form  of the $10$-point binary subdivision scheme is same as given in (\ref{const76}) of Corollary \ref{cor-6}  and the values of unknowns can be get by \cite{Siddiqi2}, that are:
	\begin{eqnarray}\label{const82}
		\left\{\begin{array}{cccc}
			\beta_{-8}&=&{\frac {1}{95126814720}}, \,\ \beta_{-6}={\frac {390623}{19025362944}}, \,\
			\beta_{-4}={\frac {13138903}{3397386240}},  \,\ \beta_{-2}={\frac {1704546247}{23781703680}}, \\ \\  \beta_{0}&=&{\frac {14871214991}{47563407360}}, \,\  \beta_{2}={\frac {19761725357}{47563407360}}, \,\ \beta_{4}={\frac {833871641}{4756340736}}, \,\ \beta_{6}={\frac {488824339}{23781703680}},\\ \\   \beta_{8}&=&{\frac {40156777}{95126814720}}, \,\ \beta_{10}={\frac {243}{1174405120}}.
		\end{array}\right.
	\end{eqnarray}
	Hence the coefficients of the following $14$-point relaxed quaternary subdivision scheme which we get from (\ref{const75a}) for $m=2$ are the non-linear combination of the values given in (\ref{const82})
	\begin{eqnarray}\label{const83}
		\left\{\begin{array}{cccc}
			{g}_{4\varphi-2}^{k+1}&=&c_{1}g_{\varphi-7}^{k}+c_{2}g_{\varphi-6}^{k}+c_{3}g_{\varphi-5}^{k}+c_{4} g_{\varphi-4}^{k}+c_{5}g_{\varphi-3}^{k}+c_{6}
			g_{\varphi-2}^{k}+c_{7}g_{\varphi-1}^{k}+c_{8}
			g_{\varphi}^{k}\\&&+c_{9}g_{\varphi+1}^{k}+c_{10}
			g_{\varphi+2}^{k}+c_{11}g_{\varphi+3}^{k}+c_{12}g_{\varphi+4}^{k}
			+c_{13}g_{\varphi+5}^{k}+c_{14}g_{\varphi+6}^{k}+c_{15}g_{\varphi+7}^{k},\\ \\
			{g}_{4\varphi-1}^{k+1}&=&c_{15}g_{\varphi-7}^{k}+c_{14}g_{\varphi-6}^{k}+c_{13} g_{\varphi-5}^{k}+c_{12}g_{\varphi-4}^{k}+c_{11} g_{\varphi-3}^{k}+c_{10}g_{\varphi-2}^{k}+ c_{9}g_{\varphi-1}^{k}+\\&& c_{8}g_{\varphi}^{k}+
			c_{7}g_{\varphi+1}^{k}+c_{6}g_{\varphi+2}^{k}+
			c_{5}g_{\varphi+3}^{k}+c_{4}g_{\varphi+4}^{k}+
			c_{3}g_{\varphi+5}^{k}+c_{2}g_{\varphi+6}^{k}+c_{1}g_{\varphi+7}^{k},\\ \\
			{g}_{4\varphi}^{k+1}&=&d_{1}g_{\varphi-6}^{k}+d_{2}g_{\varphi-5}^{k}
			+d_{3}g_{\varphi-4}^{k}+d_{4}g_{\varphi-3}^{k}
			+d_{5}g_{\varphi-2}^{k}+d_{6}g_{\varphi-1}^{k}+
			d_{7}g_{\varphi}^{k}+d_{8}g_{\varphi+1}^{k}\\&& +
			d_{9}g_{\varphi+2}^{k}+d_{10}g_{\varphi+3}^{k}+d_{11}g_{\varphi+4}^{k}+
			d_{12}g_{\varphi+5}+d_{13}g_{\varphi+6}+
			d_{14}g_{\varphi+7}^{k},\\ \\
			{g}_{4\varphi+1}^{k+1}&=&d_{14}g_{\varphi-6}^{k}+d_{13}g_{\varphi-5}^{k}+d_{12}g_{\varphi-4}^{k}+d_{11}g_{\varphi-3}^{k}
			+d_{10}g_{\varphi-2}^{k}+d_{9}g_{\varphi-1}^{k}+d_{8}g_{\varphi}^{k}+\\&& d_{7}g_{\varphi+1}^{k}+
			d_{6}g_{\varphi+2}^{k}+d_{5}g_{\varphi+3}^{k}+d_{4}g_{\varphi+4}^{k}+d_{3}g_{\varphi+5}^{k}+d_{2}g_{\varphi+6}^{k}+
			d_{1}g_{\varphi+7}^{k},
		\end{array}\right.
	\end{eqnarray}
	where
	\begin{eqnarray}\label{t83}
		\left\{\begin{array}{cccc}
			c_{1}&=& {\frac{3}{1379227385882214400}}, \,\ \,\
			c_{2}={\frac{103850537699}{1131138859846651084800}}, \,\ \,\ c_{3}={\frac{384468662194361}{603274058584880578560}},\\ \\ c_{4}&=&{\frac{193184601091081183}{1131138859846651084800}}, \,\ \,\ c_{5}={\frac{62359154801651076991}{9049110878773208678400}}, \,\ \,\ c_{6}={\frac{10438071559376400791}{141392357480831385600}},\\ \\ c_{7}&=&{\frac{116970620124289395991}{430910041846343270400}}, \,\ \,\ c_{8}={\frac{137915687699480465}{359091701538619392}}, \,\ \,\ c_{9}={\frac{648462053977087642339}{3016370292924402892800}},\\ \\ c_{10}&=&{\frac{17125519634472990817}{377046286615550361600}}, \,\ \,\ c_{11}={\frac{28434156344565775493}{9049110878773208678400}}, \,\ \,\ c_{12}={\frac{19659422070653153}{377046286615550361600}},\\ \\ c_{13}&=&{\frac{186238665573617}{1809822175754641735680}}, \,\ \,\ c_{14}={\frac{346544687}{80795632846189363200}}, \,\ \,\
			c_{15}={\frac{1}{9049110878773208678400}},\\ \\
			d_{1}&=&{\frac{213788633}{4524555439386604339200}}, \,\ \,\
			d_{2}={\frac{8408857474501}{646365062769514905600}}, \,\ \,\ d_{3}={\frac{646332081504007}{46168933054965350400}},\\ \\   d_{4}&=&{\frac{989843345911370179}{754092573231100723200}}, \,\ \,\
			d_{5}={\frac{23601083531203039271}{904911087877320867840}}, \,\ \,\ d_{6}={\frac{241716880028225596243}{1508185146462201446400}},\\ \\ d_{7}&=&{\frac{136704671280724852463}{377046286615550361600}}, \,\ \,\ d_{8}={\frac{40597867969395053203}{125682095538516787200}}, \,\ \,\ d_{9}={\frac{24196254751816504787}{215455020923171635200}},\\ \\   d_{10}&=&{\frac{1800718087517760407}{129273012553902981120}}, \,\ \,\ d_{11}={\frac{1128217272575781919}{2262277719693302169600}}, \,\ \,\
			d_{12}={\frac{7348484779529681}{2262277719693302169600}},\\ \\ d_{13}&=&{\frac{1843780220327}{1508185146462201446400}}, \,\ \,\
			d_{14}={\frac{986399}{4524555439386604339200}}.
		\end{array}\right.
	\end{eqnarray}
\begin{figure}[h]
	\begin{tabular}{cccc}
		\epsfig{file=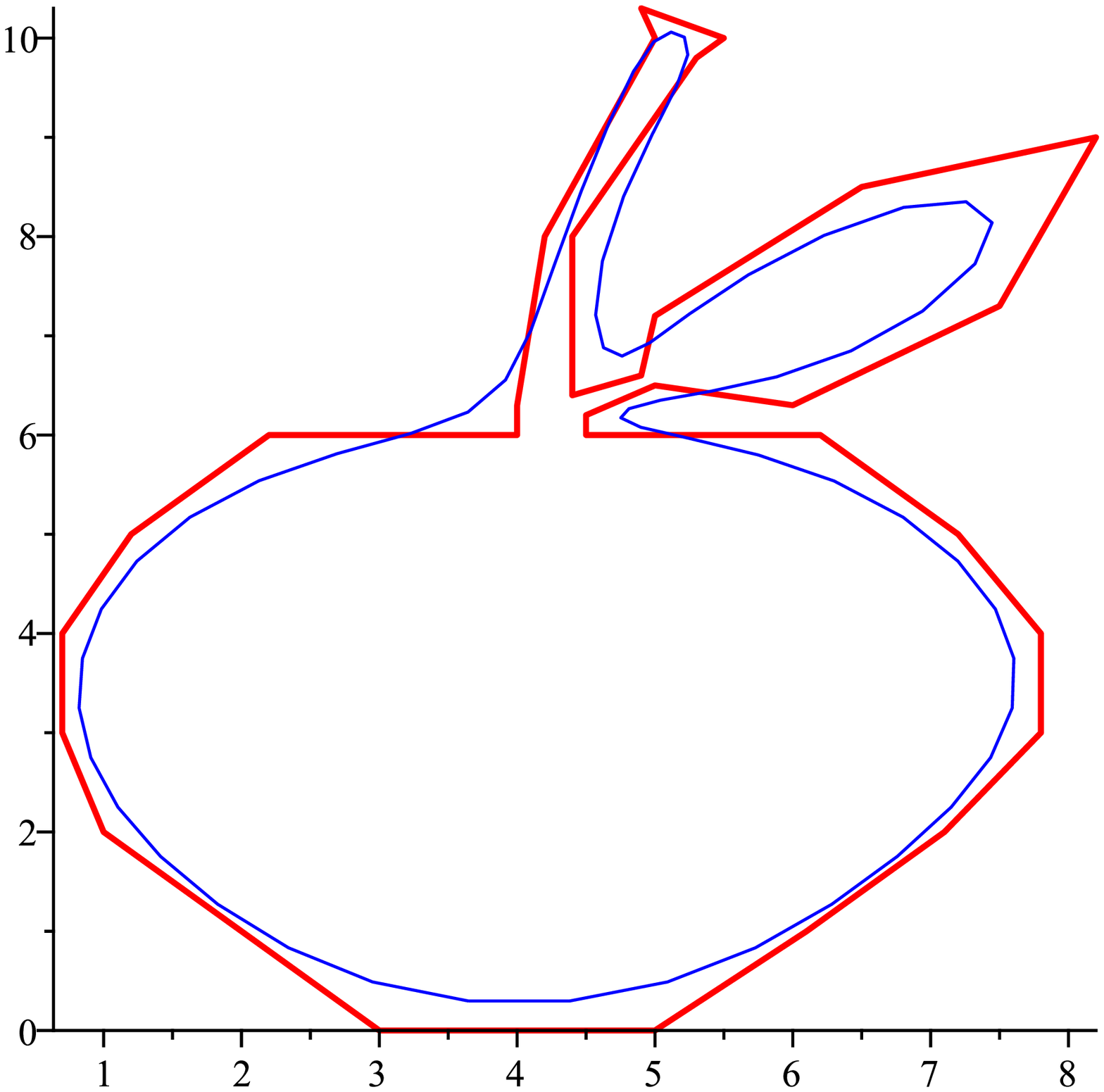, width=1.5 in} & \epsfig{file=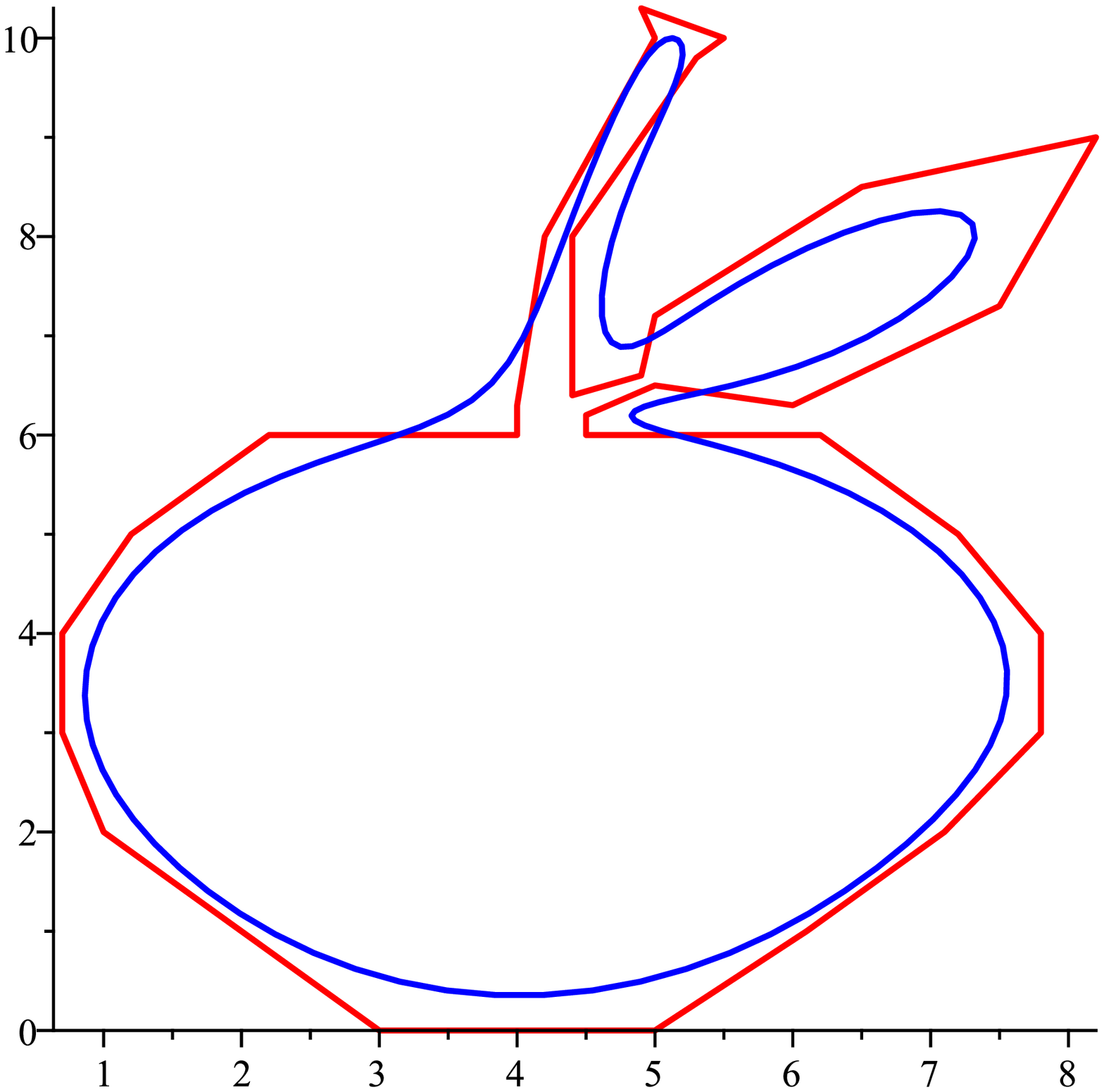, width=1.5 in} & \epsfig{file=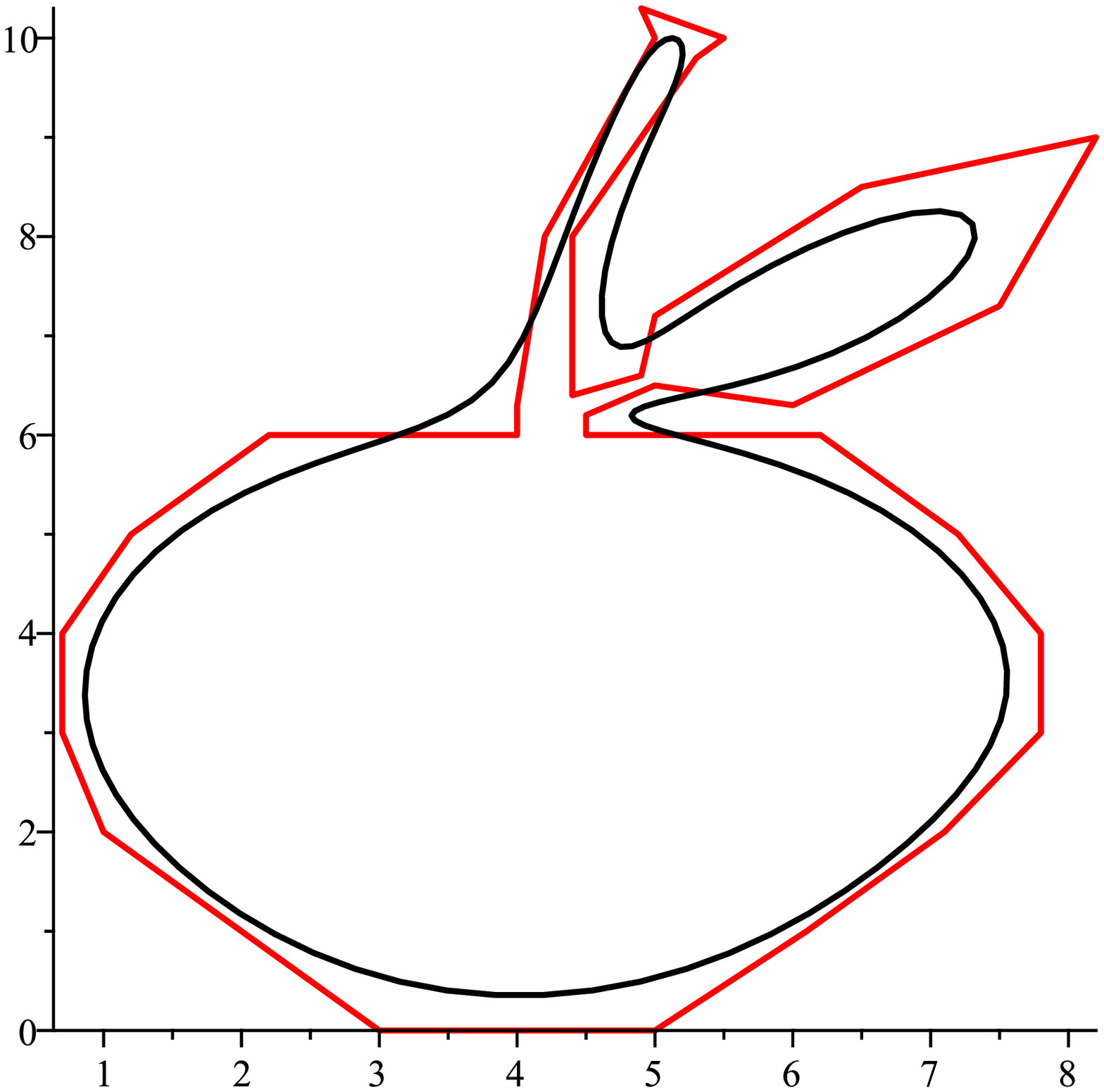, width=1.5 in} & \epsfig{file=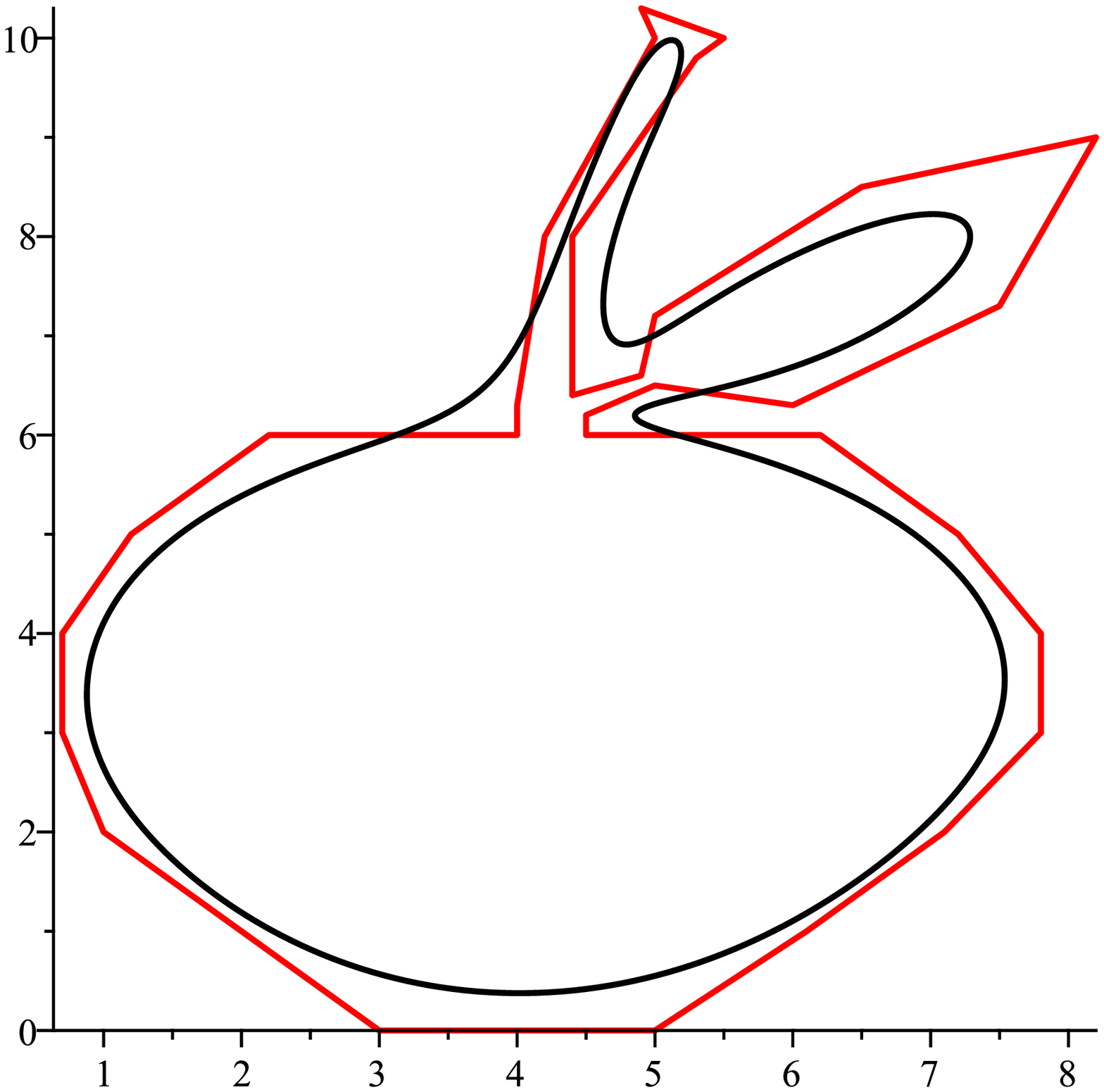, width=1.5 in} \\
		(a) One SS by BSS &(b) One SS by QSS & (c) Two SSs by BSS & (d) Two SSs by QSS
	\end{tabular}
	\caption{\label{w7}\emph{Curves generated by the pair of subdivision schemes given in Corollary \ref{cor-7}.}}
\end{figure}
\end{cor}
Figure \ref{w7} also give the comparison between the models generated by the $10$-point binary and the $14$-point relaxed quaternary subdivision schemes when masks of the subdivision schemes are all positive.
\section{H\"{o}lder's regularity of the presented pairs of subdivision schemes}
 In this section, we evaluate and compare the H\"{o}lder's regularity of the each pair of binary and quaternary subdivision schemes which we have discussed in the corollaries of Section 3. This evaluation is done by a well-known technique presented by \cite{Dyn2} which is defined here.
 \begin{defn}\label{def-2.3}
 	H\"{o}lder regularity  is an extension of the Laurent's polynomial of continuity which gives more information about any scheme. The H\"{o}lder regularity of subdivision scheme with Laurent's polynomial $\mu(c)$ can be computed in the following way. Let
 	\begin{equation*}\label{}
 		\mu(c)=\left(\frac{1+c+c^{2}+\ldots+c^{s-1}}{s}\right)^p \nu(c),
 	\end{equation*}
 	without loss of generality, we can suppose that $e_{0}$, $e_{1}$,\ldots, $e_{t-1}$,$e_{t}$ to be the non-zero coefficients of $\nu(c)$ and let $E_{0}$, $E_{1}$, \ldots, $E_{t-1}$, $E_{t}$ be the $t\times t$ matrices with elements:
 	\begin{equation}\label{a}
 		(E_{q})_{ij}=e_{t+i-sj+q}, \,\ \,\  i,j=1,2,\ldots,t-1,t \,\ \,\ and \,\ \,\ q=0,1,\ldots,t-1,t.
 	\end{equation}
 	Then the H\"{o}lder regularity of the subdivision schemes is given by
 	\begin{equation*}\label{}
 		r=p-\log_{s}(\xi).
 	\end{equation*}
 	Where $\xi$ is the joint spectral radius of the matrices $E_{0}$, $E_{1}$, \ldots, $E_{t-1}$, $E_{t}$, \\ i.e
 	\begin{equation*}\label{}
 		\xi=\rho(E_{0},E_{1},\ldots,E_{t-1},E_{t})
 		=\limsup_{l\rightarrow\infty}(\max\{\|E_{i(l)}E_{i(l-1)}\ldots E_{i(2)}E_{i(1)}\|^{\frac{1}{l}}_{\infty}: \,\ i_{l}\in[0,1]\}).
 	\end{equation*}
 	and
 	\begin{equation}\label{2aa}
 		\max\{\rho(E_{0}),\rho(E_{1}),\ldots,\rho(E_{t-1}),\rho(E_{t})\} \leq \max\{\|E_{0}\|_{\infty},\|E_{1}\|_{\infty},\ldots, \|E_{t-1}\|_{\infty},\|E_{t}\|_{\infty}\}
 	\end{equation}
 	Since $\xi$ is bounded from below by the spectral radii and above from the norm of the matrices $E_{0}$, $E_{1}$, \ldots, $E_{t-1}$, $E_{t}$.\\ Then
 	\begin{equation*}\label{}
 		\max\{\rho(E_{0}),\rho(E_{1}),\ldots,\rho(E_{t-1}),\rho(E_{t})\} \leq\xi\leq \max\{\|E_{0}\|_{\infty},\|E_{1}\|_{\infty},\ldots, \|E_{t-1}\|_{\infty},\|E_{t}\|_{\infty}\}.
 	\end{equation*}
 \end{defn}
  Given is an important remark about the Laurent polynomial representation of the binary and quaternary subdivision schemes.
 \begin{rem}
 	Througout the paper, the Laurent polynomial of the $4m$-point binary subdivision scheme is denoted by $\mu_{4m}(c)$, while the Laurent polynomial of the $(4m+2)$-point binary subdivision scheme is denoted by $\mu_{4m+2}(c)$. In the same way, the Laurent Polynomial of the $(6m-1)$-point relaxed quaternary schemes is denoted by $U_{6m-1}(c)$, while the Laurent polynomial of the $(6m+2)$-point relaxed quaternary scheme is denoted by the symbol $U_{6m+2}(c)$. 
 \end{rem}
In the following theorem, we estimate the H\"{o}lder's continuity of the $4$-point binary and its corresponding $5$-point relaxed quaternary subdivision schemes. 
\begin{thm}\label{H-T-1}
	The H\"{o}lder's regularity of the $4$-point binary subdivision scheme (\ref{constaa1})  is 4.124809715, whereas the H\"{o}lder's regularity of the $5$-point relaxed quaternary subdivision scheme (\ref{constaa2}) is 4.12397897.
\end{thm}
\begin{proof}
	To follow the procedure for H\"{o}lder's regularity, firstly we write the Laurent's polynomial $\mu_{4}(c)$ of the binary subdivision scheme (\ref{constaa1}). That is:
	\begin{eqnarray*}\label{constb1}
		\mu_{4}(c)&=&\frac{1}{384}(c^{-3}+27c^{-2}+121c^{-1}+235c^{0}+235c^{1}+121c^{2}+27c^{3}+c^{4}).
	\end{eqnarray*}
	This implies that
	\begin{equation}\label{constb2}
		\mu_{4}(c)=\left(\frac{1+c}{2}\right)^{5}\nu_{4}(c),
	\end{equation}
	where
	\begin{eqnarray}\label{constb3-1}
		\nu_{4}(c)=\frac{1+22c+c^{2}}{12c^{3}}.
	\end{eqnarray}
	Now from (\ref{a}) we know that $(E_{q})_{ij}=e_{t+i-sj+q}$, so by (\ref{constb3-1}) we have\\
	$e_{0}=\frac{1}{12}$, $e_{1}=\frac{11}{6}$, $e_{2}=\frac{1}{12}$, $p=5$, $t=2$ and $s=2$, thus $q=0, 1, 2$ and then
	$E_{0}$, $E_{1}$ and $E_{2}$ are the matrices with  the elements:
	\begin{eqnarray*}\label{constbb3}
		\left\{\begin{array}{c}
			(E_{0})_{ij}=e_{2+i-2j},\\
			(E_{1})_{ij}=e_{2+i-2j+1},\\
			(E_{2})_{ij}=e_{2+i-2j+2},
		\end{array}\right.
	\end{eqnarray*}
	where $i,j=1,2$.
	
	Hence
	\begin{eqnarray*}\label{constb4}
		E_{0}=\left[
		\begin{array}{cc}
			\frac{11}{6} & 0 \\ \\
			\frac{1}{12} & \frac{1}{12}
		\end{array}
		\right],\,\
		E_{1}=\left[
		\begin{array}{cc}
			\frac{1}{12} & \frac{1}{12}\\ \\
			0 & \frac{11}{6}
		\end{array}
		\right]
		\mbox{and} \,\
		E_{2}=\left[
		\begin{array}{cc}
			0 & \frac{11}{6} \\ \\
			0 & \frac{1}{12}
		\end{array}
		\right].
	\end{eqnarray*}
	Now we calculate the largest eigenvalues of $E_{0}$, $E_{1}$ and $E_{2}$, that are:
	\begin{eqnarray*}\label{constb5}
		\rho(E_{0})=1.8333,\,\
		\rho(E_{1})=1.8333,\,\
		\rho(E_{2})= 0.0833.
	\end{eqnarray*}
	Further, the norm-infinity of these three matrices are:
	\begin{eqnarray*}\label{constb6}
		||E_{0}||_{\infty}=  1.8352,\,\
		||E_{1}||_{\infty}=  1.8352,\,\
		||E_{2}||_{\infty}=  1.8352.
	\end{eqnarray*}
	By using (\ref{2aa}), we have
	\begin{eqnarray*}\label{constb8}
		\max\{1.8333, 1.8333, 0.0833\}\leq\xi\leq \max\{1.8352,  1.8352,  1.8352\}.
	\end{eqnarray*}
	This implies that
		\begin{eqnarray*}
    \xi&=&1.834250000
	\end{eqnarray*}
Thus the H\"{o}lder's regularity of scheme (\ref{constaa1}) is:
	\begin{eqnarray*}\label{constb9}
		r&=&p-log_{s}(\xi)
		=5-log_{2}(1.834250000)
		=4.124809715.
	\end{eqnarray*}
	The Laurent polynomial $U_{5}(c)$ of the quaternary subdivision scheme (\ref{constaa2}) is
	\begin{eqnarray*}\label{constb1}
		U_{5}(c)&=&\hat{B}_{6}c^{-10}+\hat{B}_{1}c^{-9}+\hat{A}_{5}c^{-8}+\hat{A}_{1}c^{-7}+\hat{B}_{5}c^{-6}+\hat{B}_{2}c^{-5}
		+\hat{A}_{4}c^{-4}+\hat{A}_{2}c^{-3}+\hat{B}_{4}c^{-2}+\\&&\hat{B}_{3}c^{-1}+\hat{A}_{3}c^{0}+\hat{A}_{3}c^{1}+
		\hat{B}_{3}c^{2}+\hat{B}_{4}c^{3}+\hat{A}_{2}c^{4}+\hat{A}_{4}c^{5}+\hat{B}_{2}c^{6}+\hat{B}_{5}c^{7}+\hat{A}_{1}c^{8}+
		\hat{A}_{5}c^{9}\\&&+\hat{B}_{1}c^{10}+\hat{B}_{6}c^{11},
	\end{eqnarray*}
	where the values of $\hat{A}_{1}, \ldots, \hat{A}_{5}$, and $\hat{B}_{1}, \ldots, \hat{B}_{6}$ are given in (\ref{a42}).
	This implies that
	\begin{eqnarray}\label{constb3-2}
		U_{5}(c)=\left(\frac{1+c+c^{2}+c^{3}}{4}\right)^{5}V_{5}(c),
	\end{eqnarray}
	where
	\begin{eqnarray*}\label{constb4}
		V_{5}(c)=\frac{1}{144c^{10}}(1+22c+23c^{2}+484c^{3}+23c^{4}+22c^{5}+c^{6}).
	\end{eqnarray*}
	It is given from (\ref{a}) that $(E_{q})_{ij}=e_{t+i-sj+q}$, so by (\ref{constb3-2}), we have\\
	$e_{0}=\frac{1}{144}$, $e_{1}=\frac{11}{72}$, $e_{2}=\frac{23}{144}$, $e_{3}=\frac{121}{36}$, $e_{4}=\frac{23}{144}$, $e_{5}=\frac{11}{72}$, $e_{6}=\frac{1}{144}$, $p=5$, $t=6$ and $s=4$. Thus $q=0, 1,2,\ldots,6$ and then $E_{0}$, $E_{1}$, \ldots, $E_{6}$ are the matrices with the elements:
	\begin{eqnarray}\label{constb5}
		\left\{\begin{array}{c}
			(E_{0})_{ij}=e_{6+i-4j}\\
			(E_{1})_{ij}=e_{6+i-4j+1},\\
			\vdots\\
			(E_{5})_{ij}=e_{6+i-4j+5},\\
			(E_{6})_{ij}=e_{6+i-4j+6},
		\end{array}\right.
	\end{eqnarray}
	where $i,j=1,2,3,4,5,6$.
	
	So, we have
	\begin{eqnarray*}\label{constb4}
		E_{0}=\left[
		\begin{array}{cccccc}
			e_{3} & 0 & 0 & 0 & 0 & 0 \\
			e_{4} & e_{0} & 0 & 0 & 0 & 0 \\
			e_{5} & e_{1} & 0 & 0 & 0 & 0 \\
			e_{6} & e_{2} & 0 & 0 & 0 & 0 \\
			0 & e_{3} & 0 & 0 & 0 & 0 \\
			0 & e_{4} & e_{0} & 0 & 0 & 0
		\end{array}
		\right],\,\
		E_{1}=\left[
		\begin{array}{cccccc}
			e_{4} & e_{0} & 0 & 0 & 0 & 0 \\
			e_{5} & e_{1} & 0 & 0 & 0 & 0 \\
			e_{6} & e_{2} & 0 & 0 & 0 & 0 \\
			0 & e_{3} & 0 & 0 & 0 & 0 \\
			0 & e_{4} & e_{0} & 0 & 0 & 0 \\
			0 & e_{5} & e_{1} & 0 & 0 & 0
		\end{array}
		\right],\,\
		E_{2}=\left[
		\begin{array}{cccccc}
			e_{5} & e_{1} & 0 & 0 & 0 & 0 \\
			e_{6} & e_{2} & 0 & 0 & 0 & 0 \\
			0 & e_{3} & 0 & 0 & 0 & 0 \\
			0 & e_{4} & e_{0} & 0 & 0 & 0 \\
			0 & e_{5} & e_{1} & 0 & 0 & 0 \\
			0 & e_{6} & e_{2} & 0 & 0 & 0
		\end{array}
		\right],
	\end{eqnarray*}
	\begin{eqnarray*}\label{constb4}
		E_{3}=\left[
		\begin{array}{cccccc}
			e_{6} & e_{2} & 0 & 0 & 0 & 0 \\
			0 & e_{3} & 0 & 0 & 0 & 0 \\
			0 & e_{4} & e_{0} & 0 & 0 & 0 \\
			0 & e_{5} & e_{1} & 0 & 0 & 0 \\
			0 & e_{6} & e_{2} & 0 & 0 & 0 \\
			0 & 0 & e_{3} & 0 & 0 & 0
		\end{array}
		\right],\,\
		E_{4}=\left[
		\begin{array}{cccccc}
			0 & e_{3} & 0 & 0 & 0 & 0 \\
			0 & e_{4} & e_{0} & 0 & 0 & 0 \\
			0 & e_{5} & e_{1} & 0 & 0 & 0 \\
			0 & e_{6} & e_{2} & 0 & 0 & 0 \\
			0 & 0 & e_{3} & 0 & 0 & 0 \\
			0 & 0 & e_{4} & e_{0} & 0 & 0
		\end{array}
		\right],\,\
		E_{5}=\left[
		\begin{array}{cccccc}
			0 & e_{4} & e_{0} & 0 & 0 & 0 \\
			0 & e_{5} & e_{1} & 0 & 0 & 0 \\
			0 & e_{6} & e_{2} & 0 & 0 & 0 \\
			0 & 0 & e_{3} & 0 & 0 & 0 \\
			0 & 0 & e_{4} & e_{0} & 0 & 0 \\
			0 & 0 & e_{5} & e_{1} & 0 & 0
		\end{array}
		\right]
	\end{eqnarray*}
	\begin{eqnarray*}
		\mbox{and} \,\
		E_{6}=\left[
		\begin{array}{cccccc}
			0 & e_{5} & e_{1} & 0 & 0 & 0 \\
			0 & e_{6} & e_{2} & 0 & 0 & 0 \\
			0 & 0 & e_{3} & 0 & 0 & 0 \\
			0 & 0 & e_{4} & e_{0} & 0 & 0 \\
			0 & 0 & e_{5} & e_{1} & 0 & 0 \\
			0 & 0 & e_{6} & e_{2} & 0 & 0
		\end{array}
		\right]
	\end{eqnarray*}
	The largest eignevalues of $E_{0}$, $E_{1}$, \ldots, $E_{6}$ are:
\begin{eqnarray*}\label{constb7}
	\rho(E_{0})&=&3.3611, \,\
	\rho(E_{1})=0.1890, \,\
	\rho(E_{2})=0.1890, \,\
	\rho(E_{3})=3.3611, \\
	\rho(E_{4})&=&0.1890, \,\
	\rho(E_{5})=0.1890, \,\
	\rho(E_{6})=3.3611.
\end{eqnarray*}
The norm-infinity of matrices $E_{0}$, $E_{1}$, \ldots, $E_{6}$ are:
\begin{eqnarray*}\label{constb8}
	||E_{0}||_{\infty}&=&3.3745,\,\
	||E_{1}||_{\infty}=3.3756,\,\
	||E_{2}||_{\infty}=3.3756,\,\
	||E_{3}||_{\infty}=3.3745,\\
	||E_{4}||_{\infty}&=&3.3745,\,\
	||E_{5}||_{\infty}=3.3756,\,\
	||E_{6}||_{\infty}=3.3756.
\end{eqnarray*}
	Now from (\ref{2aa}), we have
	\begin{eqnarray*}\label{constb9}
		\max[\rho(E_{0}), \rho(E_{1}),\ldots, \rho(E_{5}), \rho(E_{6}]\leq\xi\leq \max[||E_{0}||_{\infty}, ||E_{1}||_{\infty},\ldots, ||E_{5}||_{\infty}, ||E_{6}||_{\infty}].
	\end{eqnarray*}
	This implies that
	\begin{eqnarray*}\label{constb10}
		\max[3.3611,0.1890,\ldots,0.1890,3.3611]\leq\xi\leq \max[3.3745,3.3756,\ldots,3.3756,3.3756 ].
	\end{eqnarray*}
	Since the largest eigenvalue and the max-norm of the matrices is between ($3.3611$ and $3.3756$) and we choose the mid value $3.368350000$ of the above given values, so the H\"{o}lder regularity is given by
	\begin{eqnarray*}\label{constb11}
		r&=&p-log_{s}(\xi)
		=5-log_{4}(3.368350000)
		=4.123978973.
	\end{eqnarray*}
\end{proof}
The proof of the following theorems follows the proof of Theorem \ref{H-T-1}.
\begin{thm}
	The H\"{o}lder's regularity of the $6$-point binary subdivision scheme (\ref{consta42})  is 6.383689358, while the H\"{o}lder's regularity of the $8$-point relaxed quaternary subdivision scheme (\ref{consta43}) is 6.378805452.
\end{thm}
\begin{thm}
	The H\"{o}lder's regularity of the $8$-point binary subdivision scheme (\ref{constaa})  is 8.575077912 and the H\"{o}lder's regularity of the $11$-point relaxed quaternary subdivision scheme (\ref{const32}) is 8.561638397.
\end{thm}
\begin{thm}
	The H\"{o}lder's regularities of the $10$-point binary subdivision scheme (\ref{const76}) and the $14$-point relaxed quaternary subdivision scheme (\ref{const80})  are 3.768111637 and  4.571743466 respectively.
\end{thm}
\begin{thm}
	The H\"{o}lder's regularity of the $10$-point binary subdivision scheme corresponding to the mask (\ref{const82})  is 10.67905327, although the H\"{o}lder's regularity of the $14$-point relaxed quaternary subdivision scheme (\ref{const83}) is 10.65483615.
\end{thm}
\begin{thm}
	The H\"{o}lder's regularity of the $12$-point binary subdivision scheme (\ref{consta32})  is 12.72368201, whereas the H\"{o}lder's regularity of the $17$-point relaxed quaternary subdivision scheme (\ref{const36}) is 12.69332847.
\end{thm}
\section{Degree of Precision}
Degree of precision of a subdivision scheme is the ability of a subdivision scheme to produce the same polynomial from which the initial data is taken. In other words, degree of precision of a subdivision scheme is $n$ if it produces polynomials of degree $0,1,\ldots,n$ when the initial data is taken from these polynomials respectively, but not produces the polynomial of degree $n+1$ when initial data is taken from that specific polynomial of degree $n+1$.
Whereas the degree of polynomial generation of a subdivision scheme is its ability to produces the polynomial of same degree from which the initial data is chosen.

In this section, we discuss the response of the pair of binary and quaternary schemes on polynomial data. We summerize these responses in Table  \ref{table-com}. In this table, BSS, QSS, DoP and DoG denote binary subdivision scheme, quaternary subdivision scheme, degree of precision and degree of polynomial generation respectively. The table indicates that there is no impact on these two characteristics of the pair of schemes when we move from binary to its corresponding quaternary subdivision scheme. 
\begin{table}
\caption{\label{table-com}\emph{Response of the pairs of binary and quaternary subdivision schemes to the polynomial data.}}
\begin{center}
\begin{tabular}{|c|c|c||c|c|c|}
\hline
\textbf{BSS}   & \textbf{DoP} &   \textbf{DoG} &\textbf{QSS}   & \textbf{DoP} &   \textbf{DoG}\\
\hline
2-points(\ref{consta40})  & 1         & 2  &3-points(\ref{consta41})  & 1         & 2 \\
4-points(\ref{constaa1})  & 1         & 4  & 6-points(\ref{constaa2})  & 1         & 4 \\
6-points(\ref{consta42})  & 1         & 6  & 9-points(\ref{consta43})  & 1         & 6 \\
8-points(\ref{constaa})  & 1         & 8   & 12-points(\ref{const32})  & 1         & 8 \\
10-points(\ref{const76})   & 9         & 10  & 15-points(\ref{const80})   & 9         & 10 \\
10-points(\ref{const82})   & 1         & 10  & 15-points(\ref{const83})   & 1         & 10 \\
12-points(\ref{consta32}) & 1         & 12  & 18-points(\ref{const36}) & 1         & 12 \\
\hline
\end{tabular}
\end{center}
\end{table}
\section{Conclusion}
In this research, we presented a new study about the binary and quaternary subdivision schemes. We proved that every even-point binary subdivision scheme can be used to get a relaxed quaternary subdivision scheme. We presented an intresting link between the masks of theses pairs of schemes, that the mask of the quaternary scheme is just the non-linear combination of the mask of the binary subdivision scheme. The results are applicable on all linear and stationary even-point binary subdivision schemes without any restriction.  Moreover, we validated our general results by different even-point binary subdivision schemes.The graphical inspections and theoratical analysis of the pairs of schemes are also presented. Which shows that the final models of both type of schemes are almost same but quaternary subdivision schemes give us final models in less number of iterations as compare to the parent binary subdivision schemes.


\section*{Conflicts of interest}
The authors declare that there are no conflicts of interest regarding the publication of this paper.

\end{document}